\documentclass{siamltex}
\usepackage{amsmath,amsfonts,amssymb,hyperref,multirow,array}
\usepackage{algorithm}
\usepackage[noend]{algpseudocode}
\usepackage{xcolor}
\hypersetup{colorlinks}
\usepackage{longtable}
\usepackage{rotating}
\usepackage[normal]{caption}
\usepackage{framed}
\usepackage{xcolor}
\hypersetup{colorlinks}
\usepackage{fec,dpr}

\newcommand{\s}[1]{^{\mbox{\protect\tiny {\rm #1}}}}
\newcommand{\Proj}[1]{{\rm Proj}(#1)}
\newcommand{\algname}{fast reduced space algorithm}
\newcommand{\algacro}{FaRSA{}}
\algrenewcommand{\algorithmiccomment}[1]{\hfill[#1]}

\title{A Reduced-Space Algorithm for Minimizing $\ell_1$-Regularized Convex Functions}
\author{
Tianyi Chen\footnotemark[3]~\footnotemark[5]
\and
Frank E. Curtis\footnotemark[2]
\and
Daniel P. Robinson\footnotemark[4]~\footnotemark[5]
}

\date{\today}

\begin{document}

%*******
% Title
%*******
\maketitle

% Footer for front page.
\renewcommand{\thefootnote}{\fnsymbol{footnote}}
\footnotetext[2]{Department of Industrial and Systems Engineering, Lehigh University, Bethlehem, PA, USA. E-mail: \href{mailto:frank.e.curtis@gmail.com}{frank.e.curtis@gmail.com}.  This author was supported by the U.S.~Department of Energy, Office of Science, Applied Mathematics, Early Career Research Program under Award Number DE--SC0010615 and by the U.S.~National Science Foundation, Division of Mathematical Sciences, Computational Mathematics Program under Award Number DMS--1016291.}
\footnotetext[3]{Department of Applied Mathematics and Statistics, Johns Hopkins University, Baltimore, MD, USA. Email: \href{mailto:tchen59@jhu.edu}{tchen59@jhu.edu}.}
\footnotetext[4]{Department of Applied Mathematics and Statistics, Johns Hopkins University, Baltimore, MD, USA. Email: \href{mailto:daniel.p.robinson@gmail.com}{daniel.p.robinson@gmail.com}.}
\footnotetext[5]{This author was supported by IDIES Seed Funding at Johns Hopkins University.}
\renewcommand{\thefootnote}{\arabic{footnote}}

%**********
% Abstract
%**********
\begin{abstract}
We present a new method for minimizing the sum of a differentiable convex function and an $\ell_1$-norm regularizer.  The main features of the new method include: $(i)$ an evolving set of indices corresponding to variables that are predicted to be nonzero at a solution (i.e., the support); $(ii)$ a reduced-space subproblem defined in terms of the predicted support; $(iii)$ conditions that determine how accurately each subproblem must be solved, which allow for Newton, Newton-CG, and coordinate-descent techniques to be employed; $(iv)$ a computationally practical condition that determines when the predicted support should be updated; and $(v)$ a reduced proximal gradient step that ensures sufficient decrease in the objective function when it is decided that variables should be added to the predicted support.  We prove a convergence guarantee for our method and demonstrate its efficiency on a large set of model prediction problems.
\end{abstract}
%---------------------------------------------------------------------------------------------------------------------------
\begin{keywords}
nonlinear optimization, convex optimization, sparse optimization, active-set methods, reduced-space methods, subspace minimization, model prediction
\end{keywords}
%---------------------------------------------------------------------------------------------------------------------------
\begin{AMS}
90C06,
90C25,
90C30,
90C55,
90C90,
49J52,
49M37,
62--07,
62M20,
65K05
\end{AMS}

%*********
% Section
%*********
\section{Introduction}\label{sec.introduction}

In this paper, we propose, analyze, and provide the results of numerical experiments for a new method for solving $\ell_1$-norm regularized convex optimization problems of the form
\bequation\label{prob.x}
  \minimize{x\in\R{n}} \mgap F(x), \mgap \text{where} \mgap F(x) := f(x) + \lambda\onenorm{x},
\eequation
$f:\Re^n \rightarrow \Re$ is a twice continuously differentiable convex function, and $\lambda > 0$ is a weighting parameter.  A necessary and sufficient optimality condition for~\eqref{prob.x} is
\begin{equation}\label{optcond.x}
0 \in \partial F(x) = \Grad f(x) + \lambda \partial\onenorm{x}
\end{equation}
with $\partial F$ and $\partial \onenorm{\cdot}$ denoting the subdifferentials of $F$ and $\onenorm{\cdot}$, respectively.  Our method for solving~\eqref{prob.x} generates a sequence of iterates such that any limit point of the sequence satisfies~\eqref{optcond.x}.  It is applicable when only first-order derivative information is computed, but is most effective when one can at least approximate second-order derivative matrices, e.g., using limited-memory quasi-Newton techniques.

Problems of the form~\eqref{prob.x} routinely arise in statistics, signal processing, and machine learning applications, and are usually associated with data fitting or maximum likelihood estimation.  A popular setting is binary classification using logistic regression (where $f$ is a logistic cost function), although instances of such problems also arise when performing multi-class logistic regression and profit regression.  Instances of~\eqref{prob.x} also surface when using LASSO or elastic-net formulations to perform data analysis and discovery, such as in unsupervised subspace clustering on data drawn from a union of subspaces.

%************
% Subsection
%************
\subsection{Literature review and our key contributions} \label{subsec.literature}

Popular first-order optimization methods for solving~\eqref{prob.x} include ISTA, FISTA, and SpaRSA~\cite{beck2009fast,wright2009sparse}.  Second-order methods have also been proposed, which can roughly be split into the classes of proximal-Newton methods~\cite{byrd2013inexact,hsieh2011sparse,lee2012proximal,scheinberg2013practical,yuan2012improved} and orthant-based methods~\cite{andrew2007scalable,byrd2012family,keskar2015second}. Proximal-Newton methods solve problem~\eqref{prob.x} by minimizing a sequence of subproblems formed as the sum of a quadratic approximation to $f$ and the nonsmooth $\ell_1$-norm regularizer.  For example, the state-of-the-art software LIBLINEAR, which implements newGLMNET~\cite{yuan2012improved}, uses a coordinate descent algorithm to approximately minimize each piecewise quadratic subproblem.  Orthant-based methods, on the other hand, minimize smooth quadratic approximations to~\eqref{prob.x} over a sequence of orthants in $\Re^n$ until a solution is found. Of particular interest is the recently proposed orthant-based method OBA~\cite{keskar2015second} in which every iteration consists of a corrective cycle of orthant predictions and subspace minimization steps. OBA was shown to be slower than LIBLINEAR  when the Hessian matrices were diagonally dominant, but faster otherwise, at least on the collection of test problems considered in~\cite{keskar2015second}.

Since LIBLINEAR and OBA are the most relevant to the algorithm described in this paper, let us discuss their respective advantages and disadvantages in more detail.  The key advantage of LIBLINEAR is its use of a coordinate descent (CD) algorithm to approximately minimize the piecewise quadratic subproblem.  The use of CD means that one should expect excellent performance on problems whose Hessian matrices are strongly diagonally dominant. This expectation was confirmed, as mentioned above, by the OBA paper~\cite{keskar2015second}.  For some problems encountered in model prediction, e.g., when using logistic regression to perform classification, the Hessians are often strongly diagonally dominant, at least after certain data scaling techniques are used.   However, not all prediction problems have such nice diagonal dominance properties, and in some instances the user would prefer to avoid discovering a proper scaling for their data.  In these latter cases, the OBA method is typically superior.

Another potential advantage of the OBA method is its use of an active-set strategy that uses subproblems that are smaller in dimension than the ambient space.  For many $\ell_1$-norm regularized prediction problems, the number of nonzero components in a solution is a small percentage of the ambient dimension, and thus OBA spends most of its time solving small dimensional problems.  This is an advantage, at least when the zero and nonzero structure of the solution is quickly identified.

We have the perspective that both LIBLINEAR and OBA are valuable state-of-the-art algorithms that complement each other.  Our \emph{\algname{}} (\algacro) is designed to capitalize on the advantages of both while avoiding their disadvantages.  The following bulleted points summarize our key contributions.
%\vspace*{-0.2cm}
\begin{romannum}
\item
We present a new active-set line search method that utilizes reduced-space subproblems, approximate solutions of which can be computed efficiently.
\item
Unlike the active-set OBA method, our method does not require the computation of an ISTA step during each iteration to ensure convergence.  We achieve convergence by combining a new projected backtracking line search procedure, an approximate subspace minimization scheme, and a mechanism for determining when the support of the solution estimate should be updated.
\item
Our framework is flexible.  In particular, we introduce a new set of conditions that signal how accurately each subproblem should be solved and allow for various subproblem solvers to be used.  In so doing, our method easily accommodates a Newton-CG subproblem solver as in OBA and a CD solver as in LIBLINEAR. Interestingly, this allows for multiple subproblem solvers to be used in parallel, thus allowing for numerical performance that can be as good as either LIBLINEAR and OBA regardless of whether the problem Hessians are strongly diagonally dominant.
\item
As demonstrated in the numerical experiments described in this paper, the practical performance of our method is state-of-the-art.
\end{romannum}

We end this review by remarking that our proposed algorithm has similarities with the iterative method that one would obtain using the following procedure: $(i)$ at a given iterate $x_k$, construct a quadratic model of $f$ and recast the minimization of this model plus the regularization term $\lambda \|x\|_1$ into a bound-constrained quadratic optimization problem (similarly to the procedure in SpaRSA), $(ii)$ approximately solve this subproblem using the techniques developed in \cite{Dost97,Dost03,DostScho05} (see also \cite{hassanLovesPoetry}), and $(iii)$ translate the resulting solution back into the space of $x$ variables to produce a trial step from $x_k$, call it $d_k$.  Indeed, our initial developments of this work was based on these ideas.  However, the algorithm proposed in this paper involves some deviations and enhancements from this starting point.

%************
% Subsection
%************
\subsection{Notation}

Let $\Iscr\subseteq\{1,2,\dots,n\}$ denote an index set of variables.  For any $v\in\Re^n$, we let $[v]_\Iscr$ denote the subvector of $v$ consisting of elements of $v$ with indices in $\Iscr$.  Similarly, for any symmetric matrix $M\in\Re^{n\times n}$, we let $[M]_{\Iscr,\Iscr}$ denote the submatrix of $M$ consisting of the rows and columns of $M$ that correspond to the index set $\Iscr$. For any vector $v$, we let $\text{sgn}(v)$ denote the vector of the same length as $v$ whose $i$th component is $0$ when $[v]_i = 0$, is $1$ when $[v]_i > 0$, and is $-1$ when $[v]_i < 0$.  For any vector $v$, we let $\|v\|_1$ and $\norm{v}$ denote its $\ell_1$-norm and $\ell_2$-norm, respectively.

%*********
% Section
%*********
\section{Algorithm \algacro}\label{sec.algorithm}

Crucial to our algorithm is the manner in which we handle the zero and nonzero components of a solution estimate.  In order to describe the details of our approach, we first define the index sets
\begin{equation*} %\label{def:setsI}
\begin{aligned}
\Iscr^0(x) :=  \{i:[x]_i = 0\}, \  \
\Iscr^+(x) := \{i: [x]_i > 0\}, \ \ \text{and}  \ \
 \Iscr^-(x) :=  \{i: [x]_i < 0\}.
\end{aligned}
\end{equation*}
We call $\Iscr^0(x)$ the set of \emph{zero variables}, $\Iscr^+(x)$ the set of \emph{positive variables}, $\Iscr^-(x)$ the set of \emph{negative variables}, and the union of $\Iscr^-(x)$ and $\Iscr^+(x)$ the set of \emph{nonzero variables} at $x$.  We use these sets to define measures of optimality corresponding to the zero and nonzero variables at $x$.  Respectively, these measures are as follows:
\begin{equation*}
\begin{aligned}
& [\beta(x)]_i :=
\begin{cases}
[\Grad f(x)]_i + \lambda  & \text{if $i\in\Iscr^0(x)$ and $[\Grad f(x)]_i + \lambda < 0$,} \\
[\Grad f(x)]_i - \lambda  & \text{if $i\in\Iscr^0(x)$ and $[\Grad f(x)]_i - \lambda > 0$,} \\
\qquad 0 & \text{otherwise;}
\end{cases} \\
& [\phi(x)]_i := \\
& \begin{cases}
\qquad 0 & \text{if $i\in\Iscr^0(x)$,} \\
\min\{[\Grad f(x)]_i+\lambda, \max\{[x]_i,[\Grad f(x)]_i-\lambda\} \}  & \text{if $i\in\Iscr^+(x)$ and $[\Grad f(x)]_i+\lambda > 0$,} \\
\max\{[\Grad f(x)]_i - \lambda, \min\{[x]_i,[\Grad f(x)]_i + \lambda\} \}  & \text{if $i\in\Iscr^-(x)$ and $[\Grad f(x)]_i - \lambda < 0$,} \\
[\Grad f(x) + \lambda\cdot \text{sgn}(x)]_i  & \text{otherwise.}
\end{cases}
\end{aligned}
\end{equation*}

The following result shows that the functions $\beta$ and $\phi$ together correspond to a valid optimality measure for problem~\eqref{prob.x}.

\begin{lemma} \label{lem:criticality}
Let $\Sscr$ be an infinite set of positive integers such that $\{x_k\}_{k\in\Sscr}\to x_*$.  Then, $x_*$ is a solution to~\eqref{prob.x} if and only if $\{\beta(x_k)\}_{k\in\Sscr} \to 0$ and $\{\phi(x_k)\}_{k\in\Sscr} \to 0$.  Consequently, $x_*$ is a solution to~\eqref{prob.x} if and only if $\|\beta(x_*)\| = \|\phi(x_*)\| = 0$.
\end{lemma}
\begin{proof}
Suppose $\{\beta(x_k)\}_{k\in\Sscr} \to 0$ and $\{\phi(x_k)\}_{k\in\Sscr} \to 0$.  Then, first, consider any $i$ such that $[x_*]_i > 0$, which means that $[x_k]_i > 0$ for all sufficiently large $k\in\Sscr$.  We now consider two subcases.  If $[\Grad f(x_k)]_i + \lambda \leq 0$ for infinitely many $k\in\Sscr$, then  it follows from the definition of $\phi(x_k)$ and $\{\phi(x_k)\}_{k\in\Sscr} \to 0$ that $[\Grad f(x_*)]_i + \lambda = 0$. On the other hand, if $[\Grad f(x_k)]_i + \lambda > 0$ for infinitely many $k\in\Sscr$, then it follows from the definition of $\phi(x_k)$, $\{\phi(x_k)\}_{k\in\Sscr} \to 0$, and $[x_*]_i > 0$ that $[\Grad f(x_*)]_i + \lambda = 0$.  By combining both cases, we have established that $[\Grad f(x_*)]_i + \lambda = 0$, so that the $i$th component satisfies the optimality conditions~\eqref{optcond.x}.
A similar argument may be used for the case when one considers $i$ such that $[x_*]_i < 0$ to show that $[\Grad f(x_*)]_i - \lambda = 0$.

It remains to consider $i$ such that $[x_*]_i = 0$.  We have four subcases to consider.  First, if infinitely many $k\in\Sscr$ satisfy $[x_k]_i = 0$ and $[\Grad f(x_k)]_i +\lambda < 0$, then it follows from the definition of $\beta(x_k)$ and $\{\beta(x_k)\}_{k\in\Sscr} \to 0$ that $[\Grad f(x_*)]_i + \lambda = 0$; a similar argument shows that if infinitely many $k\in\Sscr$ satisfy $[x_k]_i = 0$ and $[\Grad f(x_k)]_i - \lambda > 0$, then  $[\Grad f(x_*)]_i - \lambda = 0$.  Second, if infinitely many $k\in\Sscr$ satisfy $[x_k]_i = 0$ and $|[\Grad f(x_k)]_i| < \lambda$, then, trivially, $|\Grad f_i(x_*)| \leq \lambda$.  Third, if infinitely many $k\in\Sscr$ satisfy $[x_k]_i > 0$ and $[\Grad f(x_k)]_i + \lambda \leq 0$, then it follows from the definition of $\phi(x_k)$ and $\{\phi(x_k)\}_{k\in\Sscr} \to 0$ that $[\Grad f(x_*)]_i + \lambda = 0$; a similar argument shows that if infinitely many $k\in\Sscr$ satisfy $[x_k]_i < 0$ and $[\Grad f(x_k)]_i - \lambda \geq 0$, then $[\Grad f(x_*)]_i - \lambda = 0$.  Fourth, if infinitely many $k\in\Sscr$ satisfy $[x_k]_i > 0$ and $[\Grad f(x_k)]_i + \lambda > 0$, then it follows from the definition of $\phi(x_k)$ and $\{\phi(x_k)\}_{k\in\Sscr} \to 0$ that $|[\Grad f(x_*)]_i| \leq \lambda$; a similar argument shows that if infinitely many $k\in\Sscr$ satisfy $[x_k]_i < 0$ and $[\Grad f(x_k)]_i - \lambda < 0$, then $|[\Grad f(x_*)]_i| \leq \lambda$.  By combining these subcases, we conclude that $|[\Grad f(x_*)]_i| \leq \lambda$, so the $i$th component satisfies the optimality condition~\eqref{optcond.x}.

To prove the reverse implication, now suppose that $x_*$ is a solution to problem~\eqref{prob.x}. If $[x_*]_i > 0$, then $[\beta(x_k)]_i = 0$ for all sufficiently large $k\in\Sscr$ and $\{[\phi(x_k)]_i\}_{k\in\Sscr} \to 0$ since $[\Grad f(x_*)]_i + \lambda = 0$.  If $[x_*]_i < 0$, then $[\beta(x_k)]_i = 0$ for all sufficiently large $k\in\Sscr$ and $\{[\phi(x_k)]_i\}_{k\in\Sscr} \to 0$  since $[\Grad f(x_*)]_i - \lambda = 0$.  Finally, if $[x_*]_i = 0$, then $|[\Grad f(x_*)]_i| \leq \lambda$, which with the definitions of $\beta(x_k)$ and $\phi(x_k)$ implies that $\{[\beta(x_k)]_i\}_{k\in\Sscr} \to 0$ and $\{[\phi(x_k)]_i\}_{k\in\Sscr} \to 0$.
\end{proof}

We now state our proposed method, \algacro, as Algorithm~\ref{alg:main.x}.  When considering a reduced-space subproblem defined by a chosen index set $\Ical_k$ (see lines~\ref{line:Ik-phi} and \ref{line:Ik-beta}), the algorithm makes use of a quadratic model of the objective of the form (see line~\ref{line:dbar})
\bequationn
  m_k(d) := g_k\T d + \half d\T H_k d.
\eequationn
\algacro\ also makes use of two line search subroutines, stated as Algorithms~\ref{alg:ls-phi} and \ref{alg:ls-beta}, the former of which employs the following projection operator dependent on $x_k$:
$$
[\Proj{y \,; x_k}]_i :=
\begin{cases}
  \max\{0,[y]_i\} & \text{if $\Iscr^+(x_k)$,} \\
  \min\{0,[y]_i\} &  \text{if $\Iscr^-(x_k)$,}  \\
  0                    &  \text{if $\Iscr^0(x_k)$.}
\end{cases}
$$

\balgorithm[ht]
  \caption{\algacro\ for solving problem~\eqref{prob.x}.}
  \label{alg:main.x}
  \balgorithmic[1]
  \State \textbf{Input:} $x_0$
  \State \textbf{Constants:} $\{\eta_\phi,\eta_\beta,\xi\}\subset (0,1]$,  $\eta\in(0,1/2]$, and $\{\gamma,\epsilon\}\subset(0,\infty)$
  \For{$k = 0,1,2,\dots$}
%    \State Compute $\Iscr^0(x_k) = \{i:[x_k]_i = 0\}$, $\Iscr^+(x_k) = \{i:[x_k]_i > 0\}$, $\Iscr^-(x_k) = \{i:[x_k]_i < 0\}$.
%    \State Compute $\beta(x_k)$ and $\phi(x_k)$.
    \If{$\max\{ \norm{\beta(x_k)},  \norm{\phi(x_k)} \} \leq \epsilon$} \label{line:check-termination}
       \State \textbf{Return} the (approximate) solution $x_k$ of problem~\eqref{prob.x}. \label{line:return}
    \EndIf
     \If{$\norm{\beta(x_k)} \leq \gamma\norm{\phi(x_k)}$} \label{line:main.if} \Comment{$k\in\Sscr_\phi$}
         \State Choose any $\Iscr_k\subseteq\{i:[\phi(x_k)]_i \neq 0\}$ such that $\norm{[\phi(x_k)]_{\Iscr_k}} \geq \eta_\phi\norm{\phi(x_k)}$.  \label{line:Ik-phi}
         \State Set $H_k \gets  [\Hess F(x_k)]_{\Iscr_k\Iscr_k}$ and $g_k \gets [\Grad F(x_k)]_{\Iscr_k}$. \label{line:gH}
         \State Compute the reference direction \label{line:dR}
         \bequationn
           d_k^{R} \gets -\alpha_k g_k, \mgap \text{where} \mgap \alpha_k \gets \norm{g_k}^2/(g_k\T H_k g_k).
         \eequationn
         \State Compute any direction $\dbar_k$ that satisfies the inequalities \label{line:dbar}
         $$
         g_k\T \dbar_k \leq g_k\T d_k^R  \ \  \text{and} \ \  m_k(\dbar_k) \leq m_k(0).
        $$
        \State Set $[d_k]_{\Iscr_k} \gets \dbar_k$ and $[d_k]_i \gets 0$ for $i \notin \Iscr_k$. \label{line:d-phi}
        \State Use Algorithm~\ref{alg:ls-phi} to compute $x_{k+1} \gets$ {\sc linesearch\_$\phi$}$(x_k,d_k,\Iscr_k,\eta,\xi)$. \label{line:ls-phi}
     \Else \Comment{$k\in\Sscr_\beta$}
        \State Choose any $\Iscr_k\subseteq\{i:[\beta(x_k)]_i \neq 0\}$ such that $\norm{[\beta(x_k)]_{\Iscr_k}} \geq \eta_\beta \norm{[\beta(x_k)}$.   \label{line:Ik-beta}
        \State Set $[d_k]_{\Iscr_k} \gets -[\beta(x_k)]_{\Iscr_k}$ and $[d_k]_i \gets 0$ for $i \notin \Iscr_k$. \label{line:d-beta}
        \State Use Algorithm~\ref{alg:ls-beta} to compute $x_{k+1} \gets$ {\sc linesearch\_$\beta$}$(x_k,d_k,\eta,\xi)$. \label{line:ls-beta}
     \EndIf
  \EndFor
  \ealgorithmic
\ealgorithm

\balgorithm[ht]
  \caption{A line search procedure for computing $x_{k+1}$ when $k\in\Sscr_\phi$.}
  \label{alg:ls-phi}
  \balgorithmic[1]
  \Procedure{$x_{k+1} =$ linesearch\_$\phi$}{$x_k,d_k,\Iscr_k,\eta,\xi$}
     \State Set $j \gets 0$ and $y_0 \gets \Proj{x_k+d_k \, ; x_k}$.
     \While{$\text{sgn}(y_j) \neq \text{sgn}(x_k)$} \label{line:ls-phi-while}
          \If{$F(y_j) \leq F(x_k)$}
              \State \textbf{return} $x_{k+1} \gets y_j$. \label{line:ls-phi-return1} \Comment{$k\in\Sscr_\phi\s{ADD}$}
         \EndIf
         \State Set $j \gets j + 1$ and then $y_j \gets \Proj{x_k + \xi^j d_k \, ; x_k}$.
     \EndWhile
     \If{$j \neq 0$} \label{line:check-j}
        \State Set $\alpha_B \gets \text{argsup} \sgap \{ \alpha >  0: \text{sgn}(x_k + \alpha d_k) = \text{sgn}(x_k)\}$. \label{line:alphaB}
        \State Set $y_j \gets x_k + \alpha_B d_k$.  \label{line:ls-phi-y-alphaB}
	 \If{$F(y_j) \leq F(x_k) + \eta \alpha_B [\Grad F(x_k)]_{\Iscr_k}^T [d_k]_{\Iscr_k}$} \label{line:check-B}
              \State \textbf{return} $x_{k+1} \gets y_j$. \label{line:ls-phi-return2}\Comment{$k\in\Sscr_\phi\s{ADD}$}
          \EndIf
     \EndIf
     \Loop{} \label{line:ls-phi-loop}
         \If{$F(y_j) \leq F(x_k) + \eta \xi^j [\Grad F(x_k)]_{\Iscr_k}^T [d_k]_{\Iscr_k}$} \label{line:ls-phi-SD}
            \State \textbf{return} $x_{k+1} \gets y_j$. \label{line:ls-phi-return3} \Comment{$k\in\Sscr_\phi\s{SD}$}
         \EndIf
         \State Set $j \gets j + 1$ and then $y_j \gets x_k + \xi^j d_k$.
     \EndLoop
%     \State \textbf{return} $x_{k+1} \gets y_{j}$.
  \EndProcedure
  \ealgorithmic
\ealgorithm

\balgorithm[ht]
  \caption{A line search procedure for computing $x_{k+1}$ when $k\in\Sscr_\beta$.}
  \label{alg:ls-beta}
  \balgorithmic[1]
  \Procedure{$x_{k+1} =$ linesearch\_$\beta$}{$x_k,d_k,\eta,\xi$}
     \State Set $j \gets 0$ and $y_0 \gets x_k + d_k$.
     \While{$F\big(y_j) > F(x_k) - \eta \xi^{j} \norm{d_k}^2$} \label{line:while-beta}
        \State Set $j \gets j + 1$ and then $y_j \gets x_k + \xi^j d_k$.
        \EndWhile
          \State \textbf{return} $x_{k+1} \gets y_{j}$. \label{line:update2}
  \EndProcedure
  \ealgorithmic
\ealgorithm

\algacro\ computes a sequence of iterates $\{x_k\}$.  During each iteration, the sets $\Iscr^0(x_k)$, $\Iscr^+(x_k)$, and $\Iscr^-(x_k)$ are identified, which are used to define $\beta(x_k)$ and $\phi(x_k)$.  We can see in line~\ref{line:check-termination}  of Algorithm~\ref{alg:main.x} that when both $\norm{\beta(x_k)}$ and $\norm{\phi(x_k)}$ are less than a prescribed tolerance $\epsilon > 0$, it returns $x_k$ as an approximate solution to~\eqref{prob.x}; this is justified by Lemma~\ref{lem:criticality}.  Otherwise, it proceeds in one of two ways depending on the relative sizes of $\norm{\beta(x_k)}$ and $\norm{\phi(x_k)}$.  We describe these two cases next.

\begin{romannum}
\item The relationship $\norm{\beta(x_k)} \leq \gamma\norm{\phi(x_k)}$ indicates that significant progress toward optimality can still be achieved by reducing $F$ over the current set of nonzero variables at $x_k$; lines~\ref{line:Ik-phi}--\ref{line:ls-phi} are designed for this purpose. In line~\ref{line:Ik-phi}, a subset $\Iscr_k$ of variables are chosen such that the norm of $\phi(x_k)$ over that subset of variables is at least proportional to the norm of $\phi(x_k)$ over the full set of variables. This allows control over the size of the subproblem, which may be as small as one-dimensional.  Note that for $i\in\Iscr_k$, it must hold that $[\phi(x_k)]_i \neq 0$, which in turn means that $i \notin\Iscr^0(x_k)$, i.e., the $i$th variable is nonzero.  This means that the reduced space subproblem to minimize $m_k(d)$ over $d \in \Re^{|\Iscr_k|}$ is aimed at minimizing $F$ over the variables in~$\Iscr_k$. Our analysis does not require an exact minimizer of $m_k$.  Rather, we allow for the computation of any direction $\dbar_k$ that satisfies the conditions in line~\ref{line:dbar}, namely $g_k\T \dbar_k \leq g_k\T d_k^{R}$ and $m_k(\dbar_k) \leq m_k(0)$, where the reference direction $d_k^R$ is computed in line~\ref{line:dR} by minimizing $m_k$ along the steepest decent direction.  The first condition imposes how much descent is required by the search direction $\dbar_k$, while the second condition ensures that the model is reduced at least as much as a zero step.  It will be shown (see Lemma~\ref{lem:dbar-bound}) that the second condition ensures that $\dbar_k$ is bounded by a multiple of $\norm{g_k}$.  Such conditions are satisfied by a Newton step, by any Newton-CG iterate, and asymptotically by CD iterates. Once $\dbar_k$ is obtained, the search direction $d_k$ in the full space is obtained by filling its elements that correspond to the index set $\Iscr_k$ with the elements from $\dbar_k$, and setting the complementary set of variables to zero (see line~\ref{line:d-phi}).  With the search direction $d_k$ computed, we call Algorithm~\ref{alg:ls-phi} in line~\ref{line:ls-phi}, which performs a (non-standard) backtracking projected line search.  This line search procedure makes use of the projection operator $\proj(\cdot \, ;x_k)$.  This operator projects vectors onto the orthant inhabited by $x_k$, a feature shared by OBA. The while-loop that starts in line~\ref{line:ls-phi-while} of Algorithm~\ref{alg:ls-phi} checks whether the trial point $y_j$ decreases the objective function $F$ relative to its value at $x_k$ when $\text{sgn}(y_j) \neq \text{sgn}(x_k)$.  If the line search terminates in this while-loop, then this implies that at least one component of $x_k$ that was nonzero has become zero for $x_{k+1} = y_j$.  Since the dimension of the reduced space will therefore be reduced during the next iteration (provided line~\ref{line:main.if} of Algorithm~\ref{alg:main.x} tests true), the procedure only requires $F(x_{k+1}) \leq F(x_k)$ instead of a more traditional sufficient decrease condition, e.g., one based on the Armijo condition. If line~\ref{line:check-j} of Algorithm~\ref{alg:ls-phi} is reached, then the current trial iterate $y_j$ satisfies $\text{sgn}(y_j) = \text{sgn}(x_k)$, i.e., the trial iterate has entered the same orthant as that inhabited by $x_k$.  Once this has occurred, the method could then perform a standard backtracking Armijo line search as stipulated in the loop starting at line~\ref{line:ls-phi-loop}.  For the purpose of guaranteeing convergence, however, the method first checks whether the largest step along $d_k$ that stays in the same orthant as $x_k$ (see lines~\ref{line:alphaB} and~\ref{line:ls-phi-y-alphaB}) satisfies the Armijo sufficient decrease condition (see line~\ref{line:check-B}).  (This aspect makes our procedure different from a standard backtracking scheme.)  If Algorithm~\ref{alg:ls-phi} terminates in line~\ref{line:ls-phi-return1} or \ref{line:ls-phi-return2}, then at least one nonzero variable at $x_k$ will have become zero at $x_{k+1}$, which we indicate by saying $k\in\Sscr_\phi\s{ADD} \subseteq \Sscr_\phi$. Otherwise, if Algorithm~\ref{alg:ls-phi} terminates in line~\ref{line:ls-phi-return3}, then $x_{k+1}$ and $x_k$ are housed in the same orthant and sufficient decrease in $F$ was achieved (i.e., the Armijo condition in line~\ref{line:ls-phi-SD} was satisfied). Since sufficient decrease has been achieved in this case, we say that $k\in\Sscr_\phi\s{SD} \subseteq \Sscr_\phi$.
\item When $\norm{\beta(x_k)} > \gamma\norm{\phi(x_k)}$, progress toward optimality is best achieved by freeing at least one variable that is currently set to zero; lines~\ref{line:Ik-beta}--\ref{line:ls-beta} are designed for this purpose. Since $\norm{\beta(x_k)}$ is relatively large, in line~\ref{line:Ik-beta} of Algorithm~\ref{alg:main.x} a subset $\Iscr_k$ of variables is chosen such that the norm of $\beta(x_k)$ over that subset of variables is at least proportional to the norm of $\beta(x_k)$ over the full set of variables.  Similar to the previous case, this allows control over the size of the subproblem, which in the extreme case may be one-dimensional.  If $i\in\Iscr_k$, then $[\beta(x_k)]_i \neq 0$, which in turn means that $i \in\Iscr^0(x_k)$, i.e., the $i$th variable has the value zero. The components of $\beta(x_k)$ that correspond to $\Iscr_k$ are then used to define the search direction $d_k$ in line~\ref{line:d-beta}.  With the search direction $d_k$ computed, Algorithm~\ref{alg:ls-beta} is called in line~\ref{line:ls-beta}, which performs a standard backtracking Armijo line search to obtain $x_{k+1}$.  If a unit step length is taken, i.e., if $x_{k+1}  = x_k + d_k$, then $x_{k+1}$ can be interpreted as the iterate that would be obtained by taking a \emph{reduced ISTA step} in the space of variables indexed by $\Iscr_k$. (For additional details, see Lemma~\ref{lem:equi-ista} in the appendix.)
\end{romannum}

%*********
% Section
%*********
\section{Convergence Analysis} \label{sec:convergence}

Our analysis uses the following assumption that is assumed to hold throughout this section.

\begin{assumption}\label{ass:main}
The function $f : \R{n} \to \R{}$ is convex, twice continuously differentiable, and bounded below on the level set $\Lscr := \{x\in\Re^n : F(x) \leq F(x_0)\}$.  The gradient function $\nabla f : \R{n} \to \R{n}$ is Lipschitz continuous on $\Lscr$ with Lipschitz constant~$L$.  The Hessian function $\nabla^2 f : \R{n} \to \R{n\times n}$ is uniformly positive definite and bounded on $\Lscr$, i.e., there exist positive constants $\theta_{\min}$ and $\theta_{\max}$ such that
\bequationn
  \theta_{\min} \|v\|^2 \leq v^TH(x)v \leq \theta_{\max} \|v\|^2 \mgap \text{for all} \mgap \{x,v\} \subset \Re^n.
\eequationn
\end{assumption}

%\begin{assumption}\label{ass:pd}
% The Hessian of $f$ is uniformly positive definite and bounded on the level set
%$\Lscr := \{x\in\Re^n: F(x) \leq F(x_0)\}$. That is, there exist constants $\theta_{\min} > 0$ and $\theta_{\max} > 0$ such that the smallest eigenvalue of $\Hess f(x)$ is greater than $\theta_{\min}$ and $\norm{\Hess f(x)} \leq \theta_{\max}$ for all $x\in\Lscr$.
%\end{assumption}

%\begin{assumption} \label{ass:f-bounded}
%The function $f$ is bounded below on the level set $\Lscr$ defined in Assumption~\ref{ass:pd}.
%\end{assumption}

Our analysis uses the index sets (already shown in Algorithms~\ref{alg:main.x}--\ref{alg:ls-phi})
\begin{equation*} %\label{def:setsS}
\begin{aligned}
\Sscr_\phi &:= \{k: \text{lines~\ref{line:Ik-phi}--\ref{line:ls-phi} in Algorithm~\ref{alg:main.x} are performed during iteration $k$} \}, \\
\Sscr_\phi\s{ADD} &:= \{k\in\Sscr_\phi: \text{sgn}(x_{k+1}) \neq \text{sgn}(x_k)\}, \\
\Sscr_\phi\s{SD} &:= \{k\in\Sscr_\phi: \text{sgn}(x_{k+1}) = \text{sgn}(x_k)\}, \ \ \text{and} \\
\Sscr_\beta &:= \{k: \text{lines~\ref{line:Ik-beta}--\ref{line:ls-beta} in Algorithm~\ref{alg:main.x} are performed during iteration $k$} \}.
\end{aligned}
\end{equation*}

We start with a lemma that establishes an important identity for iterations in $\Sscr_\beta$.

\begin{lemma} \label{lem:facts}
If $k\in\Sscr_\beta$, then $(\Iscr_k,d_k)$ in lines~\ref{line:Ik-beta} and \ref{line:d-beta} of Algorithm~\ref{alg:main.x} yield
\begin{equation} \label{eq:identity}
[d_k]_{\Iscr_k} = -[\Grad f(x_k) + \lambda \cdot {\rm sgn}(x_k + \xi^j d_k)]_{\Iscr_k} \mgap \text{for any integer $j$.}
\end{equation}
Consequently, the right-hand side of \eqref{eq:identity} has the same value for any integer $j$.
\end{lemma}
\begin{proof}
We prove that \eqref{eq:identity} holds for an arbitrary element of $\Iscr_k$.  To this end,
let $j$ be any integer and $i\in \Iscr_k\subseteq\{\ell:[\beta(x_k)]_\ell \neq 0]\}$, where $\Iscr_k$ is defined in line~\ref{line:Ik-beta}.  It follows from the definition of $\Iscr_k$, the definition of $d_k$ in line~\ref{line:d-beta}, and $i\in\Iscr_k$ that
\begin{equation} \label{eq:d-beta}
[d_k]_i=
\begin{cases}
-\big([\nabla f(x_k)]_i +\lambda\big) &   \text{if $[\Grad f(x_k)]_i + \lambda < 0$,} \\
-\big([\nabla f(x_k)]_i -\lambda\big) &    \text{if $[\Grad f(x_k)]_i - \lambda > 0$,}
\end{cases}
\end{equation}
so that $[d_k]_i \neq 0$. Also, since $[x_k]_i = 0$ for $i\in\Iscr_k$, we know that $[x_k + \xi^j d_k]_i \neq 0$.  Thus, we need only consider the following two cases. \\
\noindent\textbf{Case 1:} Suppose $[x_k + \xi^j d_k]_i>0$.  In this case, the right-hand-side of~\eqref{eq:identity} is equal to $-([\Grad f(x_k)]_i + \lambda)$.  As for the left-hand-side, since $[x_k]_i = 0$ and $[x_k + \xi^j d_k]_i>0$, we have $ 0 < [x_k + \xi^j d_k]_i = \xi^j[d_k]_i$, which combined with $\xi^j > 0$ means that $[d_k]_i > 0$.  This fact and~\eqref{eq:d-beta} gives $[d_k]_i = -([\Grad f(x_k)]_i + \lambda)$, so~\eqref{eq:identity} holds. \\
\noindent\textbf{Case 2:} Suppose $[x_k + \xi^j d_k]_i<0$.
In this case, the right-hand-side of~\eqref{eq:identity} is equal to
$-([\Grad f(x_k)]_i - \lambda)$.  As for the left-hand-side, since $[x_k]_i = 0$ and $[x_k + \xi^j d_k]_i < 0$ we have $ 0 > [x_k + \xi^j d_k]_i = \xi^j[d_k]_i$, which when combined with $\xi^j > 0$ means that $[d_k]_i < 0$. This fact and~\eqref{eq:d-beta} gives $[d_k]_i = -([\Grad f(x_k)]_i - \lambda)$, so~\eqref{eq:identity} holds.
\end{proof}

We can now establish a bound for a decrease in the objective when $k\in\Sscr_\beta$.

\begin{lemma} \label{lem:beta-decrease}
If $k\in\Sscr_\beta$, then $d_k$ in line~\ref{line:d-beta} of Algorithm~\ref{alg:main.x} yields
$$
F(x_k + \xi^j d_k) \leq F(x_k) - \frac{\xi^j}{2} \norm{d_k}^2
\ \ \text{for any integer $j$ with $0\leq \xi^j \leq \frac{1}{L}$}.
$$
\end{lemma}
\begin{proof}
Let $j$ be any integer with $0\leq \xi^j \leq \frac{1}{L}$ and let $y_j := x_k + \xi^j d_k$.  By Lipschitz continuity of the gradient function $\Grad f$, we have
\begin{align}
  f(y_j)
    &\leq f(x_k) + \xi^j \Grad f(x_k)\T d_k+\frac{L}{2}\|\xi^j d_k\|^2 \nonumber \\
    &\leq f(x_k) + \xi^j \Grad f(x_k)\T d_k + \frac{\xi^j}{2} \|d_k\|^2. \label{eq:beta.1}
\end{align}
It then follows from~\eqref{eq:beta.1}, convexity of both $f$ and $\lambda\onenorm{\cdot}$, the fact that $\text{sgn}(y_j) \in \partial \onenorm{y_j}$, the definition of $d_k$ (in particular that $[d_k]_i = 0$ for $i\notin\Iscr_k$), and Lemma~\ref{lem:facts} that the following holds for all $z\in\Re^n$:
\begin{align}
      &\ F(y_j) \nonumber \\
     =&\ f(y_j)+ \lambda \onenorm{y_j}\nonumber \\
  \leq&\ f(x_k)+\xi^j \Grad f(x_k)\T d_k + \frac{\xi^j}{2} \|d_k\|^2 + \lambda \onenorm{y_j} \nonumber \\
  \leq&\ f(z)+\Grad f(x_k)\T(x_k-z)+\xi^j\Grad f(x_k)\T d_k+ \frac{\xi^j}{2} \|d_k\|^2+\lambda\onenorm{z}+\lambda\cdot{\rm sgn}(y_j)\T(y_j-z)\nonumber \\
  \leq&\ F(z)+[\nabla f(x_k)+\lambda\cdot {\rm sgn}(y_j)]^T(x_k-z) + \xi^j[\nabla f(x_k)+\lambda \cdot{\rm sgn}(y_j)]^T d_k + \frac{\xi^j}{2} \|d_k\|^2 \nonumber \\
     =&\ F(z)+[\nabla f(x_k)+\lambda\cdot {\rm sgn}(y_j)]^T(x_k-z)-\xi^j\| d_k\|^2 + \frac{\xi^j}{2} \| d_k\|^2\nonumber \\
     =&\ F(z)+[\nabla f(x_k)+\lambda\cdot{\rm sgn}(y_j)]^T(x_k-z) - \frac{\xi^j}{2} \| d_k\|^2. \label{eq:beta.2}
\end{align}
The desired result follows by considering $z = x_k$ in~\eqref{eq:beta.2}.
\end{proof}

We now show that Algorithm~\ref{alg:ls-beta} called in line~\ref{line:ls-beta} of Algorithm~\ref{alg:main.x} is well defined, and that it returns $x_{k+1}$ yielding sufficient decrease in the objective function.

\begin{lemma} \label{lem:suf-dec-beta}
If $k\in\Sscr_\beta$,  then $x_{k+1}$ satisfies
\begin{equation} \label{F-dec-beta}
F(x_{k+1})
\leq F(x_k) - \kappa_\beta \max\{\norm{\beta(x_k)}^2,\gamma^2\norm{\phi(x_k)}^2\},
\end{equation}
where $\kappa_\beta :=  \eta\eta_\beta \min\{1,\xi/L\}$.
\end{lemma}
\begin{proof}
Let $j$ be any integer with $0\leq \xi^j \leq \frac{1}{L}$ and let $y_j := x_k + \xi^j d_k$.  It follows from Lemma~\ref{lem:beta-decrease} and the fact that $\eta\in(0,1/2]$ in Algorithm~\ref{alg:main.x} that
$$
F\big(y_{j} \big) \leq F(x_k) - \frac{\xi^{j}}{2} \norm{d_k}^2
\leq F(x_k) - \eta \xi^{j} \norm{d_k}^2,
$$
It follows from this inequality that Algorithm~\ref{alg:ls-beta} will return the vector $x_{k+1} = x_k + \xi^{j_*} d_k$ with $\xi^{j_*} \geq \min\{1,\xi/L\}$ when called in line~\ref{line:ls-beta} of Algorithm~\ref{alg:main.x}. Using this bound, line~\ref{line:while-beta} of Algorithm~\ref{alg:ls-beta}, and lines~\ref{line:d-beta} and~\ref{line:Ik-beta} of Algorithm~\ref{alg:main.x}, we have
\begin{align*}
F(x_{k+1}) &\leq F(x_k) - \eta \xi^{j_*} \norm{d_k}^2 \leq F(x_k) - \eta\min\{1,\xi/L\} \norm{d_k}^2 \\
&=  F(x_k) - \eta \min\{1,\xi/L\} \norm{[\beta(x_k)]_{\Iscr_k}}^2 \leq F(x_k) - \eta \eta_\beta\min\{1,\xi/L\} \norm{\beta(x_k)}^2.
\end{align*}
The inequality~\eqref{F-dec-beta} follows from the definition of $\kappa_\beta$, the previous inequality, and the fact that the inequality in line~\ref{line:main.if} of Algorithm~\ref{alg:main.x} must not hold since line~\ref{line:ls-beta} is assumed to be reached.
\end{proof}

We now show that the index set $\Sscr_\beta$ must be finite.

\begin{lemma} \label{lem:Sbeta-finite}
The index set $\Sscr_\beta$ must be finite, i.e., $|\Sscr_\beta| < \infty$.
\end{lemma}
\begin{proof}
To derive a contradiction, suppose that $|\Sscr_\beta| = \infty$, which also means that Algorithm~\ref{alg:main.x} does not terminate finitely. Since  Algorithm~\ref{alg:main.x} does not terminate finitely, we know from line~\ref{line:check-termination} of Algorithm~\ref{alg:main.x} that $\max\{ \norm{\beta(x_k)},  \norm{\phi(x_k)} \} > \epsilon$ for all $k\geq 0$.  Combining this inequality with Lemma~\ref{lem:suf-dec-beta} and the fact that $F(x_{k+1}) \leq F(x_k)$ for all $k\notin \Sscr_\beta$ (as a result of Algorithm~\ref{alg:ls-phi} called in line~\ref{line:ls-phi} of Algorithm~\ref{alg:main.x}), we may conclude for any nonnegative integer $\ell$ and $\kappa_\beta > 0$ defined in Lemma~\ref{lem:suf-dec-beta} that
\begin{align*}
F(x_0) - F(x_{\ell+1})
&= \sum_{k=0}^\ell \big[F(x_k) - F(x_{k+1})\big] \\
&\geq  \sum_{\substack{k\in\Sscr_\beta,k\leq\ell}} \big[F(x_k) - F(x_{k+1})\big]  \\
&\geq  \sum_{\substack{k\in\Sscr_\beta,k\leq\ell}} \kappa_\beta \max\{\norm{\beta(x_k)}^2,\gamma^2\norm{\phi(x_k)}^2\} \\
&\geq  \sum_{\substack{k\in\Sscr_\beta,k\leq\ell}} \kappa_\beta \min\{1,\gamma^2\} \epsilon^2.
\end{align*}
Rearranging the previous inequality shows that
\begin{align*}
\lim_{l\to\infty} F(x_{\ell+1})
&\leq \lim_{\ell \to\infty} \left[F(x_0) - \sum_{\substack{k\in\Sscr_\beta,k\leq\ell}} \kappa_\beta\min\{1,\gamma^2\}\epsilon^2\right] \\
&= F(x_0) -  \sum_{\substack{k\in\Sscr_\beta}} \kappa_\beta\min\{1,\gamma^2\} \epsilon^2
= -\infty,
\end{align*}
which contradicts Assumption~\ref{ass:main}.  Thus, we conclude that $|\Sscr_\beta| < \infty$.
\end{proof}

To prove that Algorithm~\ref{alg:main.x} terminates finitely with an approximate solution to problem~\eqref{prob.x}, all that remains is to prove that the set $\Sscr_\phi$ is finite.  To establish that $\Sscr_\phi \equiv \Sscr_\phi\s{ADD} \cup \Sscr_\phi\s{SD}$ is finite, we proceed by showing individually that both $\Sscr_\phi\s{ADD}$ and $\Sscr_\phi\s{SD}$ are finite.  We begin with the set $\Sscr_\phi\s{ADD}$.

\begin{lemma} \label{lem:SphiAdd-finite}
The set $\Sscr_\phi\s{ADD}$ is finite, i.e., $|\Sscr_\phi\s{ADD}| < \infty$.
\end{lemma}
\begin{proof}
To derive a contradiction, suppose that $|\Sscr_\phi\s{ADD}| = \infty$, which in particular means that Algorithm~\ref{alg:main.x} does not terminate finitely.
Since Lemma~\ref{lem:Sbeta-finite} shows that $\Sscr_\beta$ is finite, we may also conclude that there exists an iteration $k_1$ such that $k\in\Sscr_\phi = \Sscr_\phi\s{ADD} \cup \Sscr_\phi\s{SD}$ for all $k \geq k_1$.

We proceed by making two observations.  First, if the $i$th component of $x_k$ becomes zero for some iteration $k \geq k_1$, it will remain zero for the remainder of the iterations.  This can be seen by using lines~\ref{line:d-phi} and~\ref{line:Ik-phi} of Algorithm~\ref{alg:main.x} and the definition of $\phi(x_k)$ to deduce that if $[d_k]_i \neq 0$, then $i\in\Iscr_k \subseteq \{\ell:[\phi(x_k)]_\ell \neq 0\} \subseteq \Iscr^+(x_k) \cup \Iscr^-(x_k)$ for all $k \geq k_1$; equivalently, if $i\in\Iscr^0(x_k)$, then $[d_k]_i = 0$.  The second observation is that at least one nonzero component of $x_k$ becomes zero at $x_{k+1}$ for each $k\in\Sscr_\phi\s{ADD}$. This can be seen by construction of Algorithm~\ref{alg:ls-phi} when it is called in line~\ref{line:ls-phi} of Algorithm~\ref{alg:main.x}.  Together, these observations contradict $|\Sscr_\phi\s{ADD}| = \infty$, since at most $n$ variables may become zero. Thus, we must conclude that $|\Sscr_\phi\s{ADD}| < \infty$.
\end{proof}

To establish that $\Sscr_\phi\s{SD}$ is finite, we require the following two lemmas.  The first lemma gives a bound on the size of $\dbar_k$ that holds whenever $k\in\Sscr_\phi$.

\begin{lemma} \label{lem:dbar-bound}
If $k\in\Sscr_\phi$, then $\norm{\dbar_k} \leq (2/\theta_{\min}) \norm{g_k}$ where $\theta_{\min} > 0$ is defined in Assumption~\ref{ass:main}.
\end{lemma}
\begin{proof}
Let $k\in\Sscr_\phi$ so that $\dbar_k$ is computed in line~\ref{line:dbar} of Algorithm~\ref{alg:main.x}, and let $d_k\s{N}$ be the Newton step satisfying $H_k d_k\s{N} = -g_k$ with $H_k$ and $g_k$ defined in line~\ref{line:gH} of Algorithm~\ref{alg:main.x}.  It follows  that
\begin{equation} \label{eq:Newton-bd}
\norm{d_k\s{N}} \leq \norm{H_k\inv}\norm{g_k}.
\end{equation}
Let us also define the quadratic function $\mbar_k(d) := m_k(d_k\s{N}+d)$ and the associated level set $\Lscr_k := \{d: \mbar_k(d) \leq 0 \}$. We then see that
\begin{equation} \label{eq:in-Lscr}
   (\dbar_k  - d_k\s{N}) \in \Lscr_k
\end{equation}
since $\mbar_k(\dbar_k - d_k\s{N}) = m_k(\dbar_k) \leq m_k(0) = 0$, where we have used the condition $m_k(\dbar_k) \leq m_k(0)$ that is required to hold in line~\ref{line:dbar} of Algorithm~\ref{alg:main.x}.

We are now interested in finding a point in $\Lscr_k$ with largest norm.  To characterize such a point, we consider the optimization problem
\begin{equation} \label{prob.d}
\maximize{d\in\Re^n}  \sgap \half \norm{d}^2 \sgap
\subject\sgap d \in\Lscr_k.
\end{equation}
It is not difficult to prove that a global maximizer of problem~\eqref{prob.d} is $d_* := \alpha_* v$ with $\alpha_*^2 := (-g_k\T d_k\s{N})/\theta$, where $(v,\theta)$ with $\norm{v} = 1$ is an eigenpair corresponding to the left-most eigenvalue $\theta \geq \theta_{\min}$ of $H_k$.  Thus, it follows that $\norm{d}^2 \leq \norm{d_*}^2$ for all $d \in\Lscr_k$.  Combining this with~\eqref{eq:in-Lscr}, the definition of $d_*$, and~\eqref{eq:Newton-bd}  shows that
\begin{align*}
  \norm{\dbar_k - d_k\s{N}}^2 \leq \norm{d_*}^2 = \alpha_*^2\norm{v}^2 = \frac{-g_k\T d_k\s{N}}{\theta}
  \leq \frac{\norm{g_k}\norm{d_k\s{N}}}{\theta} \leq \frac{\norm{H_k\inv}\norm{g_k}^2}{\theta}
  = \left(\frac{\norm{g_k}}{\theta}\right)^2.
\end{align*}
By combining the previous inequality with the triangle inequality and~\eqref{eq:Newton-bd}, we obtain
$$
\norm{\dbar_k} \leq \norm{\dbar_k-d_k\s{N}} + \norm{d_k\s{N}}
\leq \frac{\norm{g_k}}{\theta} + \frac{\norm{g_k}}{\theta}
= \frac{2\norm{g_k}}{\theta}
\leq \frac{2\norm{g_k}}{\theta_{\min}},
$$
which complete the proof.
\end{proof}

The next result establishes a bound on the decrease in $F$ when $k\in\Sscr_\phi\s{SD}$.

\begin{lemma} \label{lem:dec-phi-SD}
If $k\in\Sscr_\phi\s{SD}$, then $x_{k+1}$ satisfies
\begin{equation} \label{F-dec-phiSD}
F(x_{k+1})
\leq F(x_k) - \kappa_\phi \max\{\gamma^{-2}\norm{\beta(x_k)}^2,\norm{\phi(x_k)}^2\},
\end{equation}
where
$\kappa_\phi := \eta_\phi^2 \min\left\{  \frac{\eta}{\theta_{\max}}, \frac{\eta\xi(1-\eta) \theta_{\min}^2}{2\theta_{\max}^3}  \right\} > 0$.
\end{lemma}
\begin{proof}
Let $k\in\Sscr_\phi\s{SD}$.  We consider two cases.  First, suppose that $j = 0$ when line~\ref{line:check-j} in Algorithm~\ref{alg:ls-phi} is reached.  In this case, it follows by construction of Algorithm~\ref{alg:ls-phi} that $\text{sgn}(y_0) = \text{sgn}(x_k + d_k) = \text{sgn}(x_k)$, i.e., the full step $d_k$ and the vector $x_k$ are contained in the same orthant.  Consequently, the loop that starts in line~\ref{line:ls-phi-loop} is simply a backtracking Armijo line search. Thus, %it is well known (e.g., see~\cite{blank}) that the inequality in line~\ref{line:ls-phi-SD} will hold anytime
if
\begin{equation} \label{eq:step-bound-1}
\xi^j \in
\left(0,\frac{2(\eta-1)[\Grad F(x_k)]_{\Iscr_k}\T [d_k]_{\Iscr_k}}{\theta_{\max}\norm{[d_k]_{\Iscr_k}}^2}\right]
\equiv \left(0,\frac{2(\eta-1)g_k\T \dbar_k}{\theta_{\max}\norm{\dbar_k}^2}\right],
\end{equation}
then, by well known properties of twice continuously differentiable functions with Lipschitz continuous gradients, we have that
\begin{align*}
  F(x_k+\xi^jd_k)
  &\leq F(x_k) + \xi^j[\nabla F(x_k)]_{\Iscr_k}^T[d_k]_{\Iscr_k} + \thalf \xi^{2j} \theta_{\max} \|[d_k]_{\Iscr_k}\|^2 \\
  &\leq F(x_k) + \xi^j[\nabla F(x_k)]_{\Iscr_k}^T[d_k]_{\Iscr_k} + \xi^j(\eta-1)[\nabla F(x_k)]_{\Iscr_k}^T[d_k]_{\Iscr_k} \\
  &=    F(x_k) + \eta \xi^j[\nabla F(x_k)]_{\Iscr_k}^T[d_k]_{\Iscr_k},
\end{align*}
i.e., the inequality in line~\ref{line:ls-phi-SD} will hold whenever \eqref{eq:step-bound-1} holds.
On the other hand, suppose that $j > 0$ when line~\ref{line:check-j} in Algorithm~\ref{alg:ls-phi} is reached.  Then, since $k\in\Sscr_\phi\s{SD}$, we may conclude that
\bequation\label{eq.merida}
  F(x_k+\alpha_B d_k) > F(x_k) + \eta \alpha_B [\Grad F(x_k)]_{\Iscr_k}^T [d_k]_{\Iscr_k} = F(x_k) + \eta \alpha_B g_k^T \dbar_k
\eequation
in line~\ref{line:check-B}, because otherwise we would have $k\in\Sscr_\phi\s{ADD} = \Sscr_\phi \setminus \Sscr_\phi\s{SD}$. Since no points of non-differentiability of $\|\cdot\|_1$ exist on the line segment connecting $x_k$ to $x_k+\alpha_B d_k$ (which follows by the definition of $\alpha_B$ in line~\ref{line:alphaB} of Algorithm~\ref{alg:ls-phi}), we can conclude for the same reason that we acquired~\eqref{eq:step-bound-1} that \eqref{eq.merida} implies
$$
\alpha_B > \frac{2(1-\eta)|g_k\T \dbar_k|}{\theta_{\max}\norm{\dbar_k}^2}.
$$
Combining these two cases, we have that the line search procedure in Algorithm~\ref{alg:ls-phi} will terminate with $x_{k+1} = x_k + \xi^j d_k$ where
\begin{equation} \label{eq:dec-phi}
  \xi^j \geq \min\left\{1, \frac{2\xi(1-\eta)|g_k\T \dbar_k|}{\theta_{\max}\norm{\dbar_k}^2}\right\}\ \ \text{and} \ \
  F(x_{k+1}) \leq F(x_k) + \eta \xi^j g_k^T \dbar_k.
\end{equation}
Let us now consider two cases.  First, suppose that $\xi^j = 1$ is returned from the line search, i.e., $j = 0$.  Then, it follows from
\eqref{eq:dec-phi}, lines~\ref{line:dbar} and~\ref{line:dR} of Algorithm~\ref{alg:main.x}, the Cauchy-Schwarz inequality, and Assumption~\ref{ass:main}
that
\begin{align}
 F(x_k) - F(x_{k+1}) &\geq - \eta \xi^j g_k^T \dbar_k = \eta |g_k\T \dbar_k|
 \geq \eta |g_k\T d_k^R| \nonumber \\
 &= \eta \alpha_k \norm{g_k}^2
 = \eta \frac{\norm{g_k}^4}{g_k\T H_k g_k}
 \geq \frac{\eta}{\theta_{\max}} \norm{g_k}^2. \label{eq:dec-phi-2}
\end{align}
Now suppose that $\xi^j < 1$.  Then, it follows from~\eqref{eq:dec-phi}, the inequality $|g_k\T\dbar_k| \geq \norm{g_k}^2/\theta_{\max}$ established while deriving~\eqref{eq:dec-phi-2}, and Lemma~\ref{lem:dbar-bound} that
\begin{align}
 F(x_k) - F(x_{k+1}) &\geq - \eta \xi^j g_k^T \dbar_k = \eta \xi^j|g_k\T \dbar_k|
 \geq \frac{2\eta\xi(1-\eta) |g_k\T \dbar_k|^2}{\theta_{\max}\norm{\dbar_k}^2}  \nonumber \\
 &\geq \frac{2\eta\xi(1-\eta) \theta_{\min}^2\norm{g_k}^4}{4\theta_{\max}^3\norm{g_k}^2}
 =  \left(\frac{\eta\xi(1-\eta) \theta_{\min}^2}{2\theta_{\max}^3}\right) \norm{g_k}^2. \label{eq:dec-phi-3}
\end{align}
Combining~\eqref{eq:dec-phi-2} and \eqref{eq:dec-phi-3} for the two cases establishes that
\begin{align*}
  F(x_k) - F(x_{k+1})
  &\geq \min\left\{  \frac{\eta}{\theta_{\max}}, \frac{\eta\xi(1-\eta) \theta_{\min}^2}{2\theta_{\max}^3}  \right\} \norm{g_k}^2 \\
  &=   \min\left\{  \frac{\eta}{\theta_{\max}}, \frac{\eta\xi(1-\eta) \theta_{\min}^2}{2\theta_{\max}^3}  \right\} \norm{[\Grad F(x_k)]_{\Iscr_k}}^2 \\
   &\geq  \min\left\{  \frac{\eta}{\theta_{\max}}, \frac{\eta\xi(1-\eta) \theta_{\min}^2}{2\theta_{\max}^3}  \right\} \norm{[\phi(x_k)]_{\Iscr_k}}^2 \\
    &\geq  \eta_\phi^2 \min\left\{  \frac{\eta}{\theta_{\max}}, \frac{\eta\xi(1-\eta) \theta_{\min}^2}{2\theta_{\max}^3}  \right\} \norm{\phi(x_k)}^2 \ \ \text{for $k\in\Sscr_\phi\s{SD}$},
 \end{align*}
where we have also used the condition in lines~\ref{line:Ik-phi} of Algorithm~\ref{alg:main.x} and the definition of $\phi(x_k)$. The inequality~\eqref{F-dec-phiSD} follows from the previous inequality and the fact that $\norm{\beta(x_k)} \leq \gamma\norm{\phi(x_k)}$ for all $k\in\Sscr_\phi \subseteq \Sscr_\phi\s{SD}$ as can be seen by line~\ref{line:main.if} of Algorithm~\ref{alg:main.x}.
\end{proof}

We may now establish finiteness of the index set $\Sscr_\phi\s{SD}$.

\begin{lemma} \label{lem:SphiSD-finite}
The index set $\Sscr_\phi\s{SD}$ is finite, i.e., $|\Sscr_\phi\s{SD}| < \infty$.
\end{lemma}
\begin{proof}
To derive a contradiction, suppose that $|\Sscr_\phi\s{SD}| = \infty$, which means that Algorithm~\ref{alg:main.x} does not terminate finitely.  Thus, it follows from line~\ref{line:check-termination} of Algorithm~\ref{alg:main.x} that $\max\{\norm{\beta(x_k)},\norm{\phi(x_k)}\} > \epsilon$ for all $k \geq 0$.  Also, it follows from Lemmas~\ref{lem:Sbeta-finite} and \ref{lem:SphiAdd-finite} that there exists an iteration number $k_1$ such that $k\in\Sscr_\phi\s{SD}$ for all $k \geq k_1$. Thus, with Lemma~\ref{lem:dec-phi-SD}, we have for all $\ell \geq k_1$ that
\begin{align*}
F(x_{k_1}) - F(x_{\ell+1})
&= \sum_{k=k_1}^\ell \big[F(x_k) - F(x_{k+1})\big] \\
&=  \sum_{\substack{k\in\Sscr_\phi\s{SD},k_1\leq k \leq \ell}} \big[F(x_k) - F(x_{k+1})\big]  \\
&\geq  \sum_{\substack{k\in\Sscr_\phi\s{SD},k_1\leq k \leq \ell}}  \kappa_\phi \max\{\gamma^{-2}\norm{\beta(x_k)}^2,\norm{\phi(x_k)}^2\} \\
&\geq  \sum_{\substack{k\in\Sscr_\phi\s{SD},k_1\leq k \leq \ell}}  \kappa_\phi \min\{\gamma^{-2},1\} \epsilon^2.
\end{align*}
Rearranging the previous inequality shows that
\begin{align*}
\lim_{l\to\infty} F(x_{\ell+1})
&\leq \lim_{\ell \to\infty} \Big[F(x_{k_1}) - \sum_{\substack{k\in\Sscr_\phi\s{SD},k_1\leq k \leq \ell}}  \kappa_\phi\min\{\gamma^{-2},1\}\epsilon^2\Big] \\
&=  F(x_{k_1}) - \sum_{\substack{k\in\Sscr_\phi\s{SD},k_1\leq k}}  \kappa_\phi \min\{\gamma^{-2},1\} \epsilon^2
= -\infty,
\end{align*}
which contradicts Assumption~\ref{ass:main}.  Thus, we conclude that $|\Sscr_\phi\s{SD}| < \infty$.
\end{proof}

We now prove our first main convergence result.

\begin{theorem}
Algorithm~\ref{alg:main.x} terminates finitely.
\end{theorem}
\begin{proof}
Since each iteration number $k$ generated in the algorithm is an element of $\Sscr_\beta \cup \Sscr_\phi\s{ADD} \cup \Sscr_\phi\s{SD}$, the result follows by Lemmas~\ref{lem:Sbeta-finite}, \ref{lem:SphiAdd-finite}, and \ref{lem:SphiSD-finite}.
\end{proof}

Our final convergence result states what happens when the finite termination criterion is removed from Algorithm~\ref{alg:main.x}.

\begin{theorem}
Let $x_*$ be the unique solution to problem~\eqref{prob.x}.  If $\epsilon$ in the finite termination condition in line~\ref{line:check-termination} of Algorithm~\ref{alg:main.x} is replaced by zero, then either:
\begin{itemize}
\item[(i)] there exists an iteration $k$ such that $x_k = x_*$; or %, $\varphi(x_k) = 0$, and $\beta(x_k) = 0$; or
\item[(ii)] infinitely many iterations $\{x_k\}$ are computed and they satisfy
   $$
   \lim_{k\to\infty} x_k = x_*, \ \
   \lim_{k\to\infty} \varphi(x_k) = 0, \ \  \text{and} \ \
   \lim_{k\to\infty} \beta(x_k) = 0.
   $$
\end{itemize}
\end{theorem}
\begin{proof}
If case $(i)$ occurs, then there is nothing left to prove.  Thus, for the remainder of the proof, we assume that case $(i)$ does not occur.  Since case $(i)$ does not occur, we know that Algorithm~\ref{alg:main.x} performs an infinite sequence of iterations.  Let us then define the set
$\Sscr := \Sscr_\beta \cup \Sscr_\phi\s{SD}$, which must be infinite (since any consecutive subsequence of iterations in $\Sscr_\phi\s{ADD}$ must be finite by the finiteness of $n$).  It follows from~\eqref{F-dec-beta} for $k\in\Sscr_\beta$, \eqref{F-dec-phiSD} for $k\in\Sscr_\phi\s{SD}$, and Assumption~\ref{ass:main} (specifically, the assumption that $f$ is bounded below over $\Lscr$) that
$$
\lim_{k\in\Sscr} \max\{ \|\beta(x_k)\|, \|\varphi(x_k)\| \} = 0.
$$
Combining this with Assumption~\ref{ass:main} and Lemma~\ref{lem:criticality} gives
\begin{equation} \label{lim-on-S}
\lim_{k\in\Sscr} x_k = x_*.
\end{equation}
Now, we claim that the previous limit holds over all iterations.  To prove this by contradiction, suppose that there exists an infinite $\Kscr\subseteq \Sscr_\varphi\s{ADD}$ and a scalar $\varepsilon > 0$ with
\begin{equation} \label{eq:x-away}
  \|x_k - x_*\| \geq \varepsilon \ \ \text{for all $k\in\Kscr$.}
\end{equation}
From Assumption~\ref{ass:main}, we conclude that there exists $\delta > 0$ such that
\begin{equation} \label{eq:in-ball}
\text{if} \ F(x) \leq F(x_*) + \delta, \  \text{then} \   \|x-x_*\| < \varepsilon.
\end{equation}
Moreover, from~\eqref{lim-on-S} and Assumption~\ref{ass:main}, there exists a smallest $k_S\in\Sscr$ such that
\begin{equation} \label{F-bound-1}
F(x_{k_S}) \leq F(x_*) + \delta.
\end{equation}
There then exists a smallest $k_K\in\Kscr$ such that $k_K > k_S$. Since, by construction, $\{F(x_k)\}_{k\geq 0}$ is monotonically decreasing , we may conclude with \eqref{F-bound-1} that
\begin{equation} \label{F-bound-2}
F(x_{k_K}) \leq F(x_{k_S})
\leq  F(x_*) + \delta.
\end{equation}
Combining~\eqref{F-bound-2} and~\eqref{eq:in-ball}, we deduce that $\|x_{k_K} - x_*\| < \epsilon$, which contradicts~\eqref{eq:x-away} since $k_K \in \Kscr$. This completes the proof.
\end{proof}

%*********
% Section
%*********
\section{Numerical Results}
%---------------------------------------------------------

In this section, we present results when employing an implementation of \algacro{} to solve a collection of $\ell_1$-norm regularized logistic regression problems.  Such problems routinely arise in the context of model prediction, making the design of advanced optimization algorithms that efficiently and reliably solve them paramount in big data applications.  We first describe the datasets considered in our experiments, then describe some details of our implementation (henceforth simply referred to as \algacro{}), and then present the results of our experiments.

\subsection{Datasets} \label{subsec:data}
%---------------------------------------------------------
We tested \algacro{} on $\ell_1$-norm regularized logistic regression problems using 31 datasets (see Table~\ref{tab:datasetinfo}), 19 of which are available only after standard scaling practices have been applied.  For the remaining 12 datasets, we considered both unscaled and scaled versions, where, for each, the scaling technique employed is described in the last column of Table~\ref{tab:datasetinfo}.  A checkmark in the ``Unscaled'' column indicates that we were able to obtain an unscaled version of that dataset.

Most of the datasets in Table~\ref{tab:datasetinfo} can be obtained from the LIBSVM repository.\footnote{\url{https://www.csie.ntu.edu.tw/~cjlin/libsvmtools/datasets/}}  From this repository, we excluded all regression and multiple-class (greater than two) instances, except for mnist since it is such a commonly used dataset.  Since mnist is for  digit classification, we transformed it for binary classification by assigning the digits $0$--$4$ to the label $-1$, and the digits $5$--$9$ to the label $1$. The remaining datasets were binary classification examples from which we removed HIGGS, kdd2010(algebra), kdd2010(bridge to algebra), epsilon, url, and webspam since insufficient computer memory was available.  (All experiments were conducted on a 64-bit machine with an Intel I7 4.0GHz CPU and 16GB of main memory.)  Finally, for the adult data (a1a--a9a) and webpage data (w1a--w8a) we only used the largest instances, namely problems a9a and w8a. This left us with our final subset of datasets from LIBSVM.

In addition, we also tested \algacro{} on three other datasets: synthetic, gene-ad, and pathway-ad.  The synthetic set is a randomly generated non-diagonally dominant dataset created by the authors of OBA.  The sets gene-ad and pathway-ad are datasets related to Alzheimer's Disease. They were obtained by preprocessing the sets GSE4226\!~\footnote{\url{http://www.ncbi.nlm.nih.gov/geo/query/acc.cgi?acc=GSE4226}} and GSE4227\!~\footnote{\url{http://www.ncbi.nlm.nih.gov/geo/query/acc.cgi?acc=GSE4227}} using the method presented in~\cite{zhu2015pathway}, and merging the results into the single dataset: gene-ad. The gene data (gene-ad) was converted to pathway data (pathway-ad) using the ideas described in~\cite{zhu2015pathway}. The union of these three datasets and those from the LIBSVM repository comprised our complete test set.

%
%
%The data sets breast cancer, diabetes, german numer, ijcnn1, a9a, w8a, gisette, real-sim, rcv1, and news20 were obtained from the LIBSVM repository.~\footnote{\url{https://www.csie.ntu.edu.tw/~cjlin/libsvmtools/datasets/}}
%
% A9a, w8a are two binary data sets. Ijcnn1, news20, rcv1 are document data sets, where news20 is a news document collection, rcv1 consists of news stories from Reuters. Gisette is a handwriting digit recognition problem from NIPS 2003. We merge labels of mnist, another well-known handwriting digit data set into binary case(labels 0$\sim$4 into -1, labels 5$\sim$9 into 1).

%Problem gene-ad with gene expression data can be obtained at the following:~\footnote{\url{http://labs.med.miami.edu/myers/LFuN/data.html}}. Problem pathway-ad is PAS, a biomarker, data set from Qingsong Zhu, in JHMI(Johns Hopkins Medical School).

For the unscaled datasets (see column 4 in Table~\ref{tab:datasetinfo}), we adopted standard scaling techniques.  For problems scaled into $[-1,1]$ a simple linear scaling transformation was used.  For problem mnist, which was scaled into $[0,1)$, we used a common converting method in image processing.  Specifically, we defined
\begin{equation}\label{scale:mnist}
I(i,j)=\frac{P(i,j)}{2^b},
\end{equation}
where $P(i,j)$ is the given unscaled integer pixel value satisfying
$$
P(i,j)\in \{0,1,2,\cdots,2^b-1\},
$$
$b$ is the intensity resolution ($b=8$ for the mnist dataset), and $(i,j)$ range over the size of the image.  The scaled pixel values are then given by the values $I(i,j)\in[0,1)$.

%The performance of FaRSA and OBA on both scaled, and unscaled data sets is reported.

\begin{table}[ht]
%\small
%\footnotesize
%\scriptsize
%\tiny
\center
\caption{Data sets.}
  \begin{tabular}{|c|c|c|c|c|}
  \hline
  Dataset  & \# of Samples & \# of Features &  Unscaled     & Scaling Used       \\ \hline
  fourclass     & 862               & 2                  &  $\checkmark$ & into [-1,1] \\
  svmguide1     & 3089              & 4                  & $\checkmark$  & into [-1,1] \\
  cod-rna       & 59535             & 8                  &               &             \\
  breast-cancer & 683               & 10                 &               &             \\
  australian    & 690               & 14                 &               &             \\
  SUSY          & 5000000           & 18                 &  $\checkmark$ & into [-1,1] \\
  splice        & 1000              & 60                 &               &             \\
  heart         & 270               & 13                 &               &             \\
  german.numer  & 1000              & 24                 &  $\checkmark$ & into [-1,1] \\
  diabetes      & 768               & 8                  &  $\checkmark$ & into [-1,1] \\
  liver-disorders& 345              & 6                  & $\checkmark$  & into [-1,1] \\
  w8a           & 49749             & 300                &               &             \\
  madelon       & 2000              & 500                &  $\checkmark$ & into [-1,1] \\
  a9a           & 32561             & 123                &               &             \\
  mnist         & 30001             & 784                &  $\checkmark$ & into [0,1)  \\
  skin-nonskin  & 245057            & 3                  &  $\checkmark$ & into [-1,1] \\
  sonar         & 208               & 60                 &               &             \\
  ijcnn1        & 49990             & 22                 &               &             \\
  svmguide3     & 1243              & 22                 &               &             \\
  synthetic     & 5000              & 5000               &               &             \\
  gisette       & 6000              & 5000               &               &             \\
  pathway-ad    & 278               & 71                 &               &             \\
  real-sim      & 72309             & 20958              &               &             \\
  covtype.binary& 581012            & 8                  &               &             \\
  mushrooms     & 8124              & 112                &               &             \\
  rcv1.binary   & 20242             & 47236              &               &             \\
  leukemia      & 34                & 7129               &               &             \\
  duke-breast-cancer& 38            & 7129               &  $\checkmark$ & into [-1,1] \\
  gene-ad       & 71                & 17375              &  $\checkmark$ & into [-1,1] \\
 % gene-ad2      & 123               & 17375              &  $\checkmark$ & scaled into [-1,1] \\
  colon-cancer  & 62                & 2000               &  $\checkmark$ & into [-1,1] \\
  news20        & 19996             & 1355191            &               &             \\
  \hline
\end{tabular}
%\vspace*{-0.2cm}
\label{tab:datasetinfo}
\end{table}

\subsection{Implementation details}
%----------------------------------------------
We developed a preliminary \Matlab{} implementation of \algacro{} that we are happy to provide upon request.  In this section, we describe the algorithm-specific choices made to obtain the results that we present.

First, the weighting parameter in~\eqref{prob.x} was defined as
$$
\lambda = \frac{1}{\text{\# of Samples}}.
$$
For determining the iteration type, we chose $\gamma = 1$ in line~\ref{line:main.if} of Algorithm~\ref{alg:main.x} so that no preference was given to iterations being in either $\Sscr_\phi$ or $\Sscr_\beta$.

For any $k\in\Sscr_\phi$, we made the simple choice of $\Iscr_k = \{i:[\phi(x_k)]_i \neq 0\}$.  This made the inequality in line~\ref{line:Ik-phi} satisfied for any $\eta_\phi \in (0,1]$, making the choice of this parameter irrelevant.  (In a more sophisticated implementation, one might consider other choices of $\Iscr_k$, say to adaptively control $|\Iscr_k|$, to improve efficiency.)  With this choice for $\Iscr_k$ made, Algorithm~\ref{alg:main.x} allows for great flexibility in obtaining a search direction that satisfies the conditions in line~\ref{line:dbar} (see (iii) in Section~\ref{subsec.literature} for additional comments). For our tests, we applied the linear-CG method to the system $H_k d = -g_k$ defined by the terms constructed in line~\ref{line:gH}, except that we added a diagonal matrix with entires $10^{-8}$ to $H_k$ (an approach also adopted by OBA and LIBLINEAR).  As discussed in Section~\ref{sec.algorithm}, the conditions that are required to be satisfied by the trial step will hold if CG is terminated during any iteration. To help limit the number of backtracking steps required by the subsequent backtracking line search, we terminated CG as soon as one of three conditions was satisfied.  To describe these conditions, we let $d_j$ denote the $j$th CG iteration, $r_j = \norm{H_k d_j + g_k}$ denote the $j$th CG residual, and $v_j$ denote the number of components in $x_k+d_j$ that fall into a different orthant than $x_k$. With these definitions, we terminated CG as soon as one of the following was satisfied:
\begin{align*}
r_j &\leq \max\{10^{-1}r_0, 10^{-12}\}, \\
v_j &\geq \max\{ 10^3, 10^{-1}|\Iscr_k| \}, \sgap \text{or} \\
\norm{d_j} &\geq \delta_{k,\phi} := \max\{10^{-3}, \min\{10^3,10 \norm{x_{k_\phi(k)+1}-x_{k_\phi(k)}}\} \},
\end{align*}
where $k_\phi(k) := \max\{\kbar : \kbar\in\Sscr_\phi \ \text{and} \ \kbar < k\}$. This first condition is a standard requirement of asking the residual to be reduced by a fraction of the initial residual.  We used the second condition to trigger termination when a CG iterate predicted that ``too many'' of the variables at $x_k+d_j$ are in the ``wrong'' orthant.  Finally, the third condition ensured that the size of the trial step was moderate, thus functioning as an implicit trust-region constraint; this condition was motivated by the well-known fact that CG iterations $\{d_j\}$ are monotonically increasing in norm.

When $k\in\Sscr_\beta$, we again made the simples choice of $\Iscr_k = \{i:[\beta(x_k)]_i \neq 0\}$, making the choice of $\eta_\beta \in (0,1]$ irrelevant in our tests (though adaptive choices of $\Iscr_k$ might be worthwhile in a more sophisticated implementation).  Since there is no natural scaling for the direction $\beta(x_k)$ because it is based on first derivative information only, it is important from a practical perspective to adaptively scale the direction.  Therefore, in line~\ref{line:d-beta}, we used the alternative safeguarded direction defined by
\begin{equation}
[d_k]_{\Iscr_k} = -\delta_{k,\beta}\frac{[\beta(x_k)]_{\Iscr_k}}{\norm{[\beta(x_k)]_{\Iscr_k}}},
\end{equation}
where
$$
\delta_{k,\beta} := \max\{10^{-5}, \min\{1,\norm{x_{k_\beta(k)+1}-x_{k_\beta(k)}}\} \}
$$
with $k_\beta(k) := \max\{\kbar : \kbar\in\Sscr_\beta \ \text{and} \ \kbar < k\}$. Since this is a safeguarded scaling of the $d_k$ defined in line~\ref{line:d-beta},  it is fully covered by the theory that we developed in Section~\ref{sec:convergence}.

During each iteration, the values $\eta = 10^{-2}$ and $\xi = 0.5$ were used during the line search regardless of whether it was the line search performed by Algorithm~\ref{alg:ls-phi} when called by Algorithm~\ref{alg:main.x} (line~\ref{line:ls-phi}) or if it was the line search performed by Algorithm~\ref{alg:ls-beta} when called by Algorithm~\ref{alg:main.x} (line~\ref{line:ls-beta}). The starting point $x_0$ was chosen as the zero vector for all problems, and the  termination tolerance, maximum allowed number of iterations, and maximum allowed time limit values were chosen to be $\epsilon = 10^{-6}$,  $1000$, and $10$ minutes, respectively.

\begin{table}[ht]
%\small
%\footnotesize
\center
\caption{CPU time and sparsity for \algacro{} and OBA on scaled problem variants.}
  \begin{tabular}{|c|ccc|cc|} \hline
\multicolumn{1}{|c|}{} &
\multicolumn{3}{c|}{\rm Time (seconds)} &
\multicolumn{2}{c|}{\rm \% of zeros} \\ \hline
  Problems & \algacro & OBA & OBA/\algacro & \algacro & OBA\\
  \hline
  fourclass     &  \textcolor{red}{0.00326} & 0.00705                   & 2.1626 & 0             & 0              \\
  svmguide1     &  \textcolor{red}{0.0384}  & 0.06457                   & 1.6815 & 0             & 0              \\
  cod-rna       &  0.48762                  & \textcolor{red}{0.18618}  & 0.3818 & 0             & 0              \\
  breast-cancer & \textcolor{red}{0.0089}   & 0.03769                   & 1.9674 & 0             & 0             \\
  australian    & \textcolor{red}{0.01443}  & 0.0174                    & 1.2058 & 0             & 0              \\
  SUSY          & 241.2437                  & \textcolor{red}{205.1242} & 0.8502 & 0             & 0              \\
  splice        & \textcolor{red}{0.0101}   & 0.01982                   & 1.9624 & 5             & 5              \\
  heart         & \textcolor{red}{0.00706}       & 0.01357                   & 1.9221 & 7.7           & 7.7             \\
  german.numer  & \textcolor{red}{0.01159}  & 0.02111                   & 1.8214 & 8.3           & 8.3           \\
  diabetes      & \textcolor{red}{0.00581}  & 0.00979                   & 1.6850 & 12.5          & 12.5          \\
  liver-disorders& \textcolor{red}{0.01254}                  & max iter                  & Inf    & 16.7          & ---             \\
  w8a           & \textcolor{red}{0.97079}  & 0.99154                   & 1.0214 & \textcolor{red}{19.1} & 18.7  \\
  madelon       & \textcolor{red}{0.26604}  & 0.37497                   & 1.4094 & 19.8          & 19.8              \\
  a9a           & \textcolor{red}{0.78203}  & 3.26994                   & 4.1813 & \textcolor{red}{22.0} & 20.3  \\
  mnist         & \textcolor{red}{18.78034} & 54.432                    & 3.0545 & 37.7          & \textcolor{red}{37.8} \\
  skin-nonskin  & \textcolor{red}{3.20594}  & ascent                    & Inf & 41.7          & ---          \\
  sonar         & \textcolor{red}{0.02012}  & 0.02938                   & 1.4602 & 41.7          & 41.7              \\
  ijcnn1        & \textcolor{red}{0.06153}  & 0.08178                   & 1.3291 & 45.5          & 45.5          \\
  svmguide3      & \textcolor{red}{0.01856}  & 0.03478                   & 1.8739 & 45.5          & 45.5            \\
  synthetic     & 45.82688                  & \textcolor{red}{20.42464} & 0.4457 & \textcolor{red}{57.4} & 49.5          \\
  gisette       & \textcolor{red}{13.30533} & 28.22136                  & 2.1211 & 84.6          & 84.6          \\
  pathway-ad    & \textcolor{red}{0.16054}  & 1.30585                   & 8.1341 & 87.4          & 87.4          \\
  real-sim      & \textcolor{red}{2.3221}   & 2.43764                   & 1.0214 & \textcolor{red}{91.9} & 91.8          \\
  covtype.binary& \textcolor{red}{1.49449}  & 5.95536                   & 3.9849 & \textcolor{red}{96.3}          & 90.7               \\
  mushrooms     & \textcolor{red}{0.03089}  & 0.05815                   & 1.8825 & 97.3          & 97.3              \\
  rcv1.binary   & \textcolor{red}{0.39186}  & 0.80563                   & 1.7427 & 98.8          & 98.8          \\
  %gene-ad2      & 0.46874                   & \textcolor{red}{0.24338}  & 0.5192 & 99.7          & 99.7          \\
  leukemia      & \textcolor{red}{0.09151}  & 0.12086                   & 1.3207 & 99.7          & 99.7          \\
  duke-breast-cancer& \textcolor{red}{0.06227}& 0.10628                 & 1.7068 & 99.7          & 99.7              \\
  gene-ad       & 0.21525                   & \textcolor{red}{0.15943}  & 0.7407 & 99.8          & 99.8          \\
  colon-cancer  & 0.04069                   & \textcolor{red}{0.03905}  & 0.9597 & 99.9          & 99.9              \\
  news20        & \textcolor{red}{6.09086}  & 19.77945                  & 3.2474 & 99.9          & 99.9          \\
  \hline
\end{tabular}
%\vspace*{-0.2cm}
\label{tab:runtimecmpscaled}
\end{table}

\begin{table}[ht]
%\small
%\footnotesize
\center
\caption{CPU time and sparsity for \algacro{} and OBA on unscaled problem variants.}
  \begin{tabular}{|c|ccc|cc|} \hline
\multicolumn{1}{|c|}{} &
\multicolumn{3}{c|}{\rm Time (seconds)} &
\multicolumn{2}{c|}{\rm \% of zeros} \\ \hline
 Problems & \algacro & OBA & OBA/\algacro & \algacro & OBA \\
  \hline
  fourclass    & \textcolor{red}{0.00486}  & 0.00775             & 1.5946 & 0    & 0          \\
  diabetes     & \textcolor{red}{0.01964}  & 0.02159             & 1.0993 & 0    & 0          \\
  german.numer & \textcolor{red}{0.03168}  & 0.0564              & 1.7803 & 0    & 0          \\
  skin-nonskin & \textcolor{red}{0.11378}  & \text{ascent} & Inf    & 0 & --- \\
  madelon      & \textcolor{red}{6.55674}  & 41.9519             & 6.3983 & 8    & 8          \\
  liver-disorders& \textcolor{red}{0.00571} & 0.03277            & 5.7391 & 83.3 & 83.3       \\
  colon-cancer & 0.08429                   & \textcolor{red}{0.05364}& 0.6364 & 98.7 & 98.7       \\
  duke-breast-cancer& \textcolor{red}{0.08936} & 0.12487            & 1.3973 & 99.7 & 99.7       \\
  gene-ad      & \textcolor{red}{4.62839}  & \text{ascent} & Inf    & 99.8 & --- \\
  svmguide1    & max iter   & max iter      & ---    & ---  &  --- \\
  mnist        & max time   & \text{ascent} & ---    & --- & --- \\
  SUSY         & max iter   & max iter      & ---   & --- & --- \\
%  gene-ad2     & \textcolor{red}{21.56338} & \text{ascent} & Inf    & \textcolor{red}{99.4} & \text{acent}\\
  \hline
\end{tabular}
%\vspace*{-0.2cm}
\label{tab:runtimecmpunscaled}
\end{table}

\subsection{Test results}
%----------------------------------------
The output from \algacro{} for the problems corresponding to the scaled and unscaled datasets in our experiments are summarized in Tables~\ref{tab:runtimecmpscaled} and \ref{tab:runtimecmpunscaled}, respectively. These tables focus on the computational time in seconds and percentage of zeros (sparsity) in the computed solutions.  For comparison purposes, we also provide the output from the OBA solver whose \Matlab{} implementation was graciously provided by the authors.  For a fair comparison, we used the same stopping tolerance value of $\epsilon = 10^{-6}$ for OBA and made no modifications to their code.  The numbers reported for each problem (named according to the corresponding dataset) are the averages from running each problem instance $10$ times.  We do not provide the final objective values since they were the same for \algacro{} and OBA on all problems that were successfully solved by both algorithms.  We use red numbering to indicate that an average CPU time was relatively lower for an algorithm, or if the average percentage of zeros in the solution was relatively larger for an algorithm.

We can observe from Table~\ref{tab:runtimecmpscaled} that \algacro{} performed better than OBA on 26 of the 31 ($83.87\%$) scaled test problems. OBA is faster than \algacro{} only on problems cod-rna, SUSY, synthetic, gene-ad, and colon-cancer. However, \algacro{} is between $3$ and $8$ times faster than OBA on problems a9a, mnist, pathway-ad, covtype.binary, and news20, and between $1$ and $3$ times faster than OBA on the remaining $21$ problems. In terms of sparsity, the two algorithms are comparable.   Although not presented in the table, we find it interesting to note that \algacro{} required an average of $34.12$ iterations to solve the problems, with, on average, $29.93$ of them being in $\Sscr_\phi$. This indicates that \algacro{} quickly identifies the orthant that contains the optimal solution.

By turning our attention to Table~\ref{tab:runtimecmpunscaled}, we see that the performance of both \algacro{} and OBA deteriorates when the problems are unscaled. Moreover, OBA fails on problems skin-nonskin, gene-ad, and mnist because it generates iterates that increase the objective function; we denote these failures as ``ascent" in the table.  In theory, ascent is only possible for their method when their fixed estimate ($10^{8}$ in their code) of the Lipschitz constant for the gradient of $f$ is not large enough.  Although simple adaptive strategies could be used to avoid such issues, we made no such attempts because we did not want to make any edits to their code.  Overall, \algacro{} was able to solve 9 of the 12 unscaled problems, and OBA only performed better than \algacro{} on a single test problem (colon-cancer).

\begin{figure}[ht]
\centering
\includegraphics[width=1.2in]{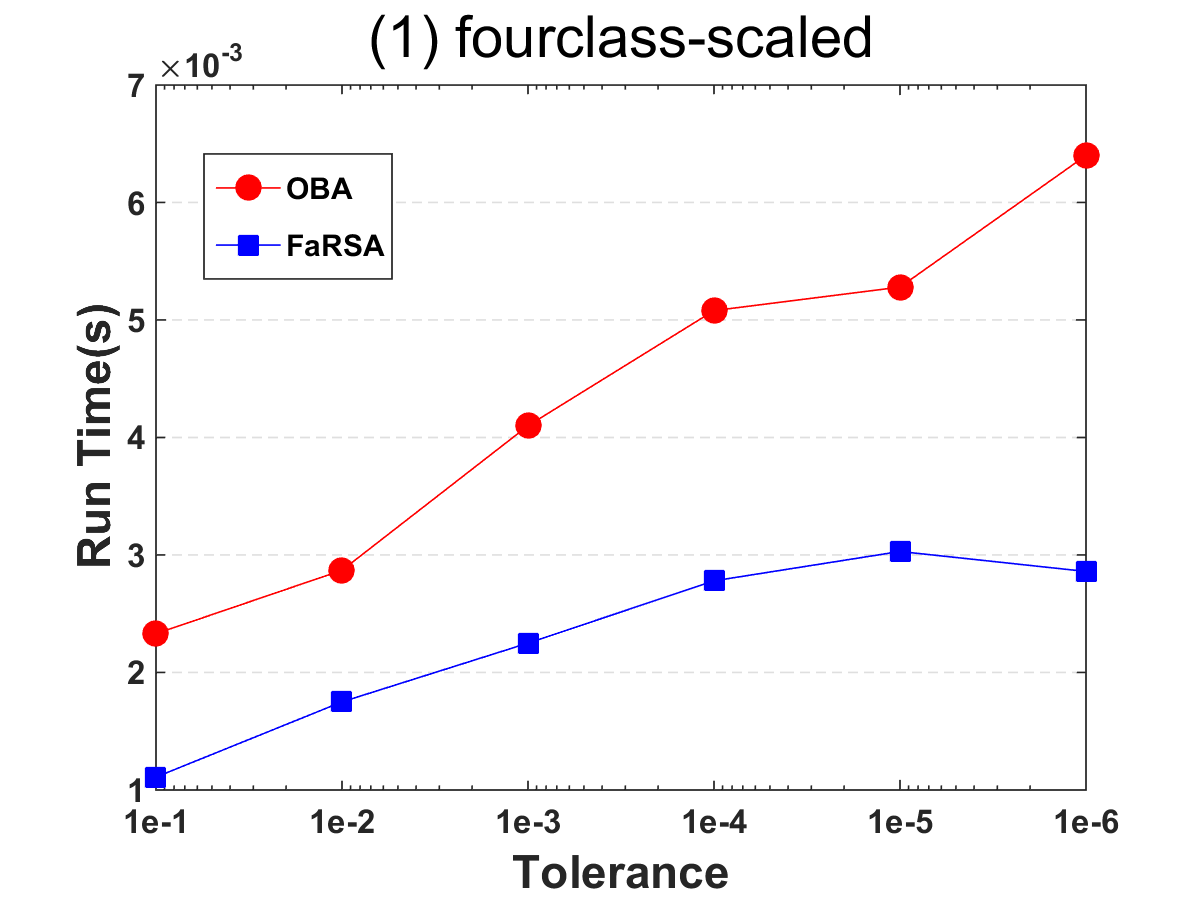}
\includegraphics[width=1.2in]{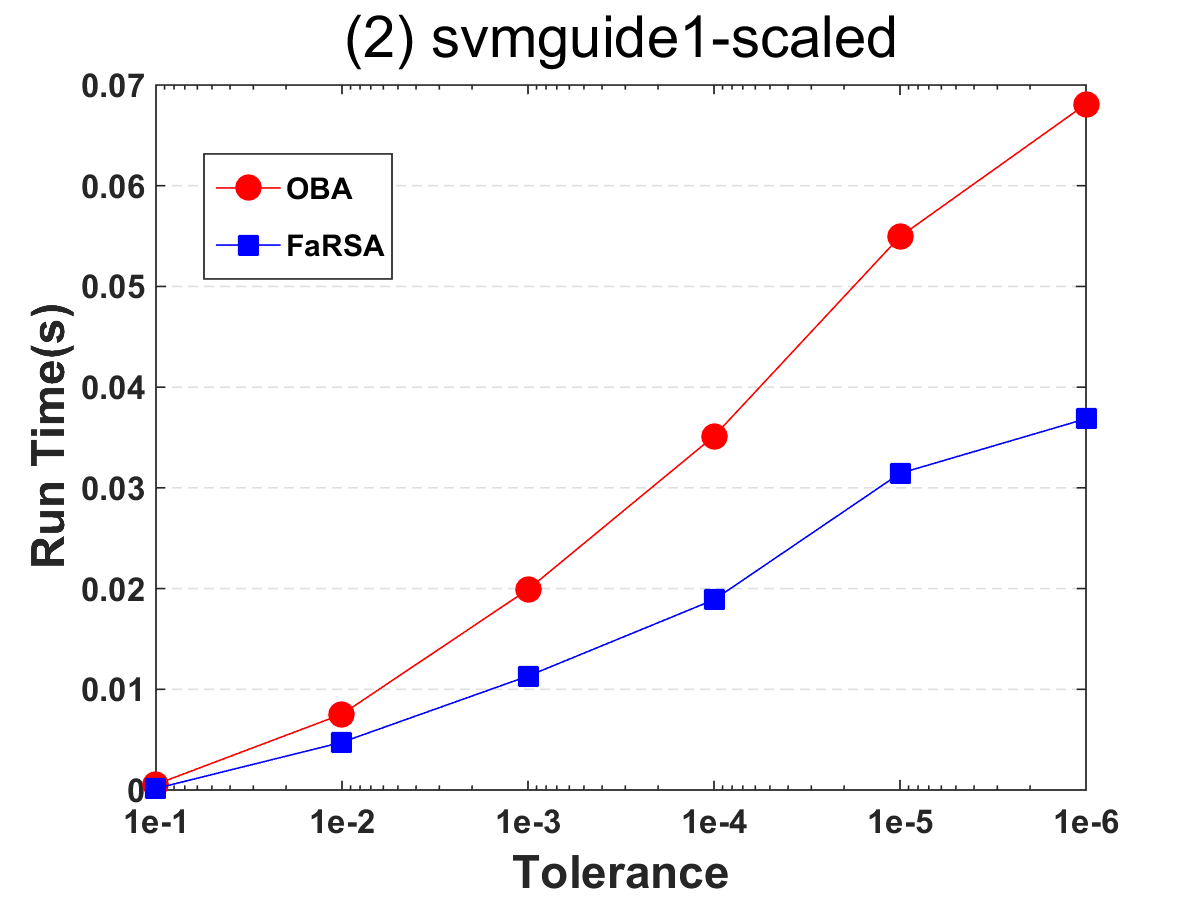}
\includegraphics[width=1.2in]{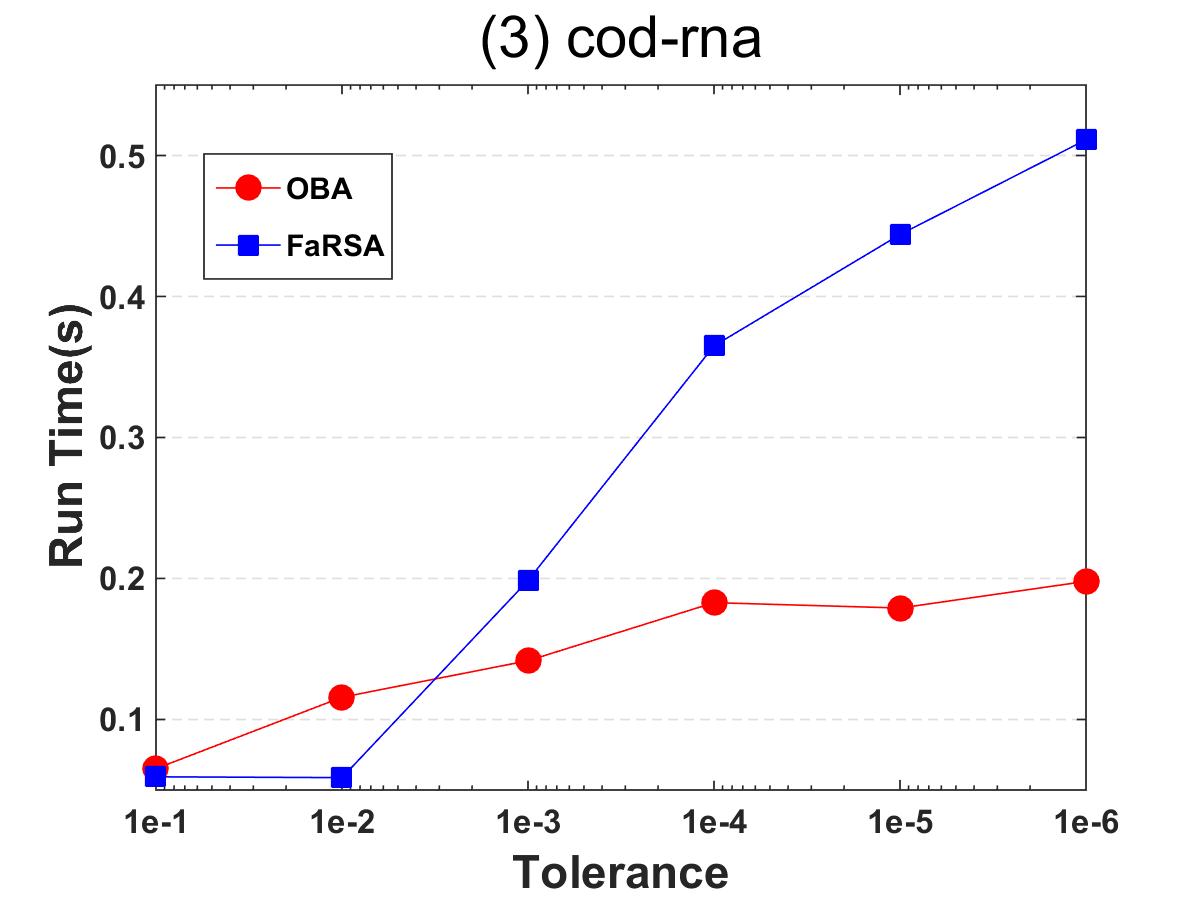}
\includegraphics[width=1.2in]{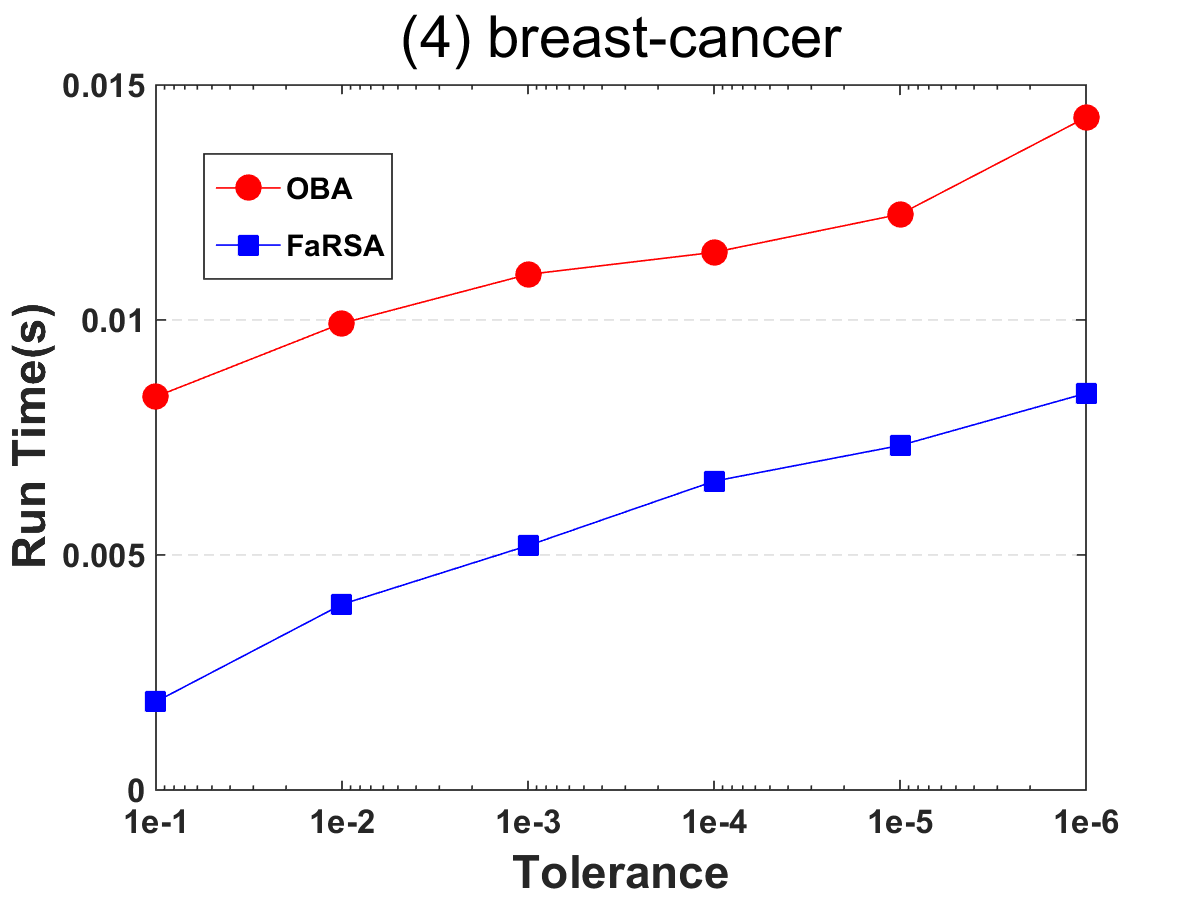}
\includegraphics[width=1.2in]{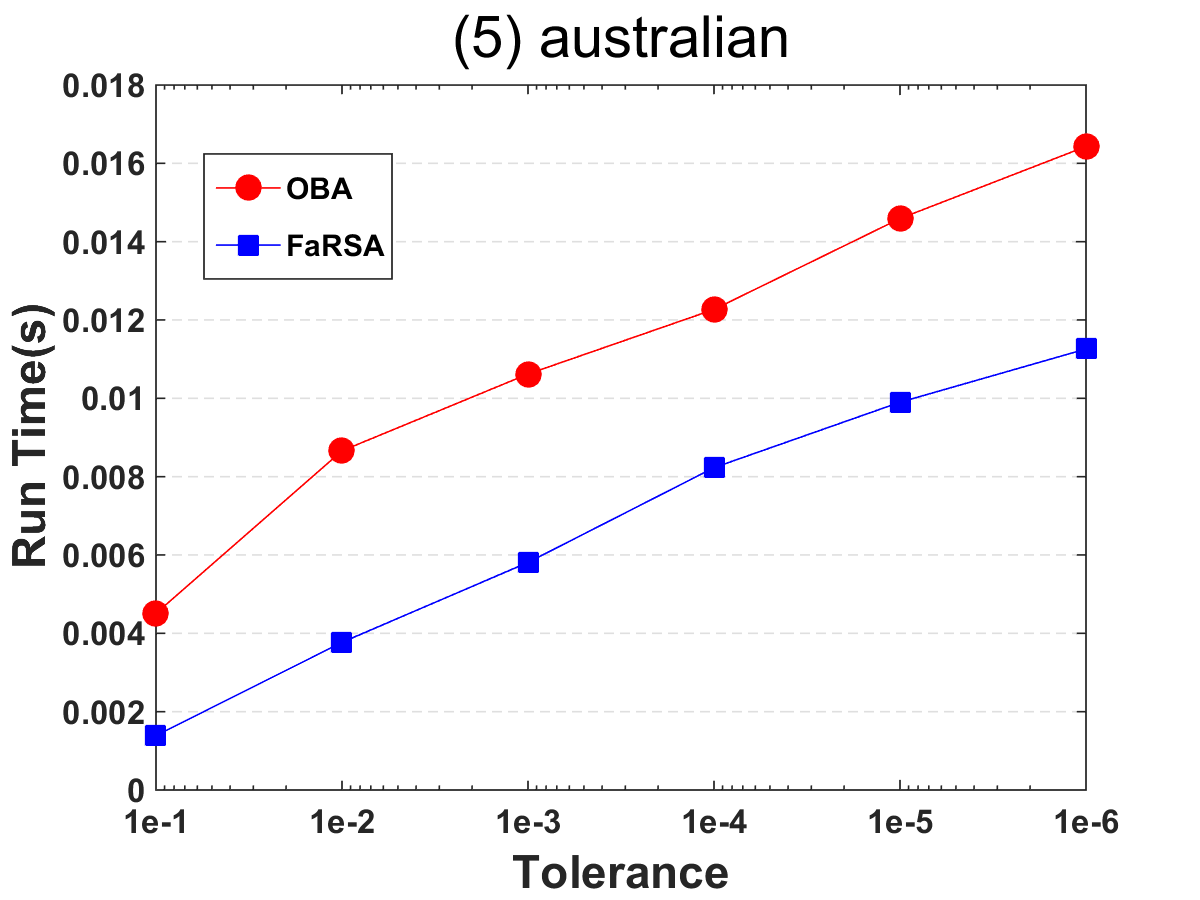}
\includegraphics[width=1.2in]{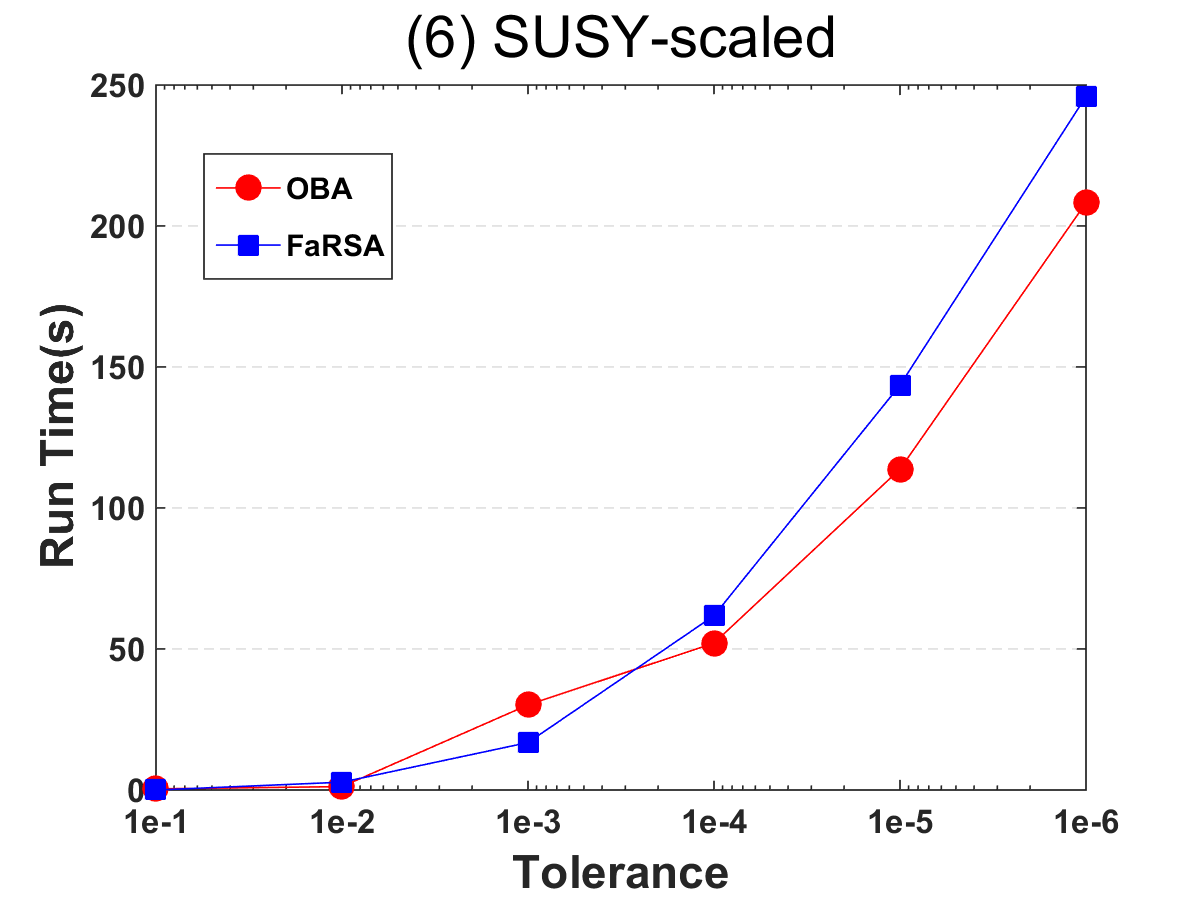}
\includegraphics[width=1.2in]{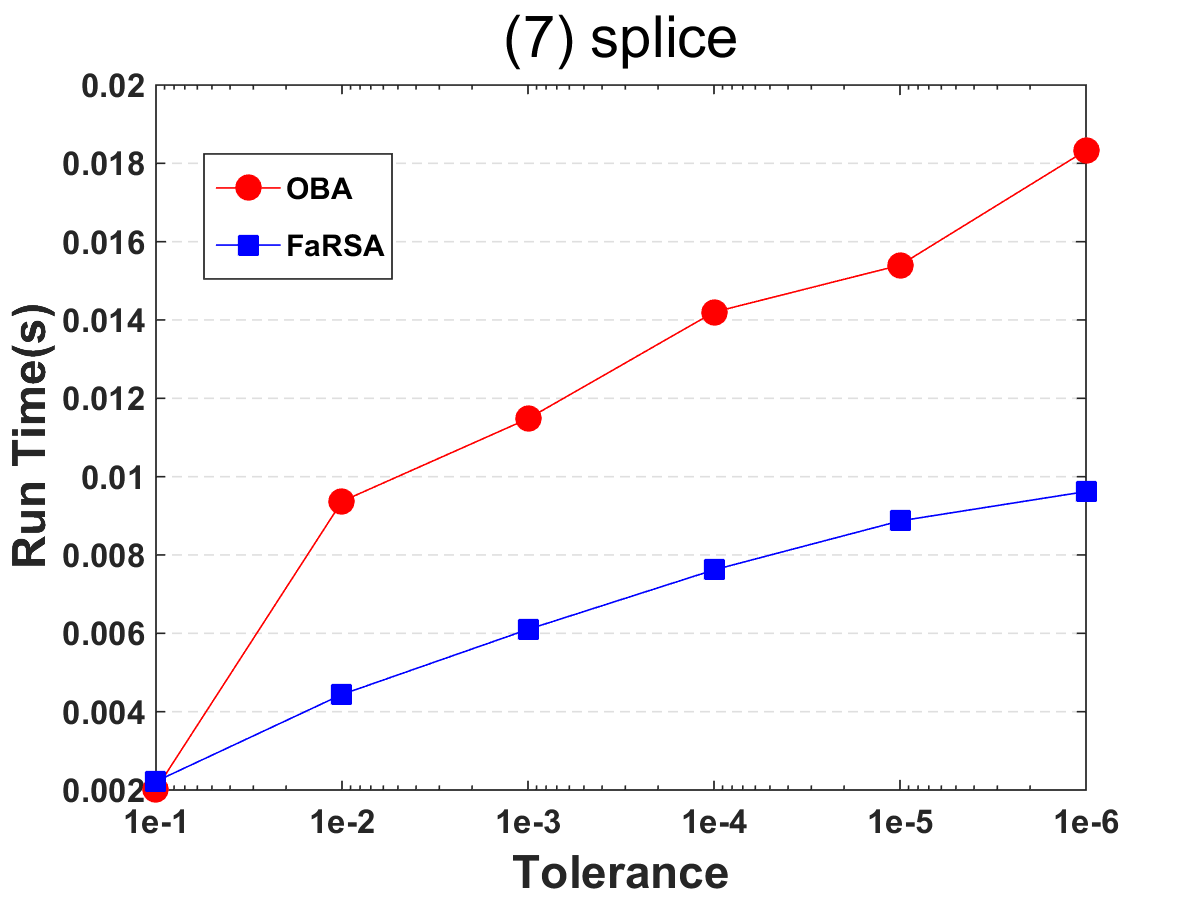}
\includegraphics[width=1.2in]{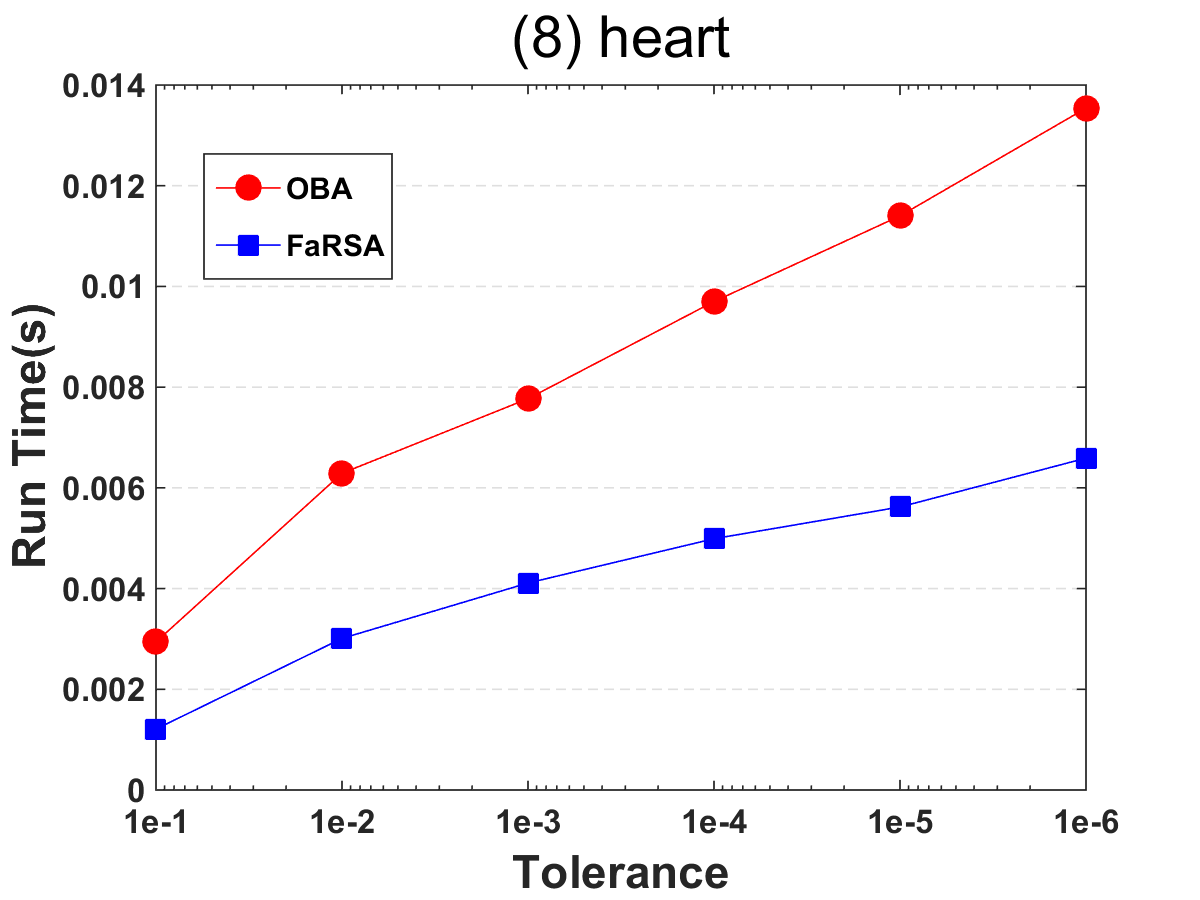}
\includegraphics[width=1.2in]{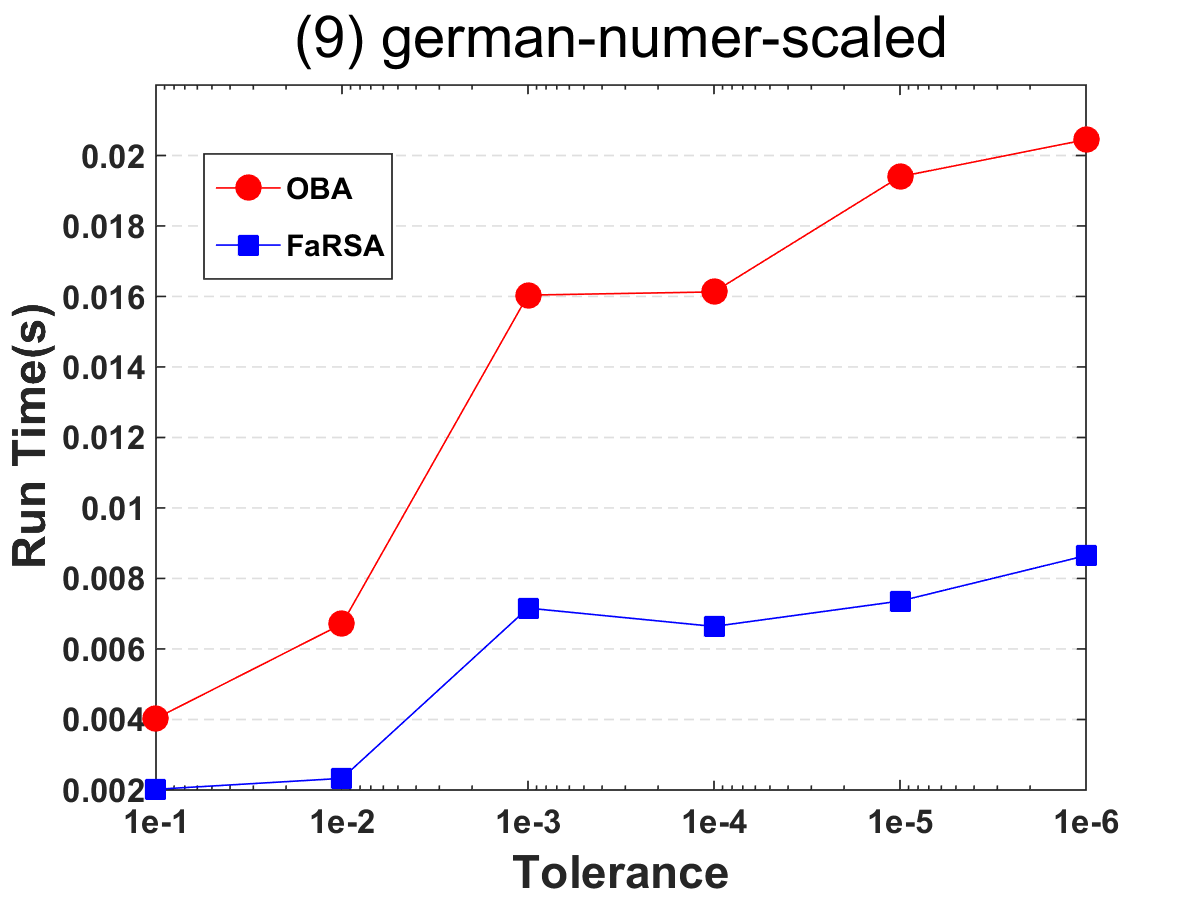}
\includegraphics[width=1.2in]{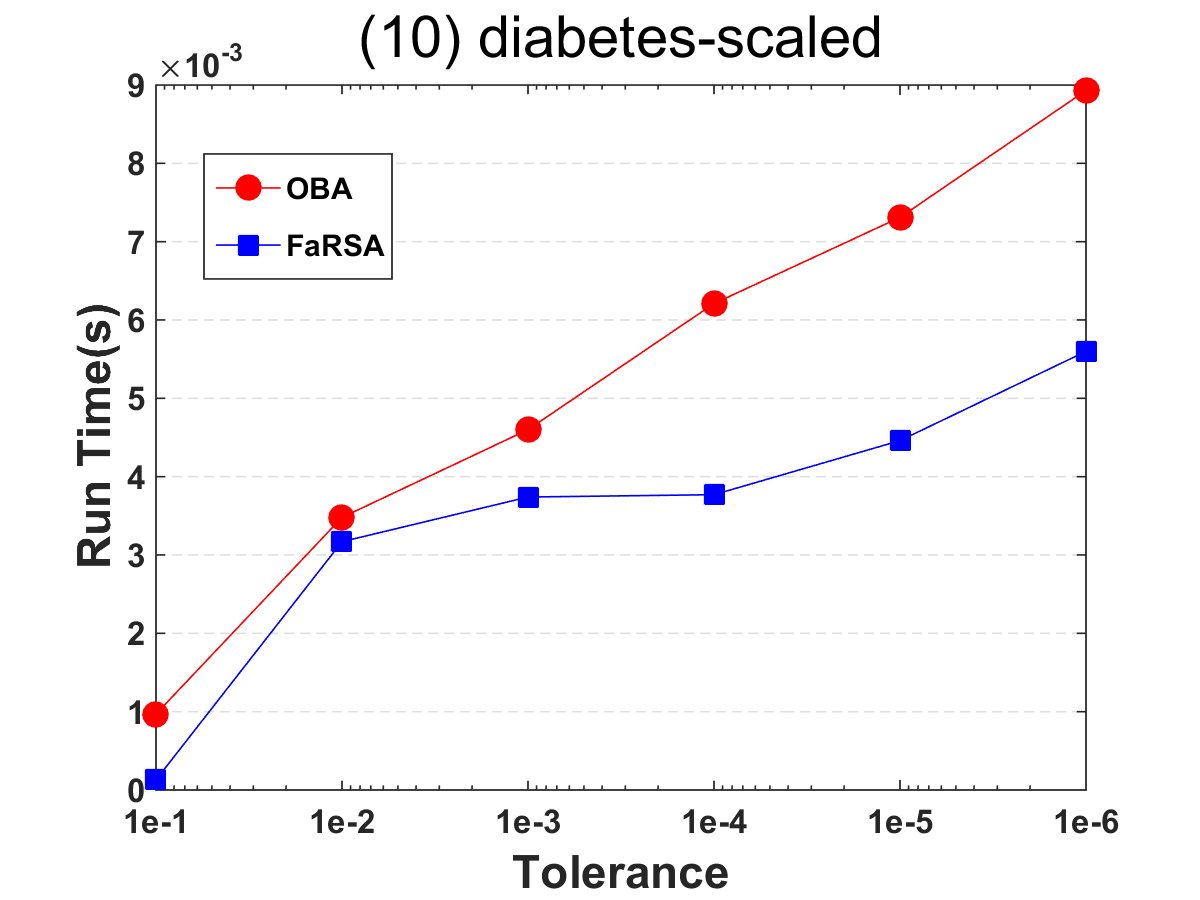}
\includegraphics[width=1.2in]{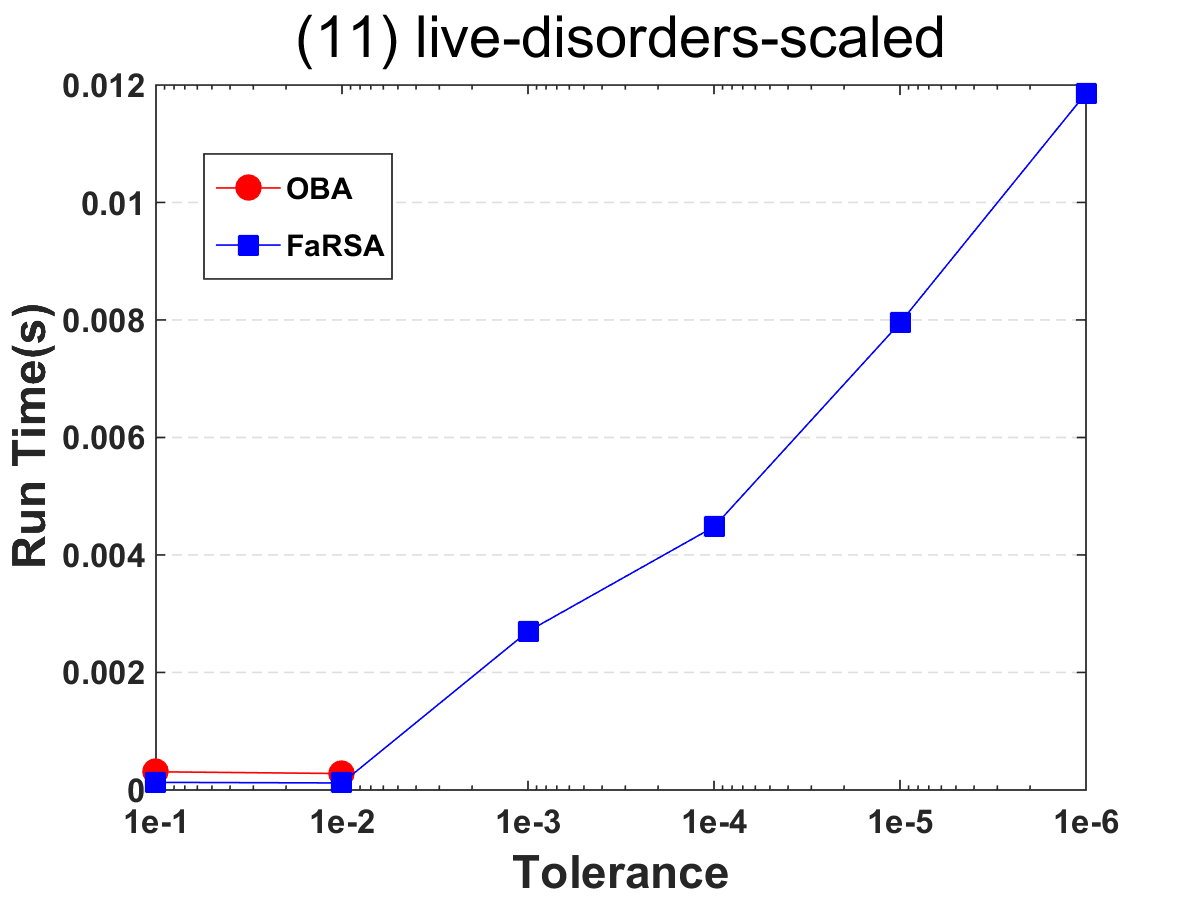}
\includegraphics[width=1.2in]{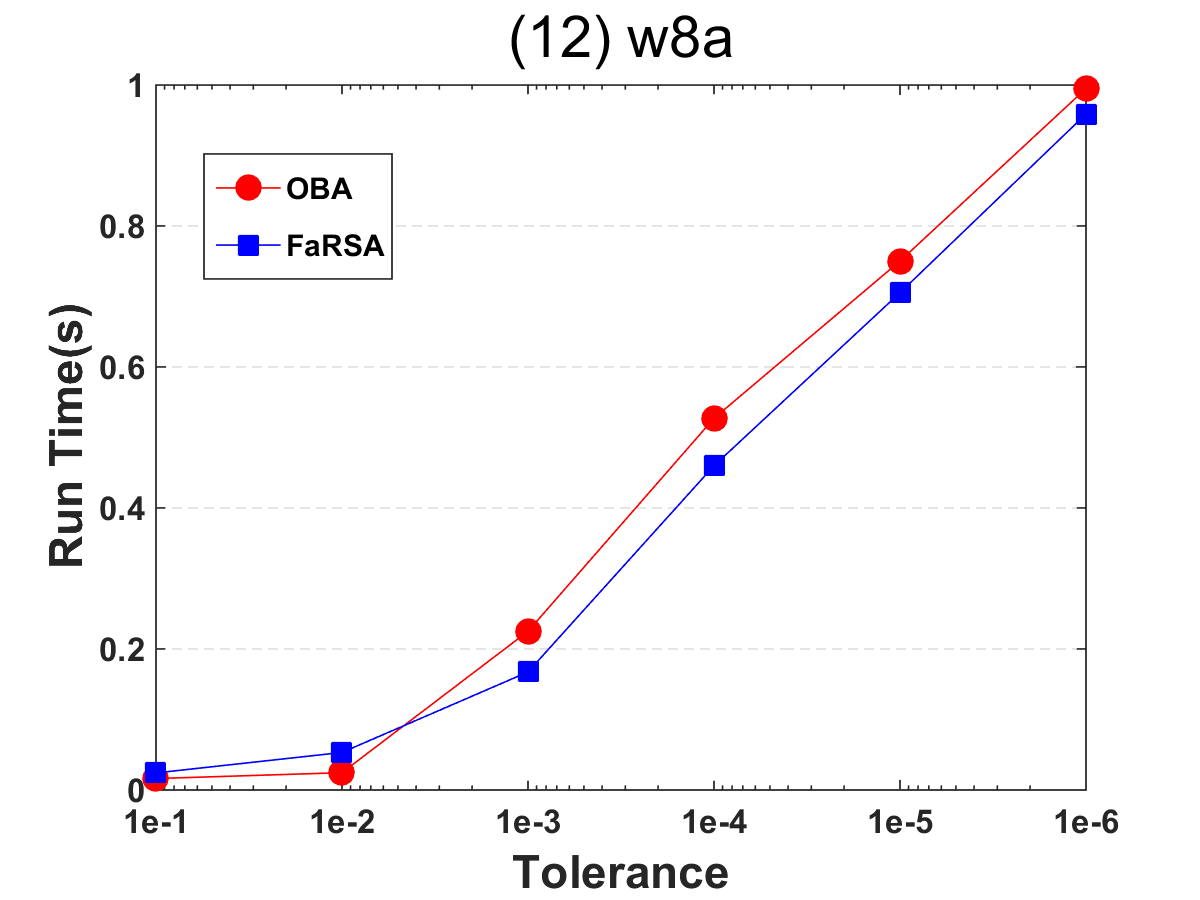}
\includegraphics[width=1.2in]{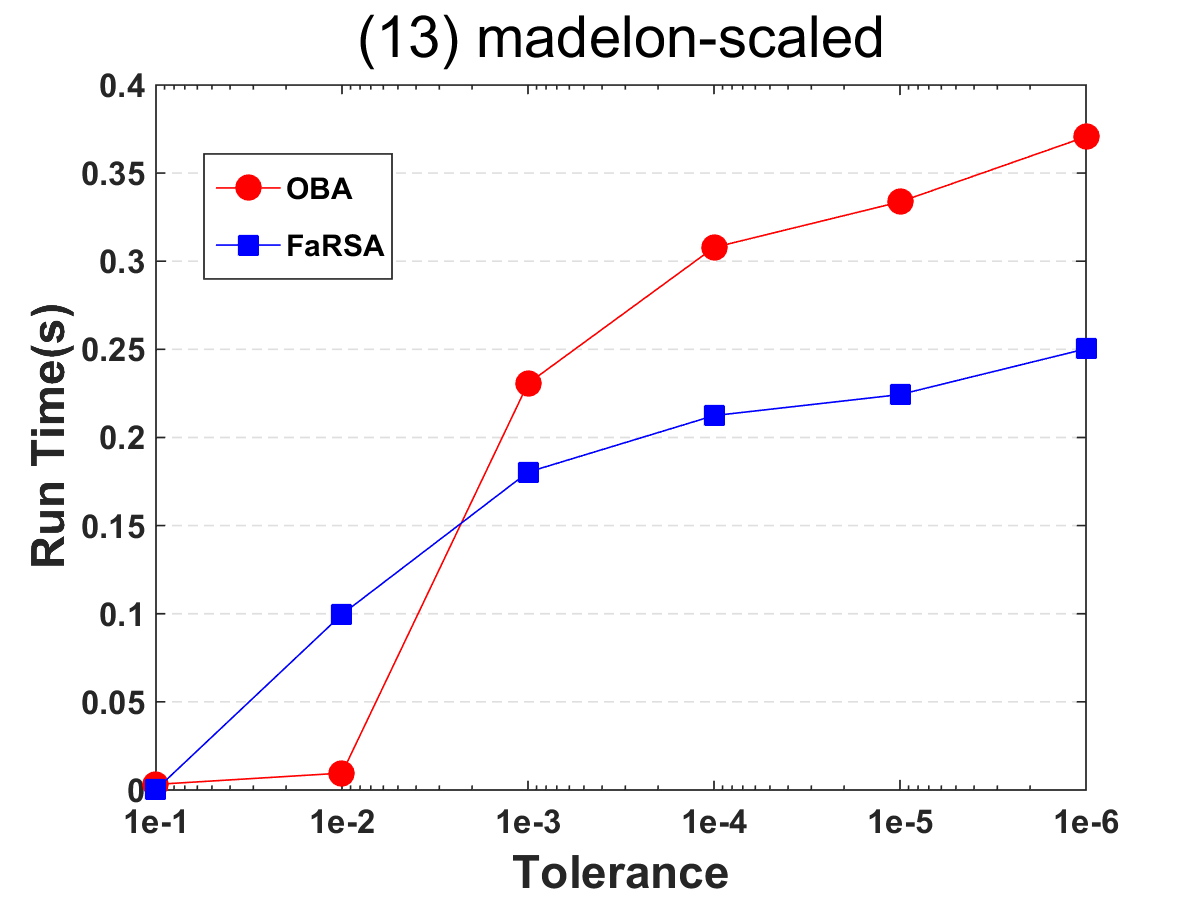}
\includegraphics[width=1.2in]{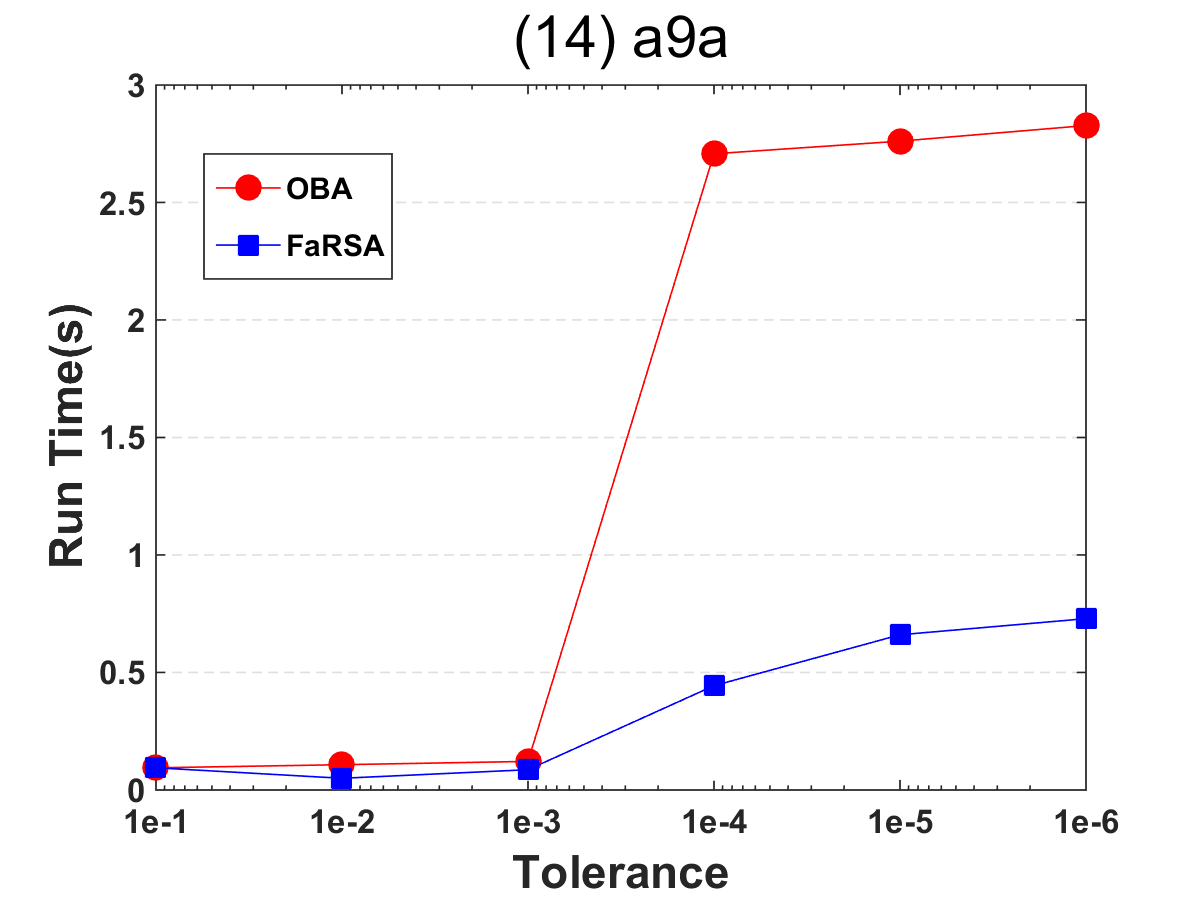}
\includegraphics[width=1.2in]{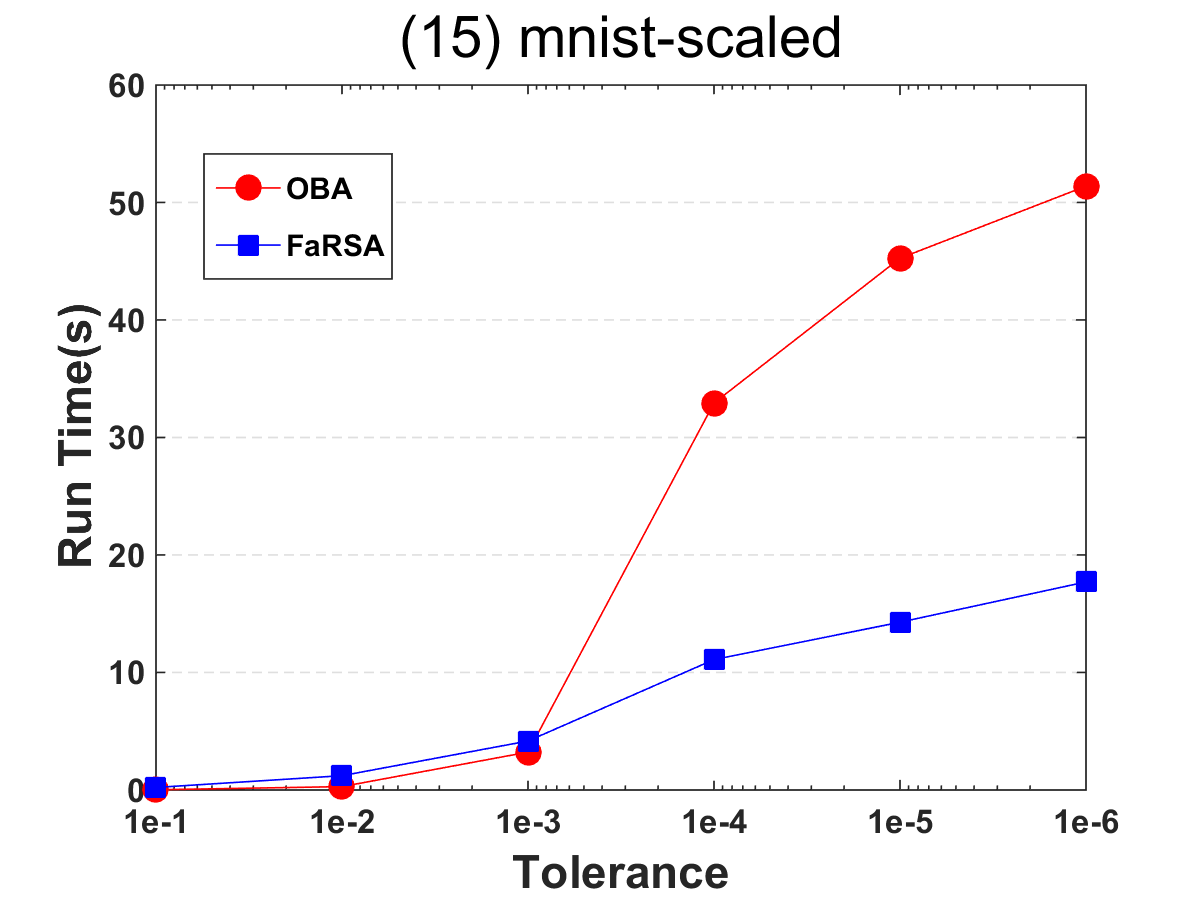}
\includegraphics[width=1.2in]{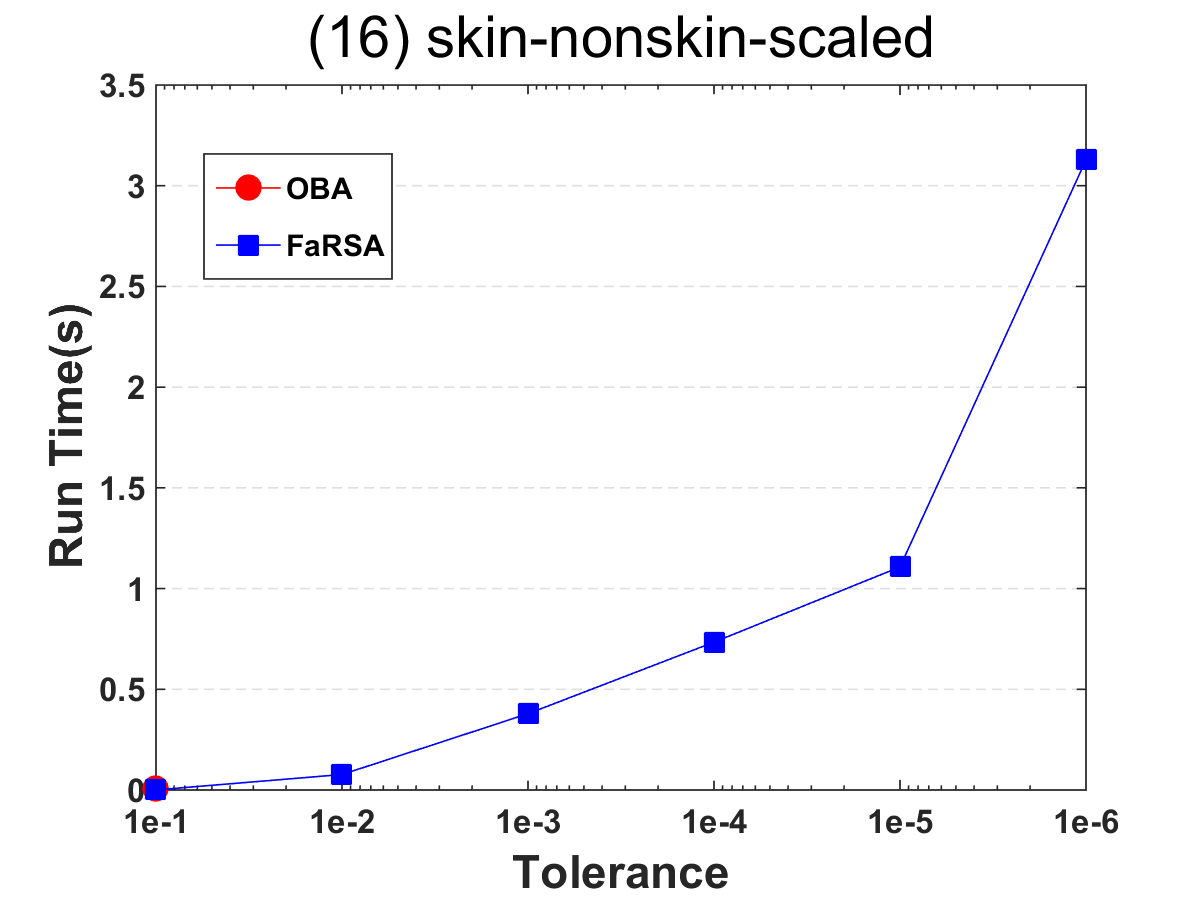}
\includegraphics[width=1.2in]{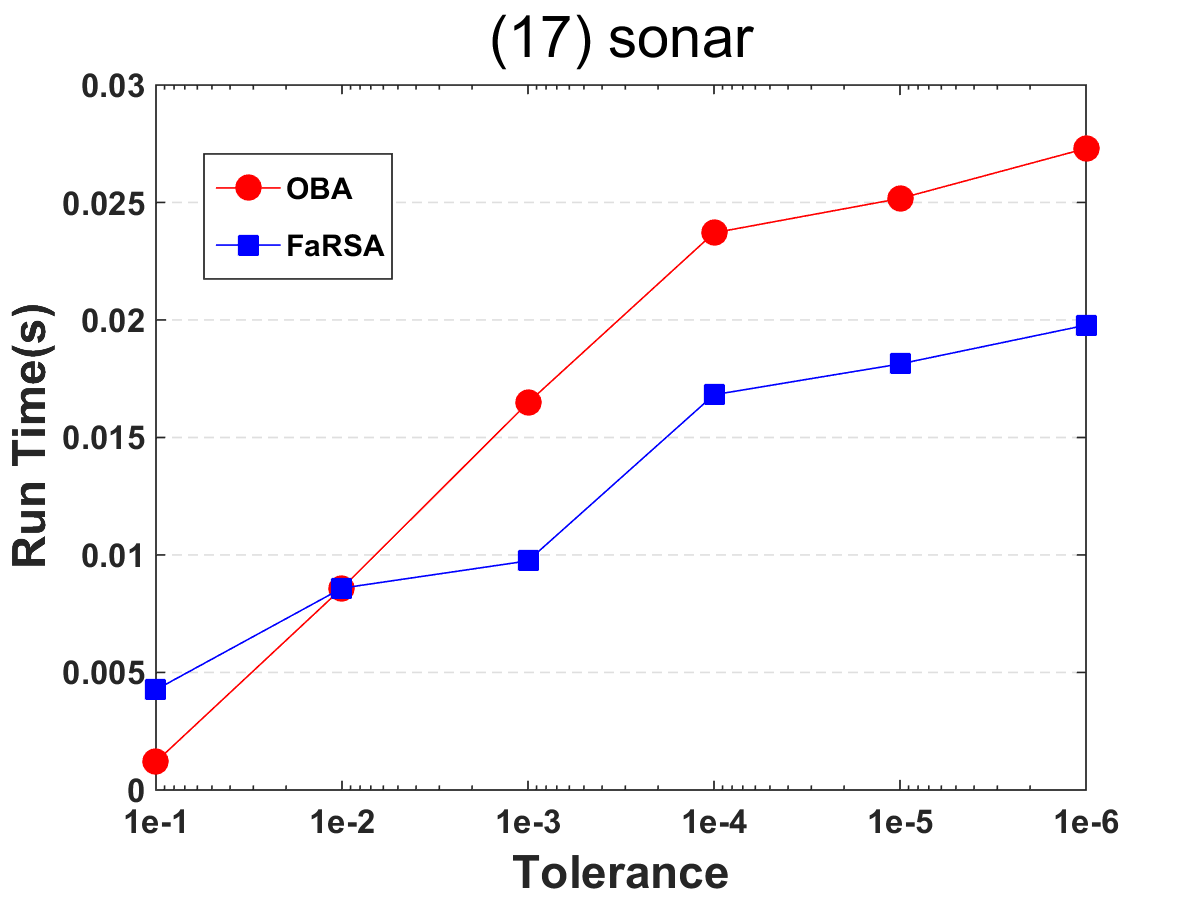}
\includegraphics[width=1.2in]{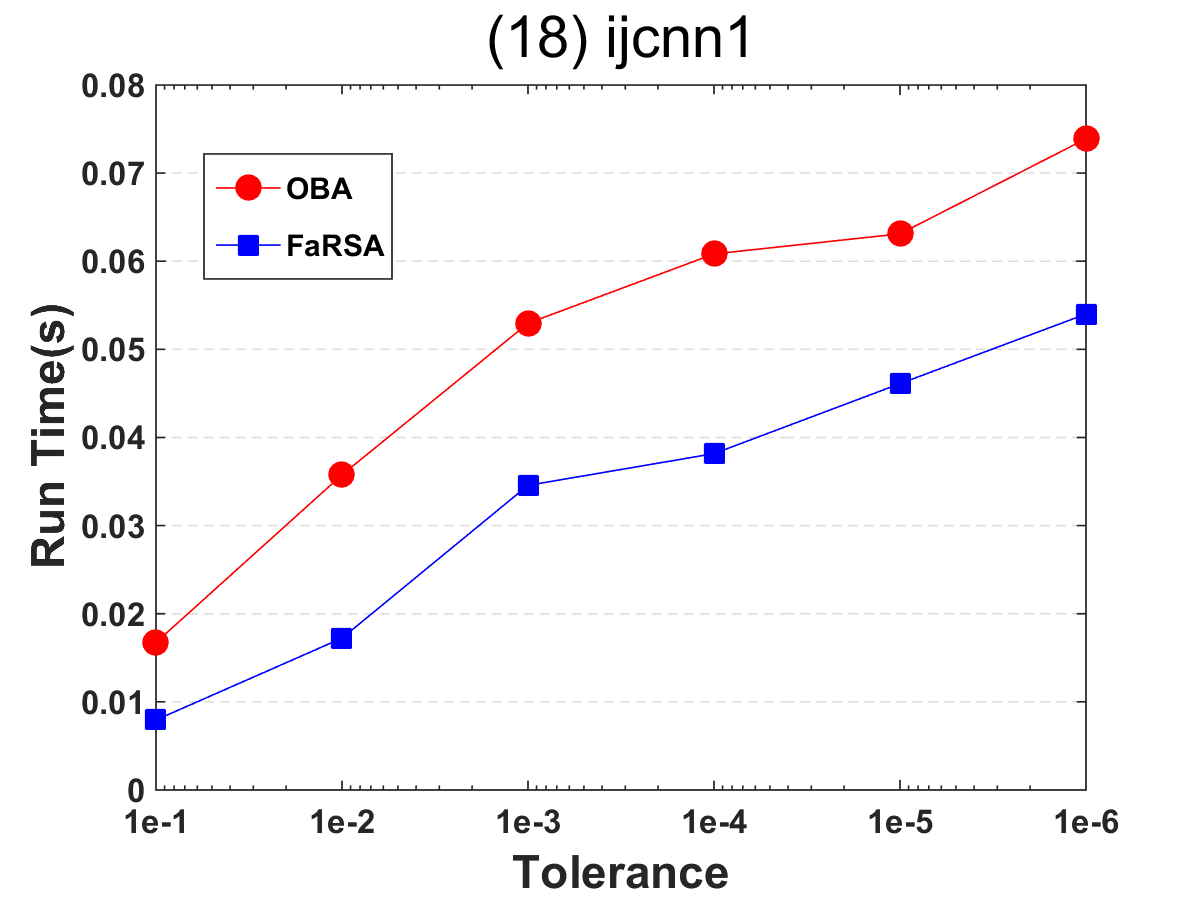}
\includegraphics[width=1.2in]{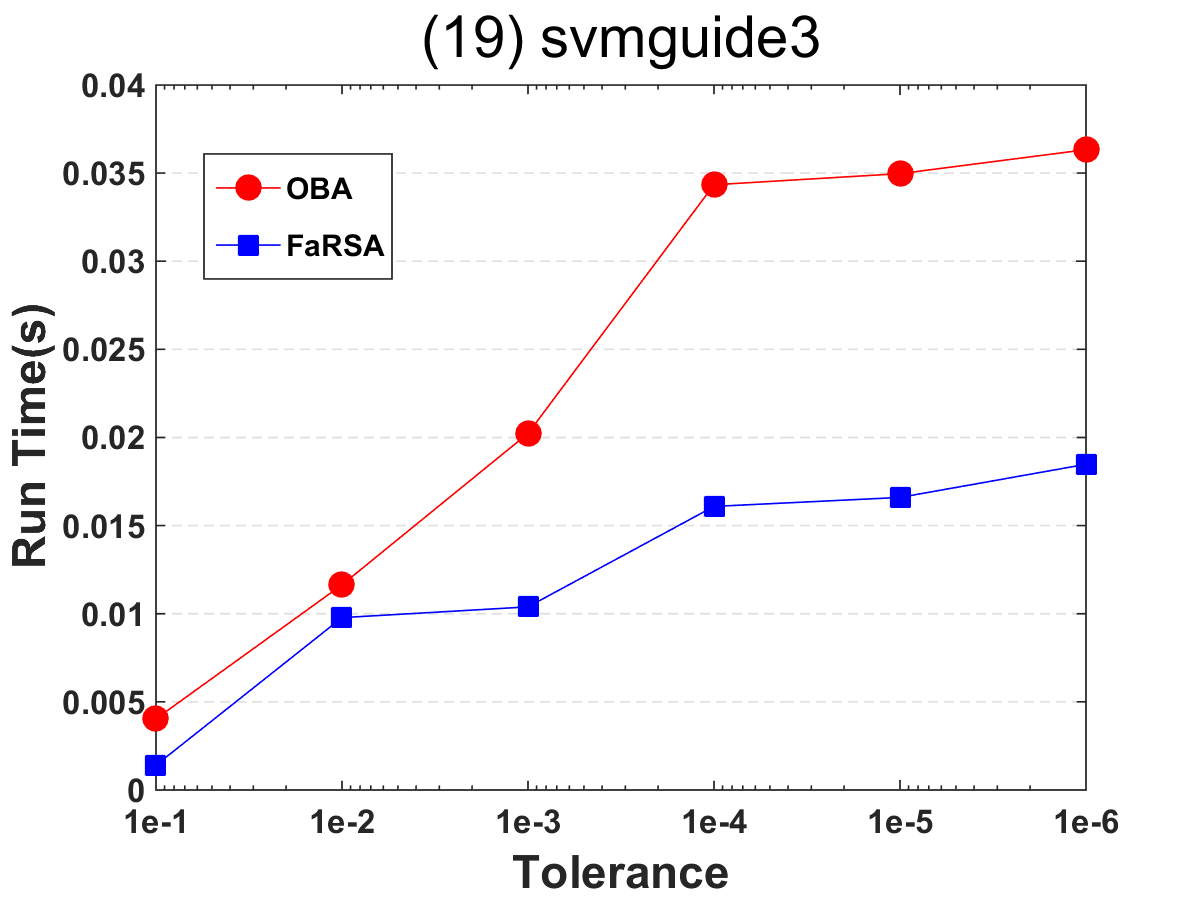}
\includegraphics[width=1.2in]{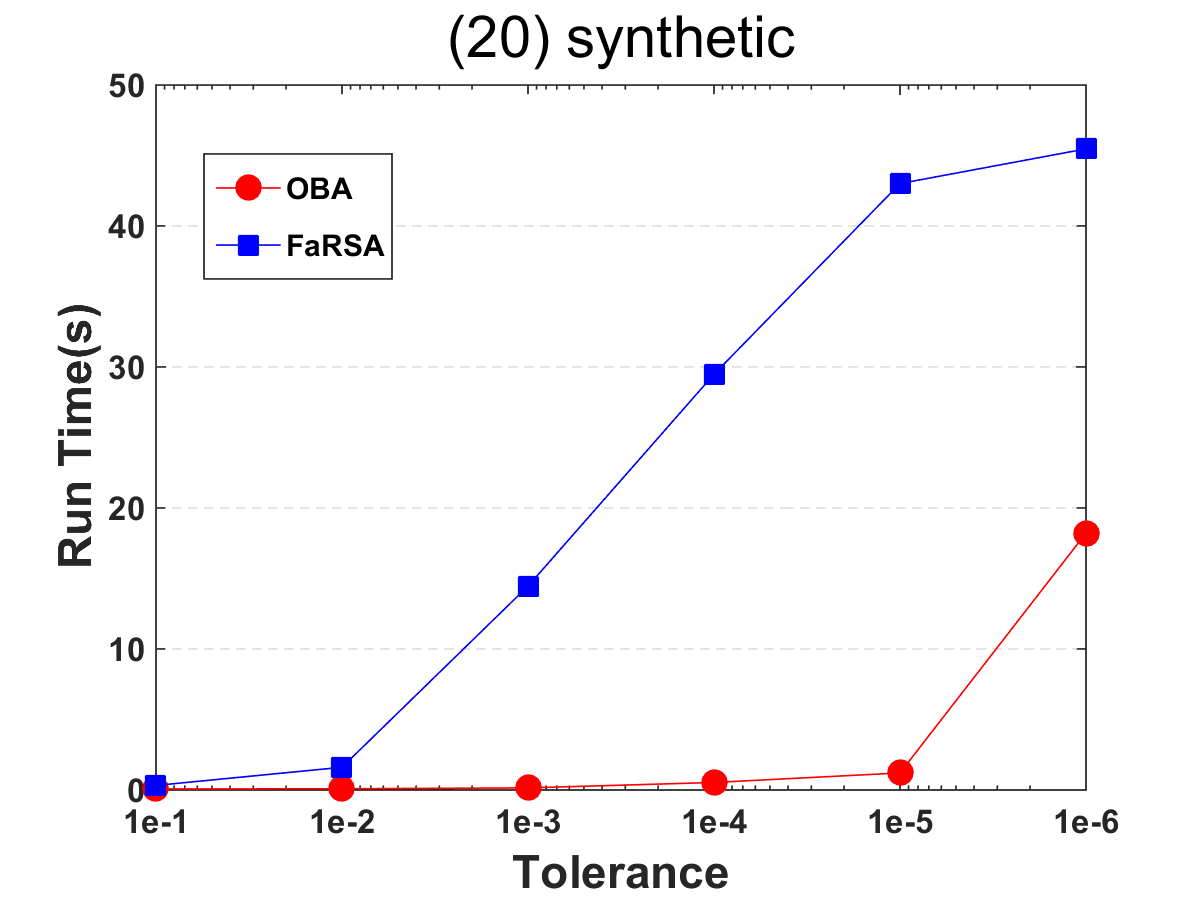}
\includegraphics[width=1.2in]{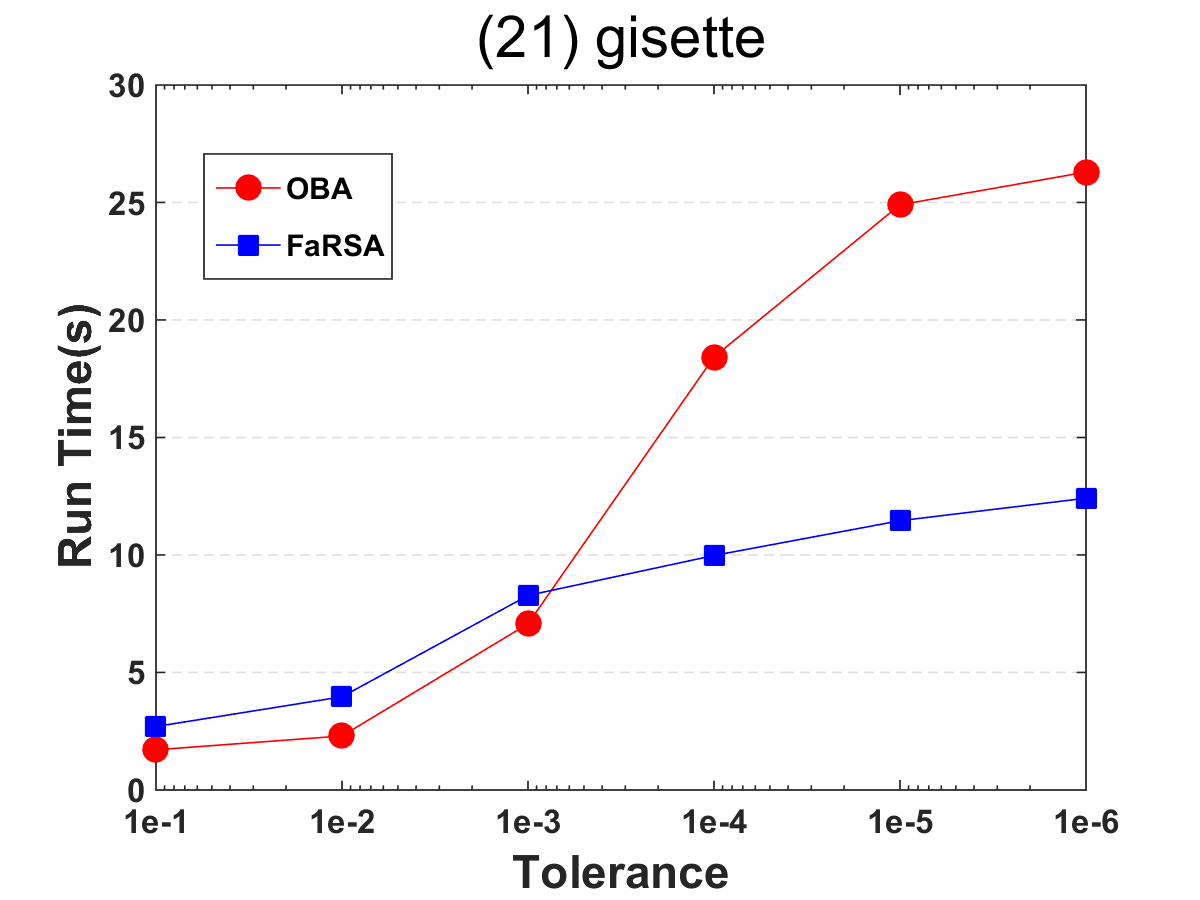}
\includegraphics[width=1.2in]{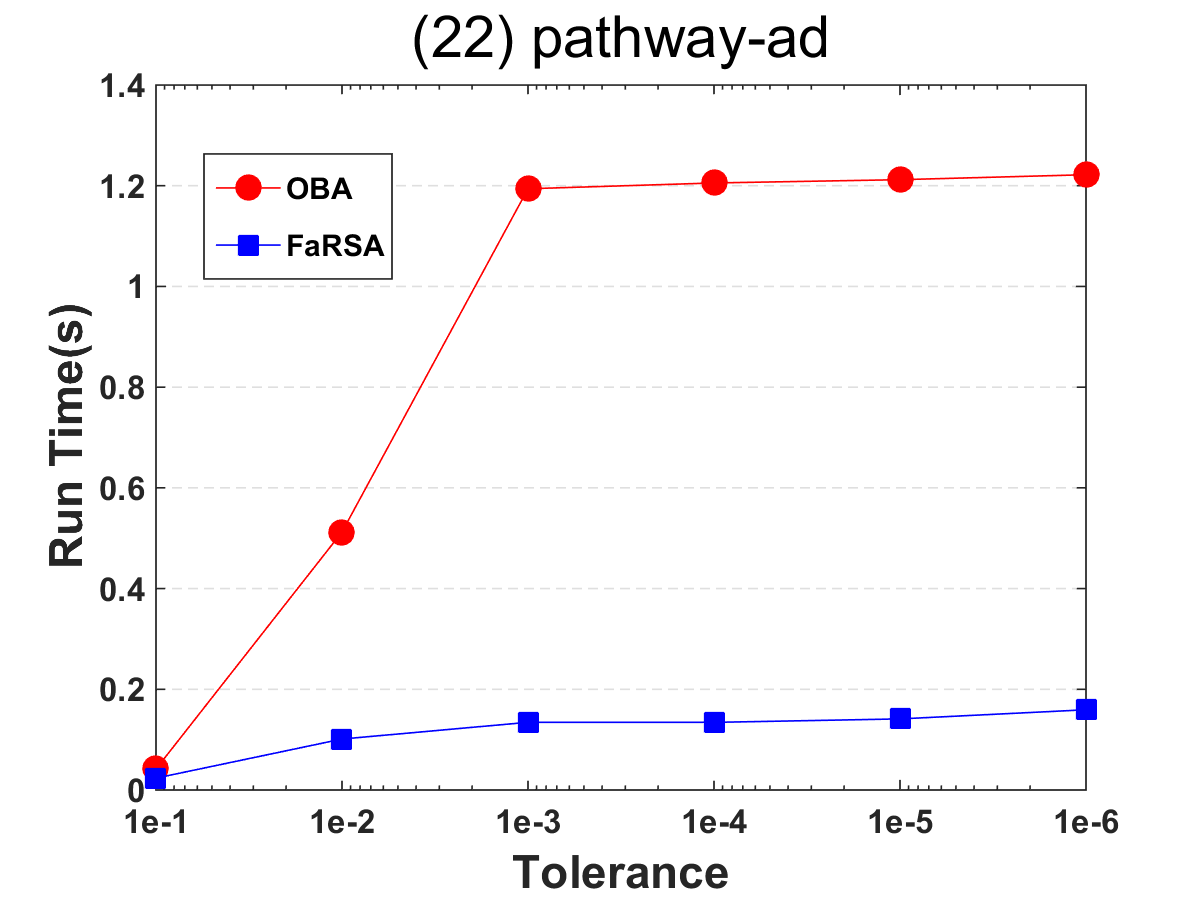}
\includegraphics[width=1.2in]{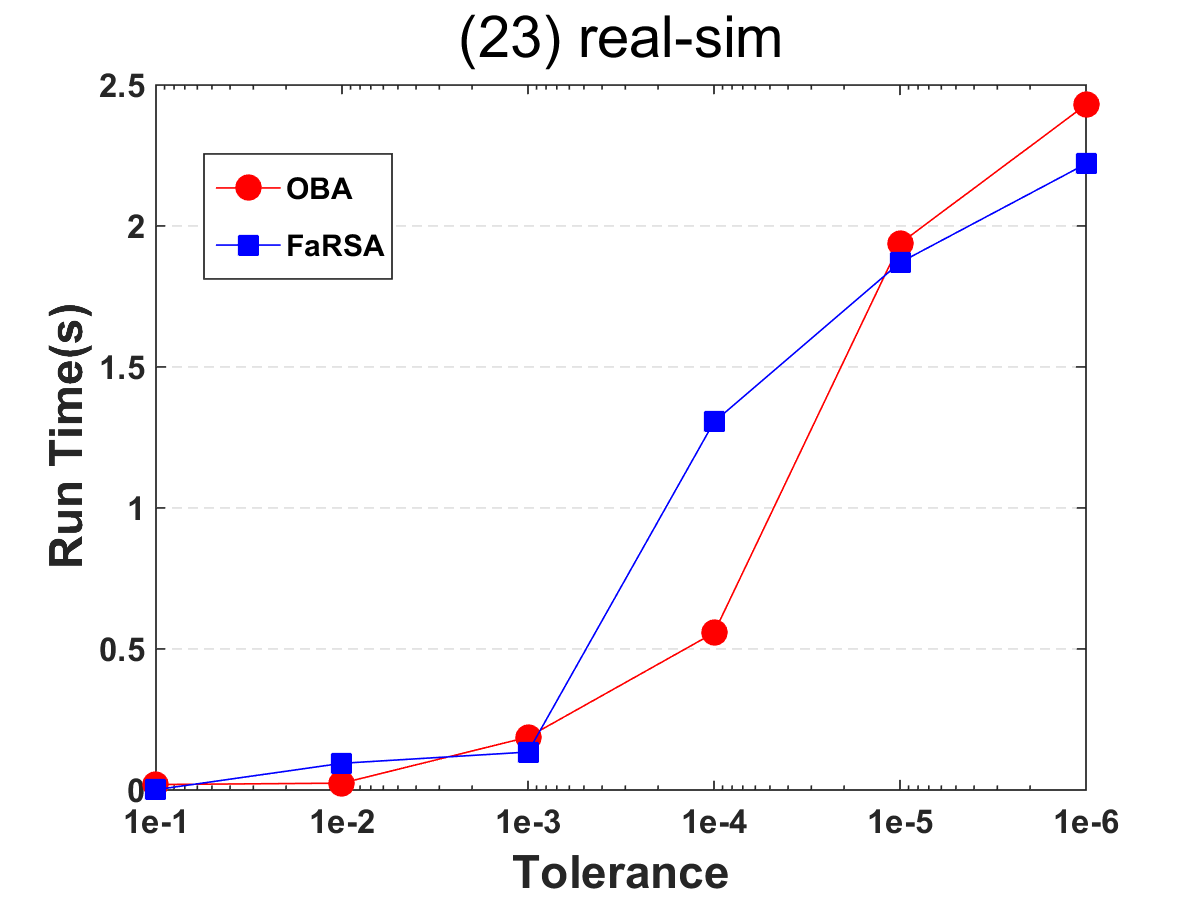}
\includegraphics[width=1.2in]{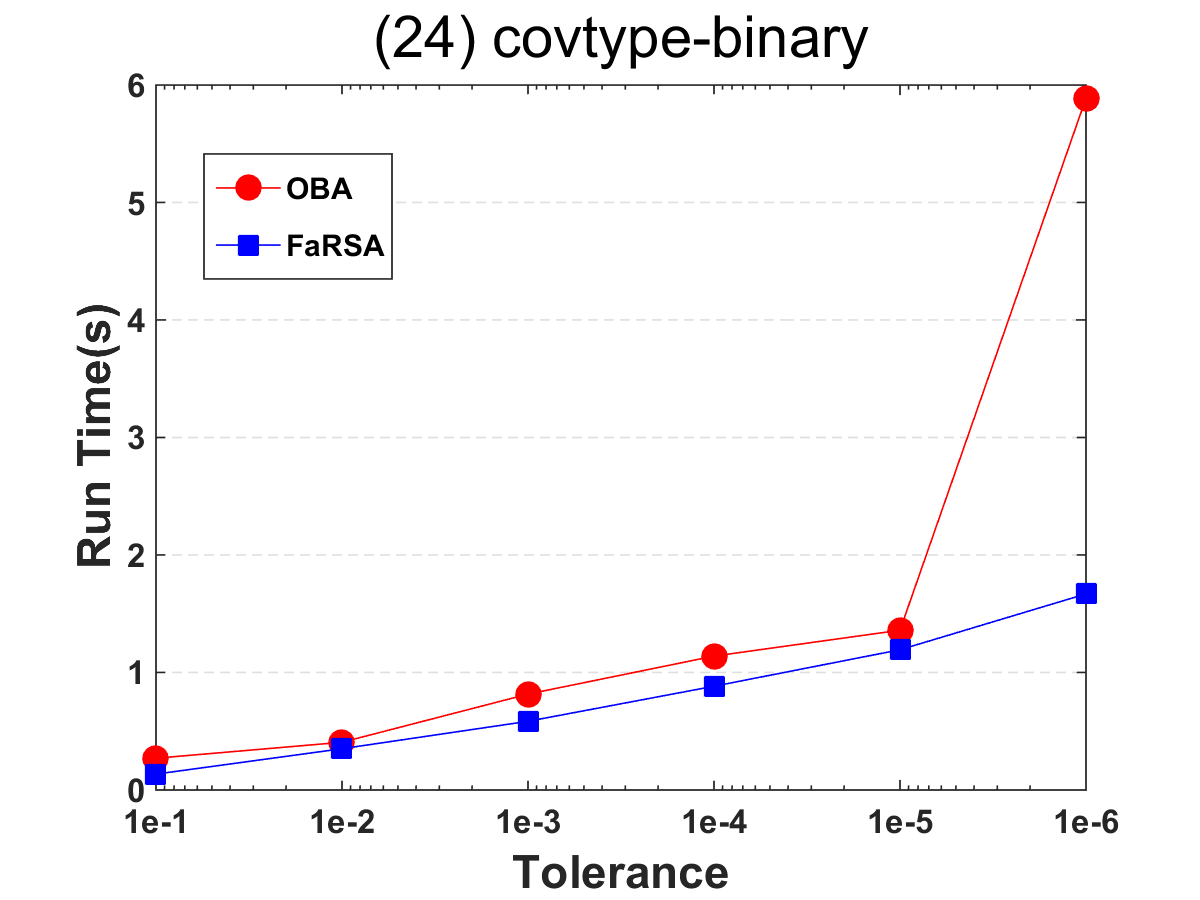}
\includegraphics[width=1.2in]{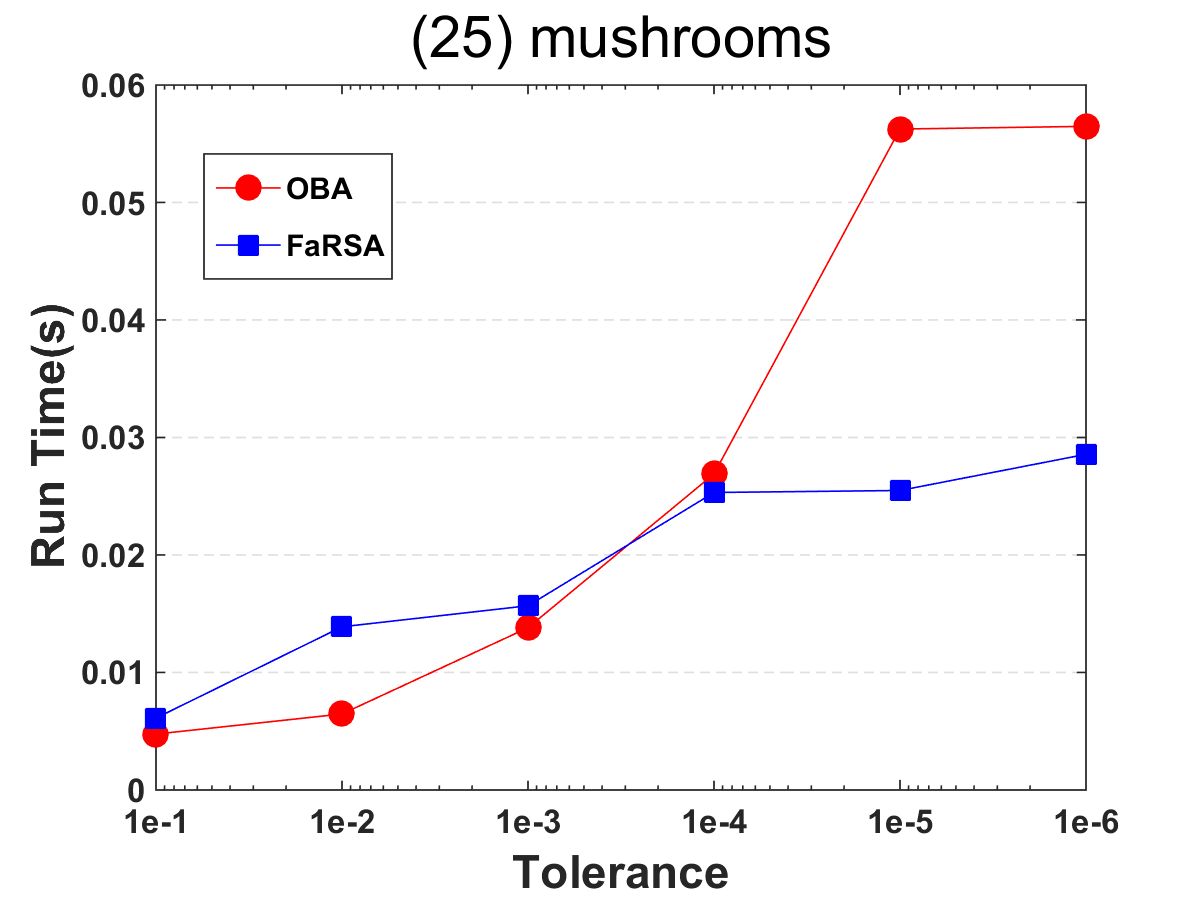}
\includegraphics[width=1.2in]{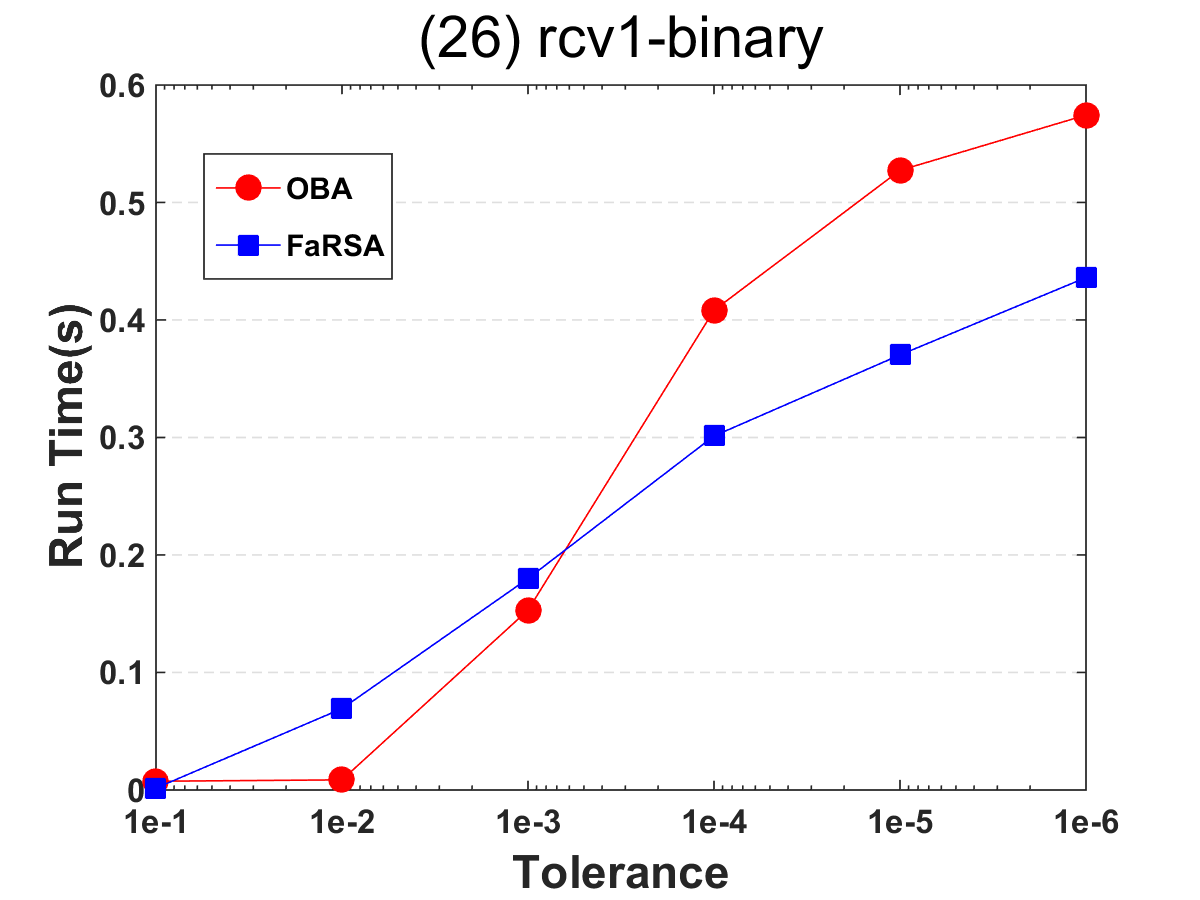}
\includegraphics[width=1.2in]{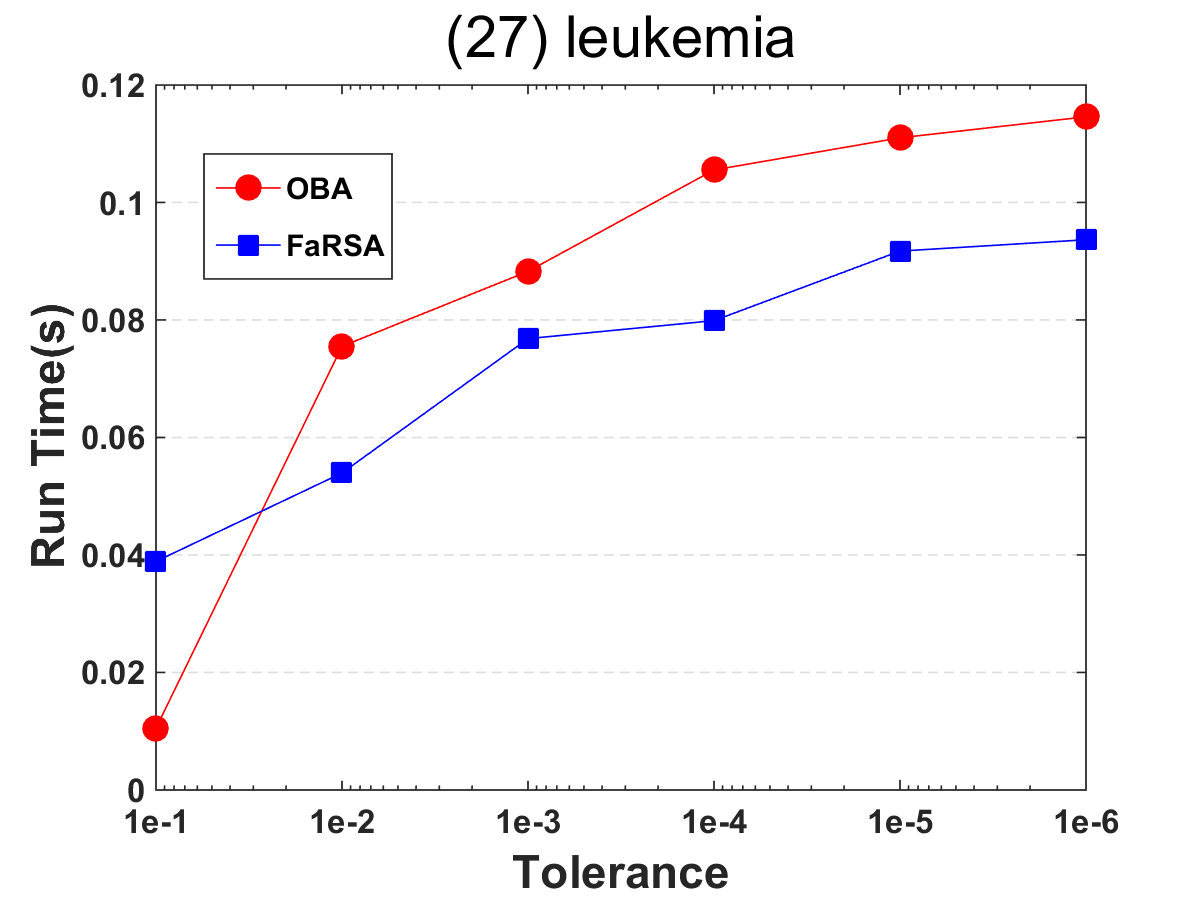}
\includegraphics[width=1.2in]{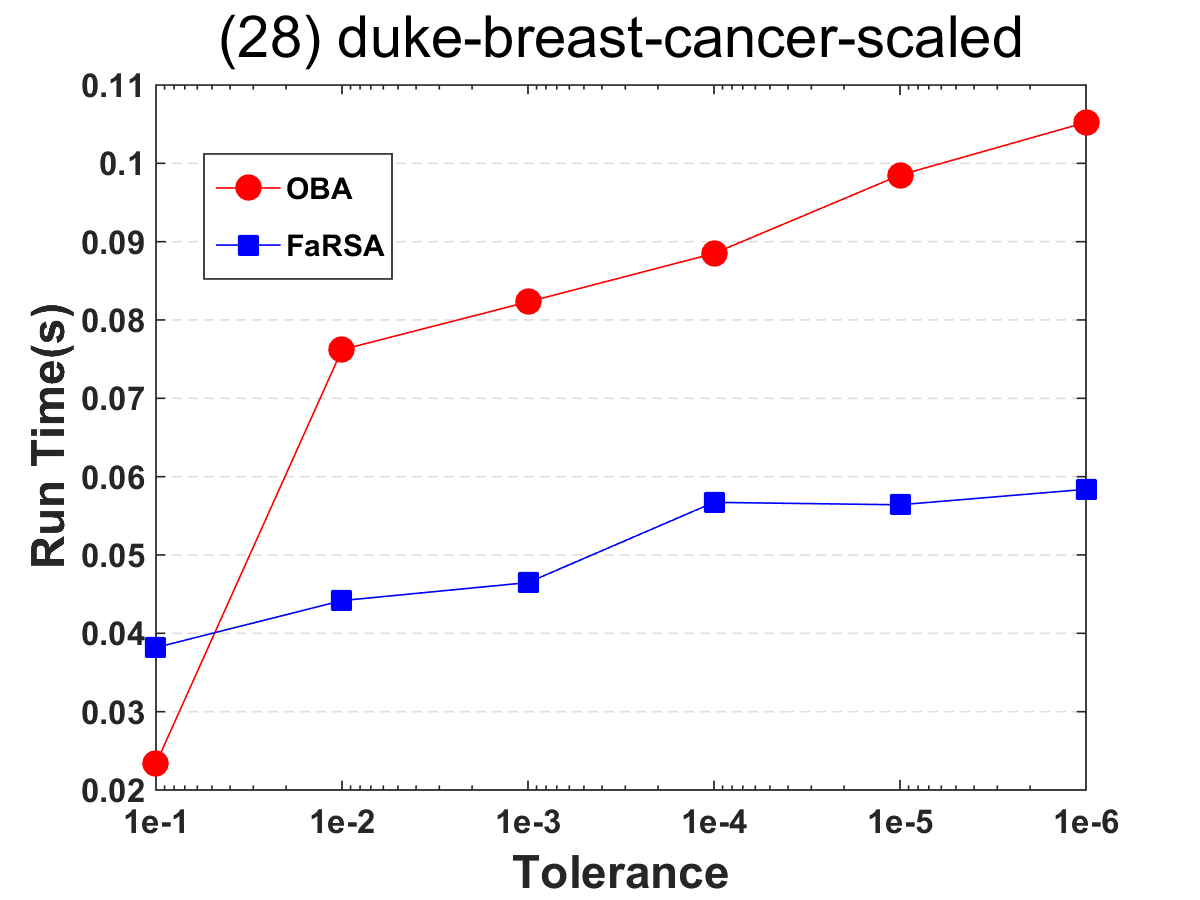}
\includegraphics[width=1.2in]{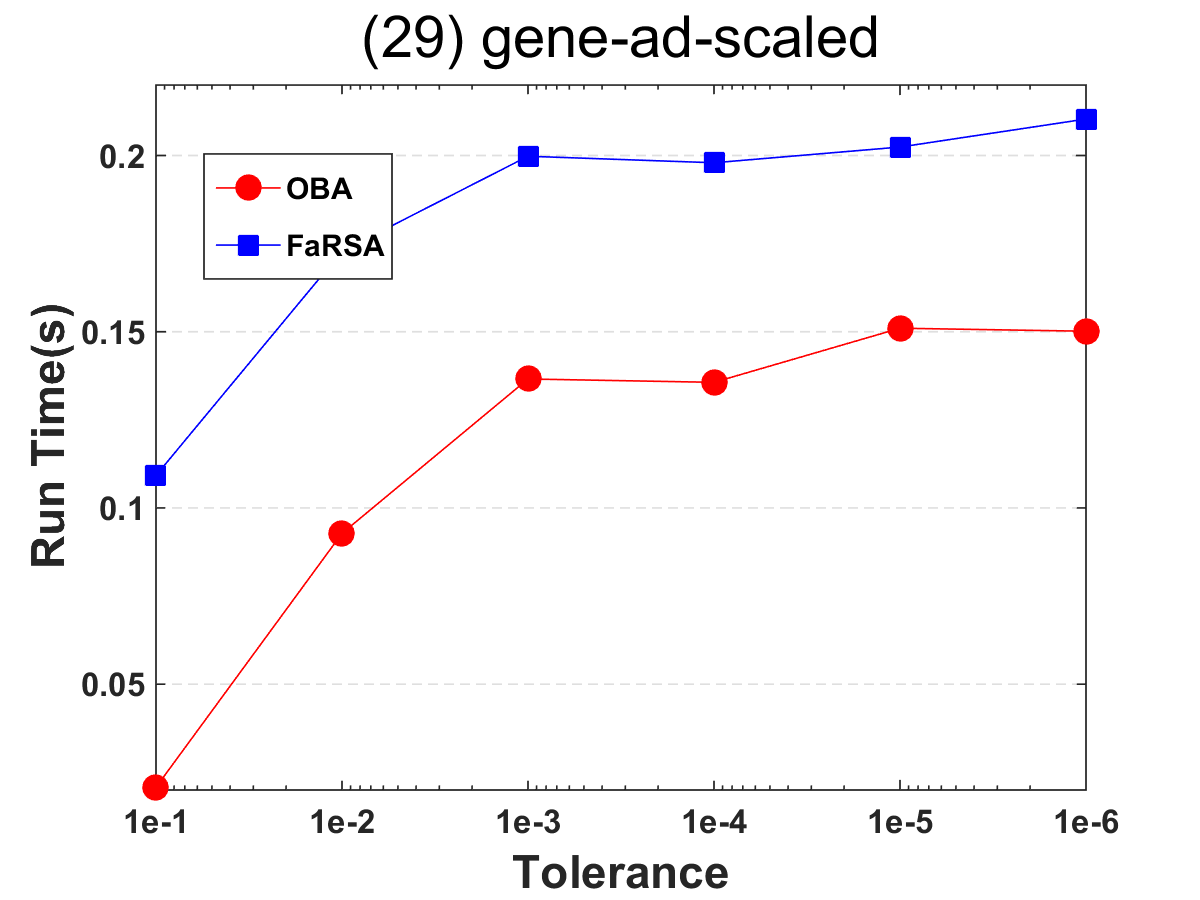}
\includegraphics[width=1.2in]{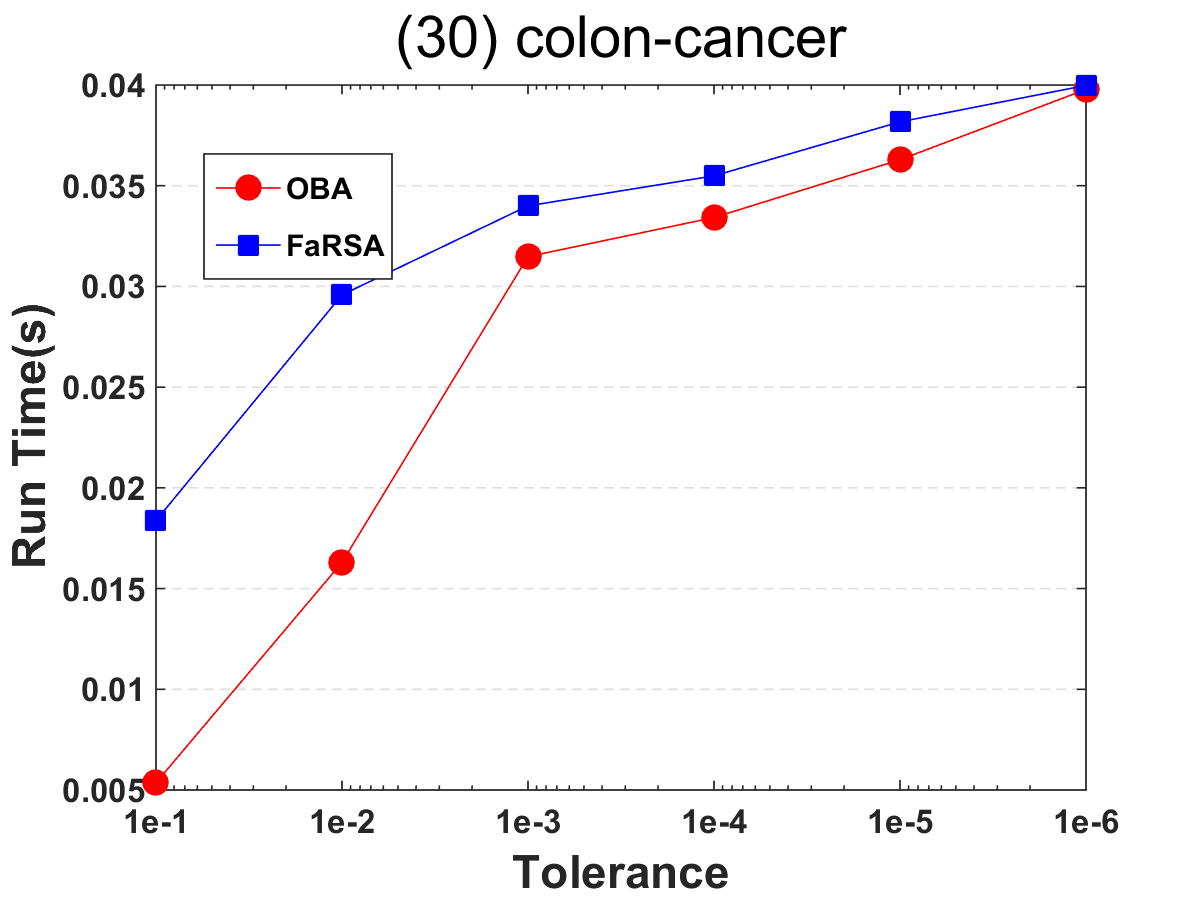}
\includegraphics[width=1.2in]{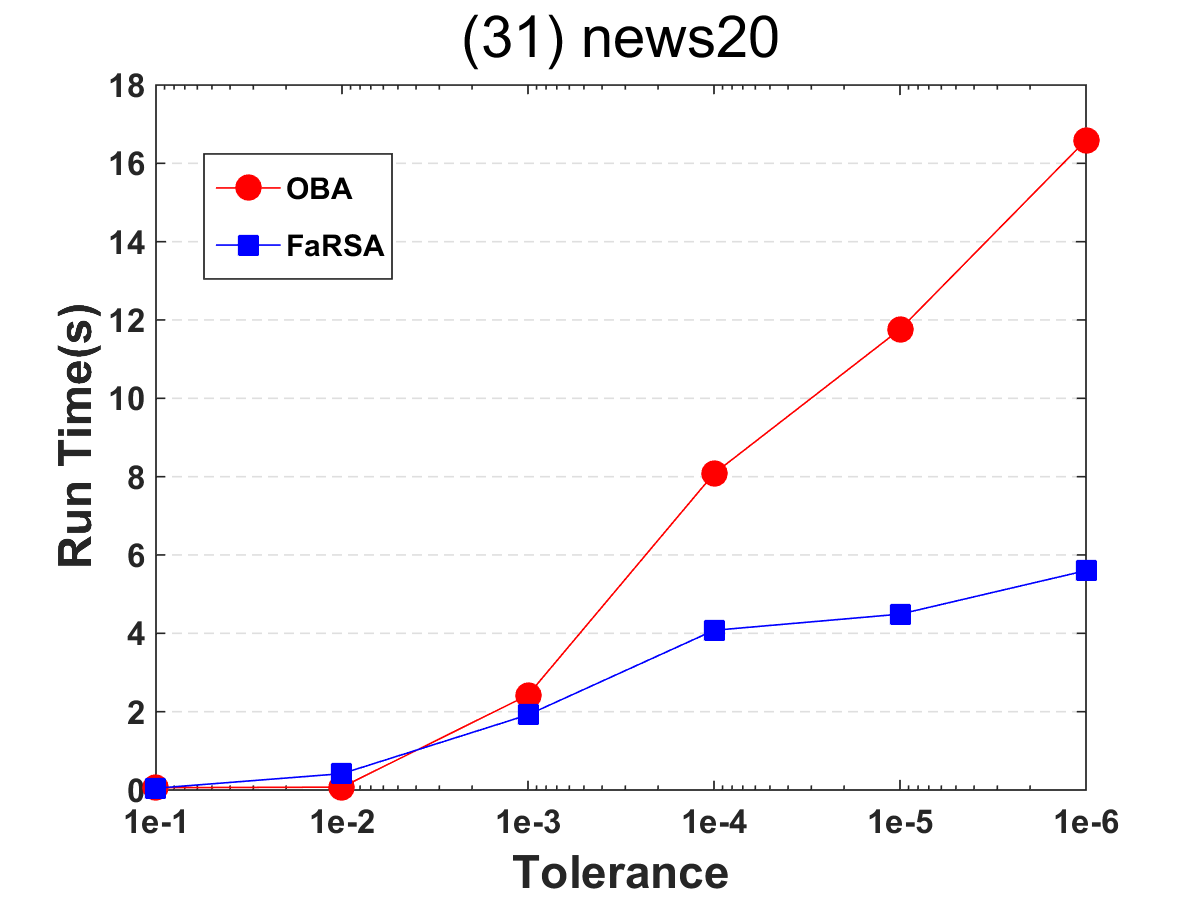}
\caption{CPU time comparison between \algacro{} and OBA for various stopping tolerances on the set of problems with scaled data.}
\label{figure:cputimescaledcmp}
\end{figure}

\begin {figure}[ht]
\centering
\includegraphics[width=1.2in]{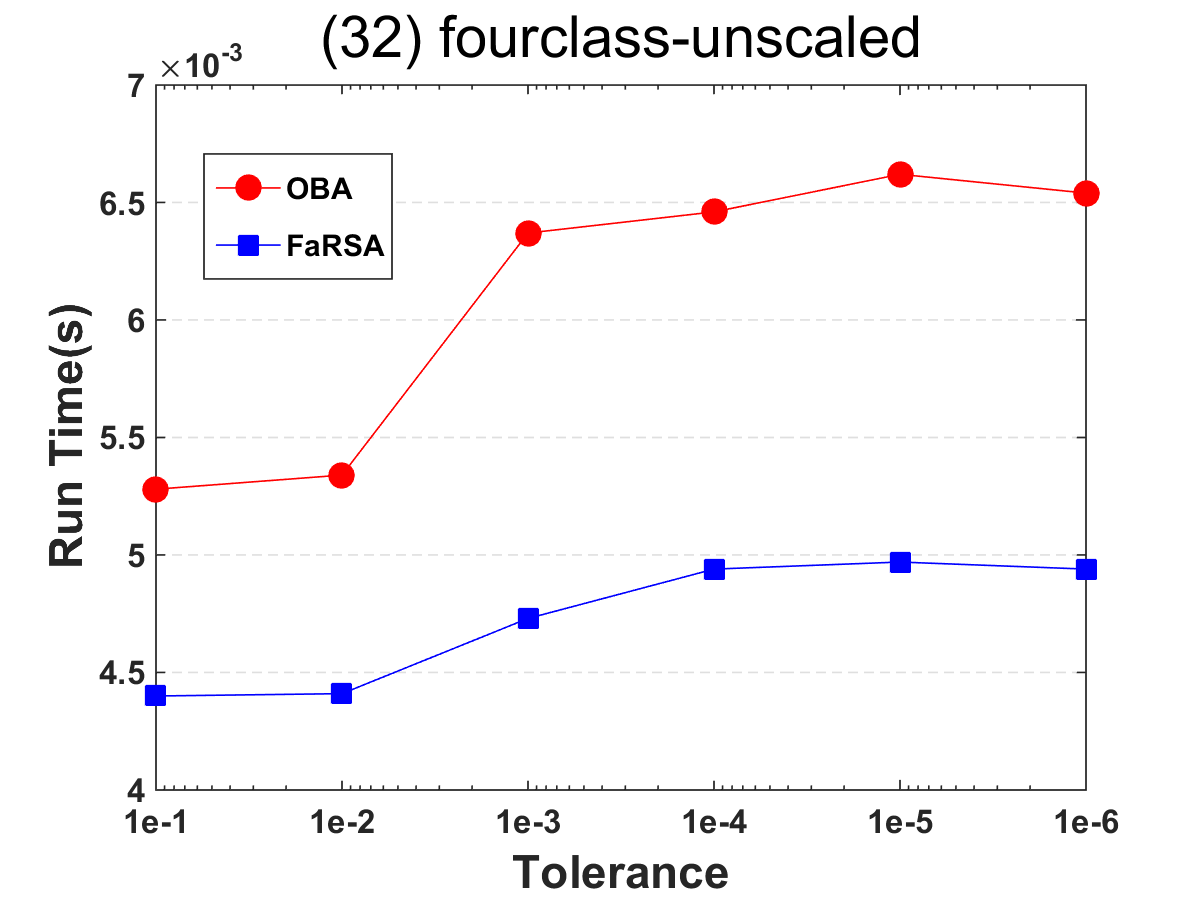}
\includegraphics[width=1.2in]{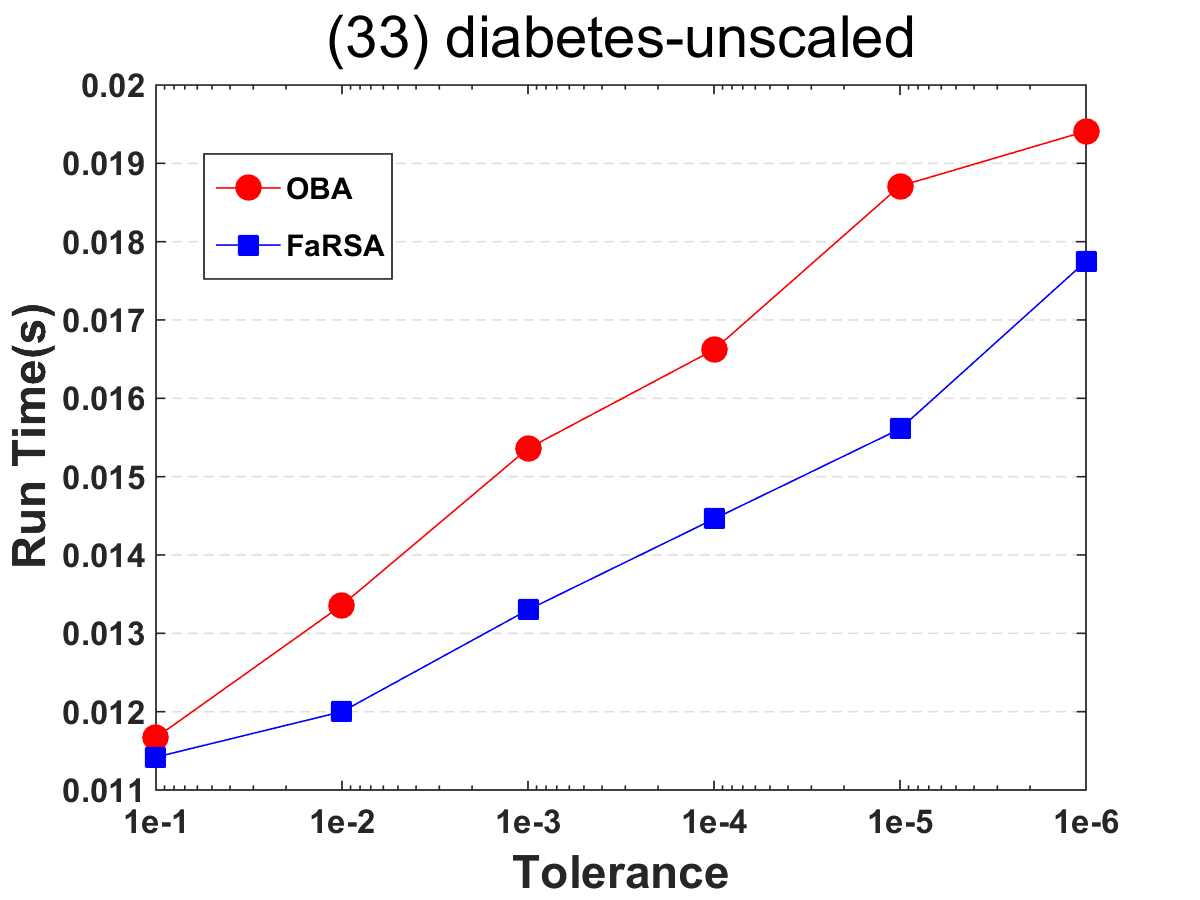}
\includegraphics[width=1.2in]{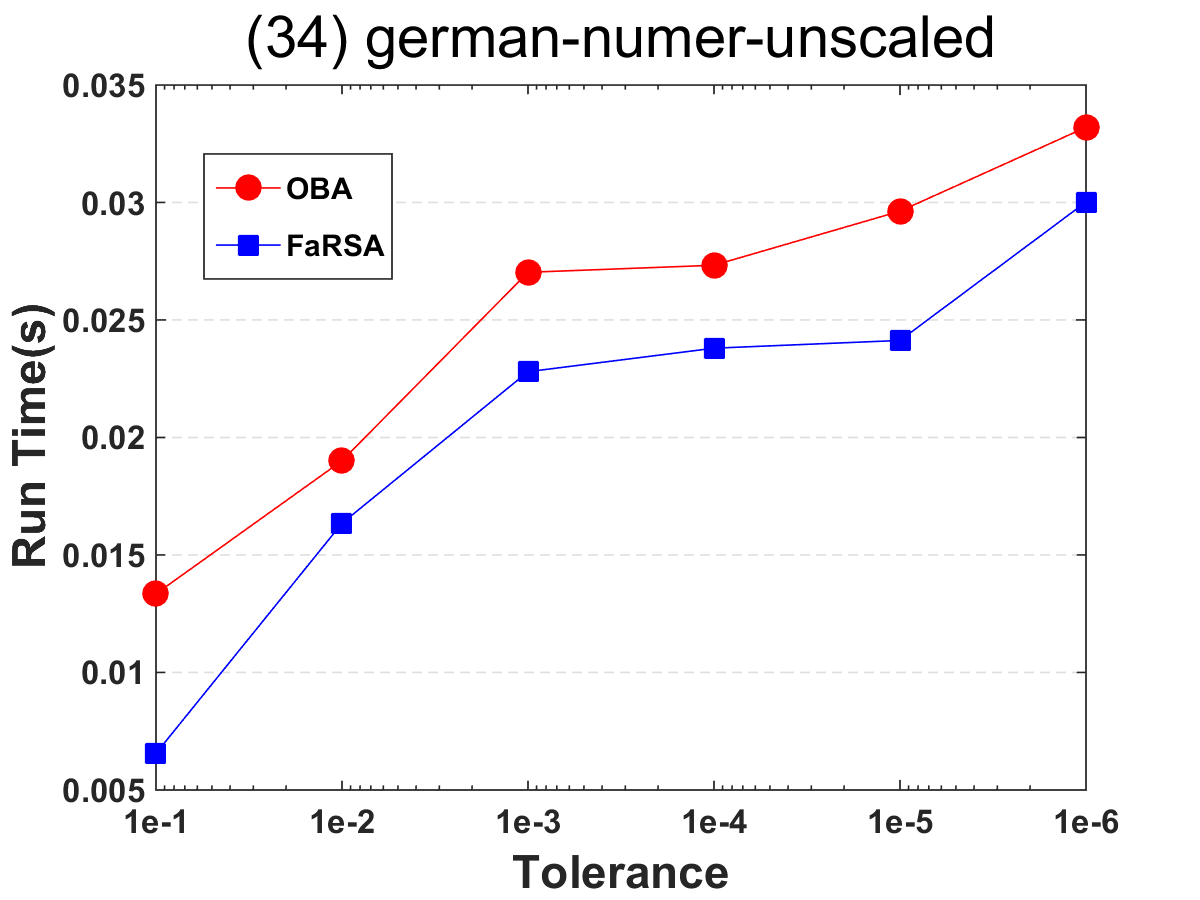}
\includegraphics[width=1.2in]{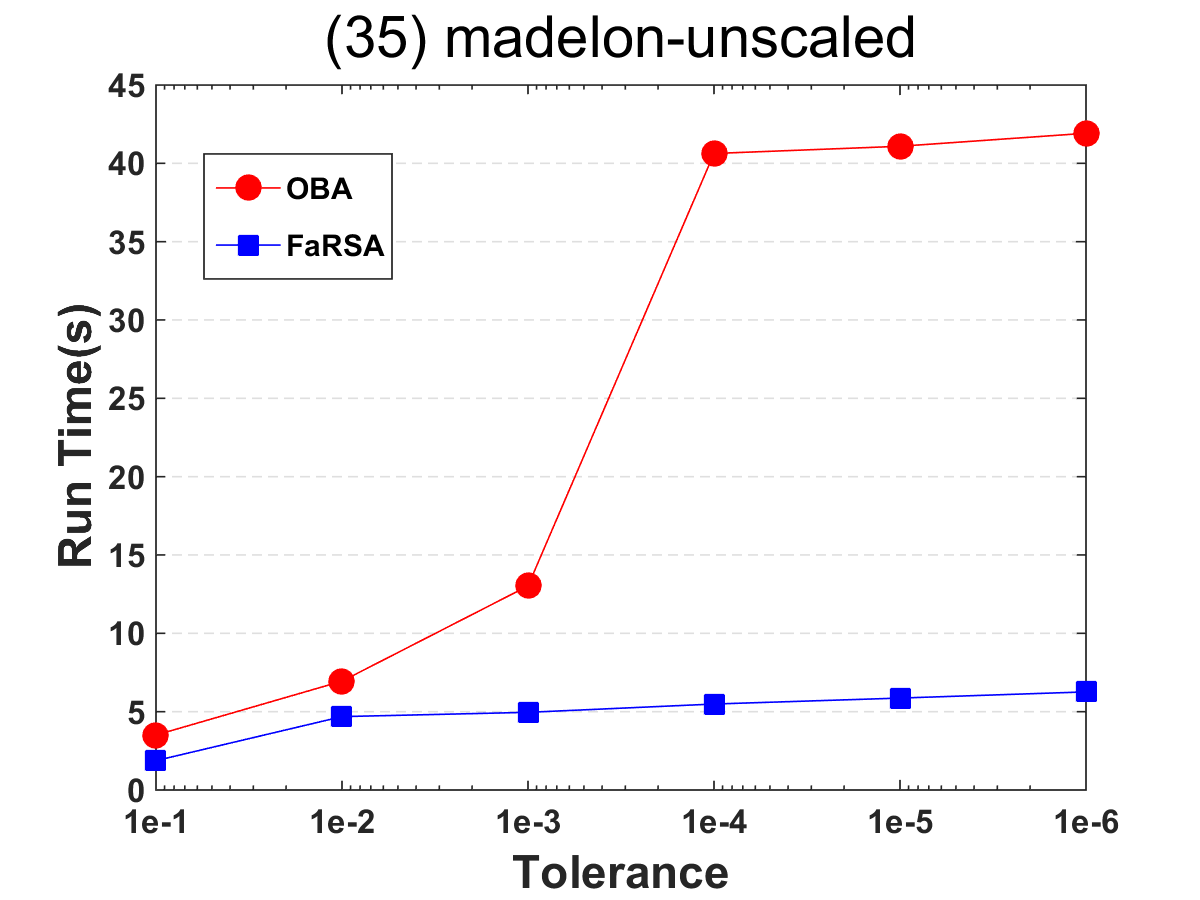}
\includegraphics[width=1.2in]{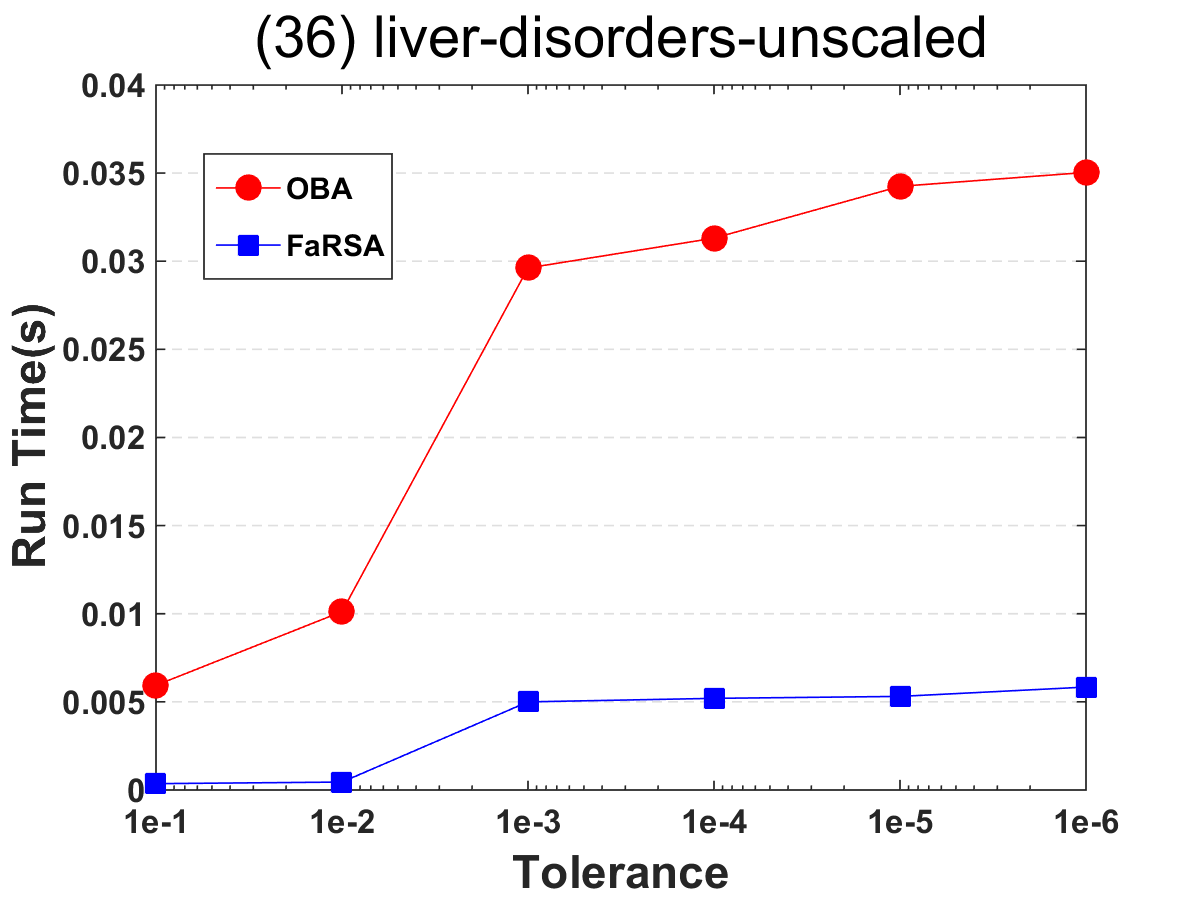}
\includegraphics[width=1.2in]{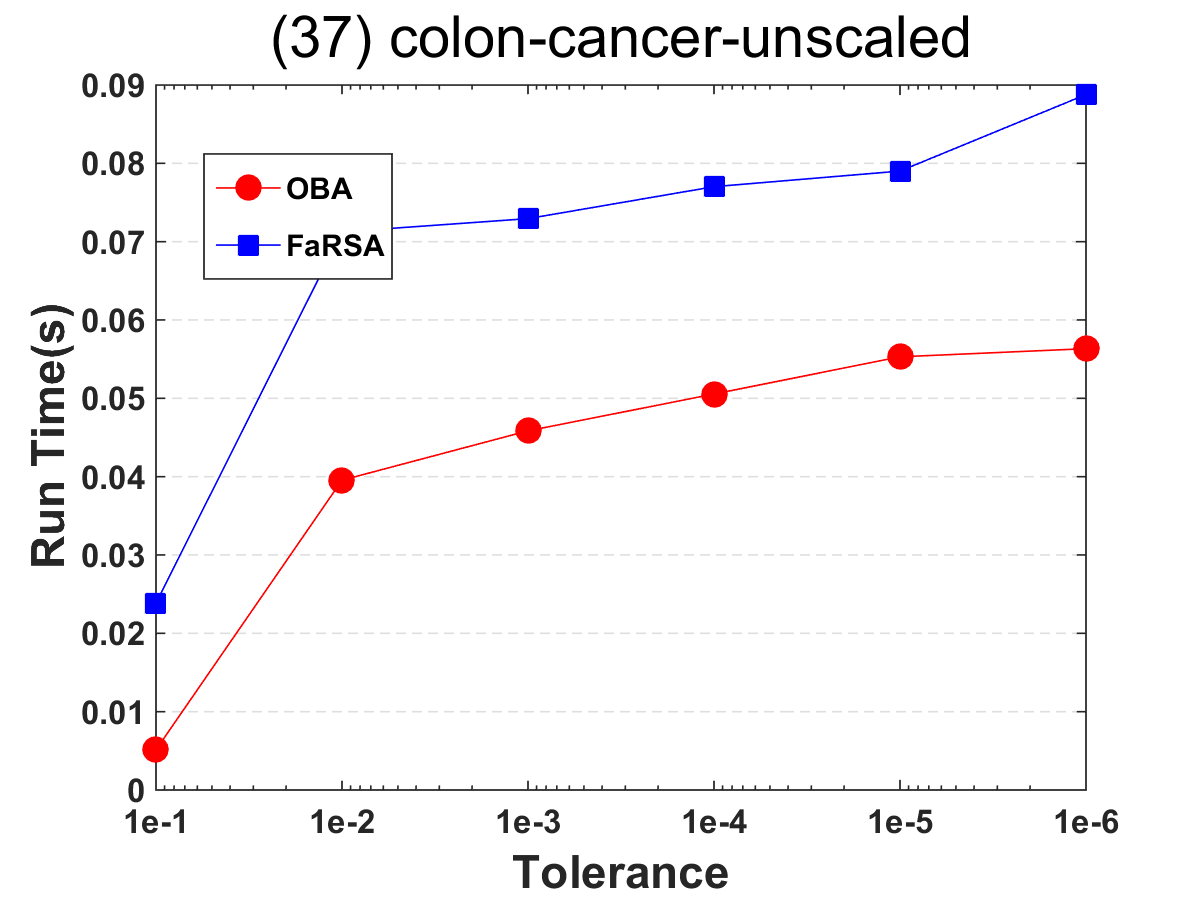}
\includegraphics[width=1.2in]{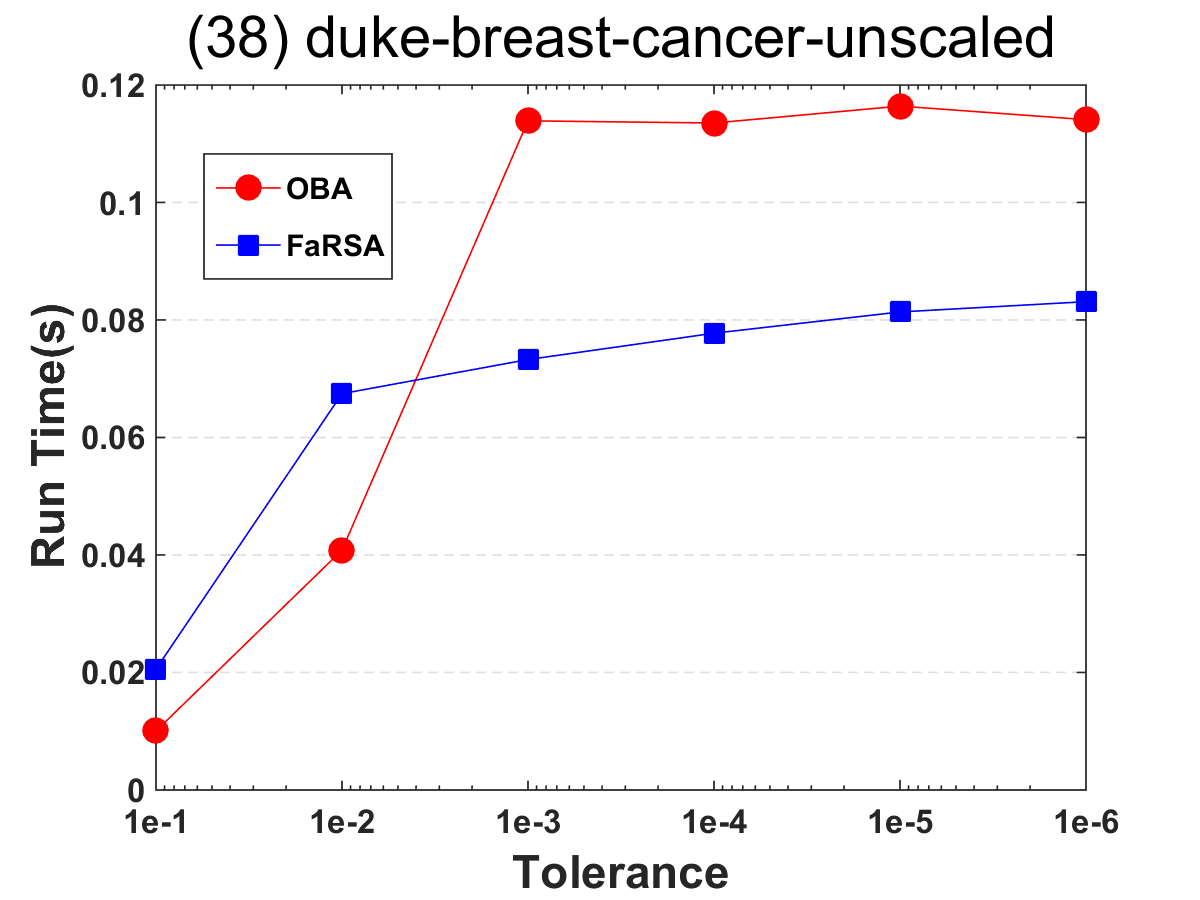}
\includegraphics[width=1.2in]{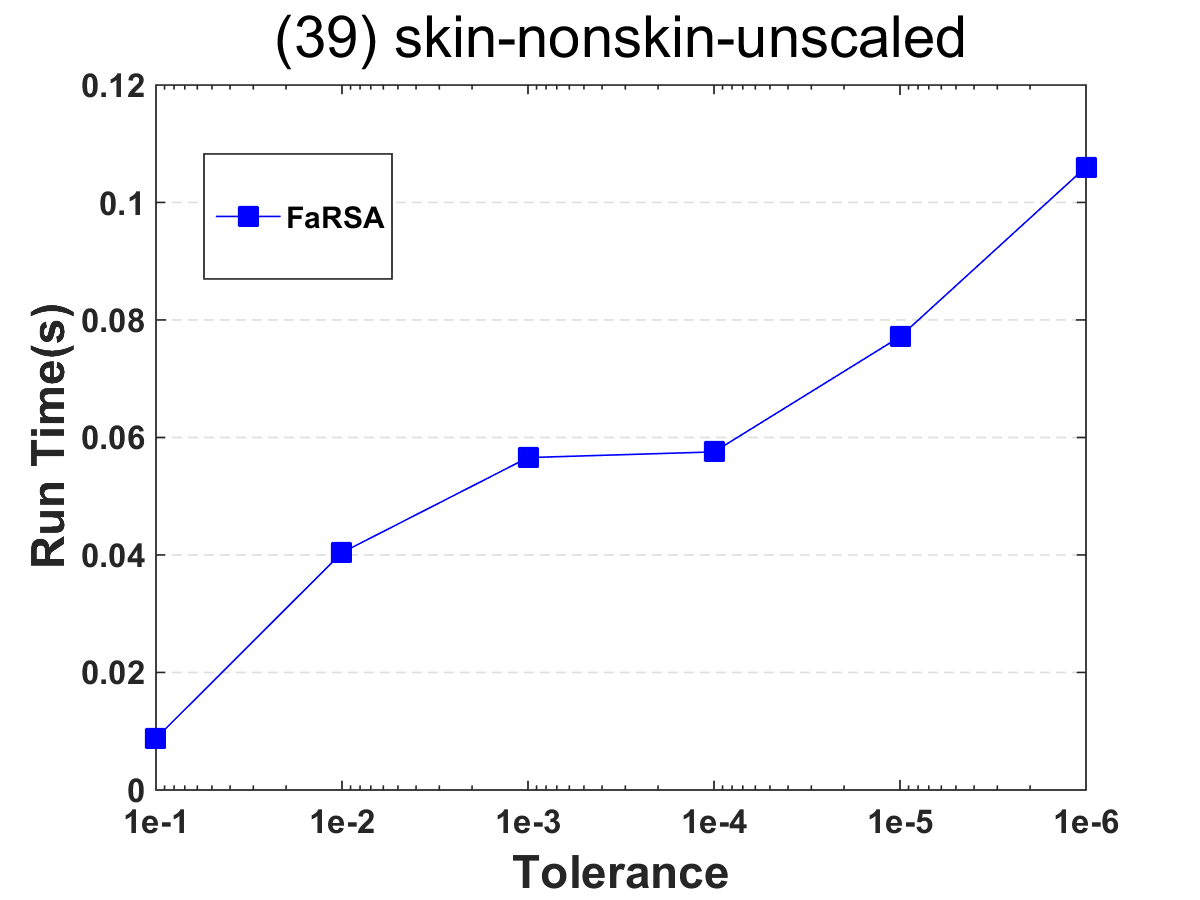}
\includegraphics[width=1.2in]{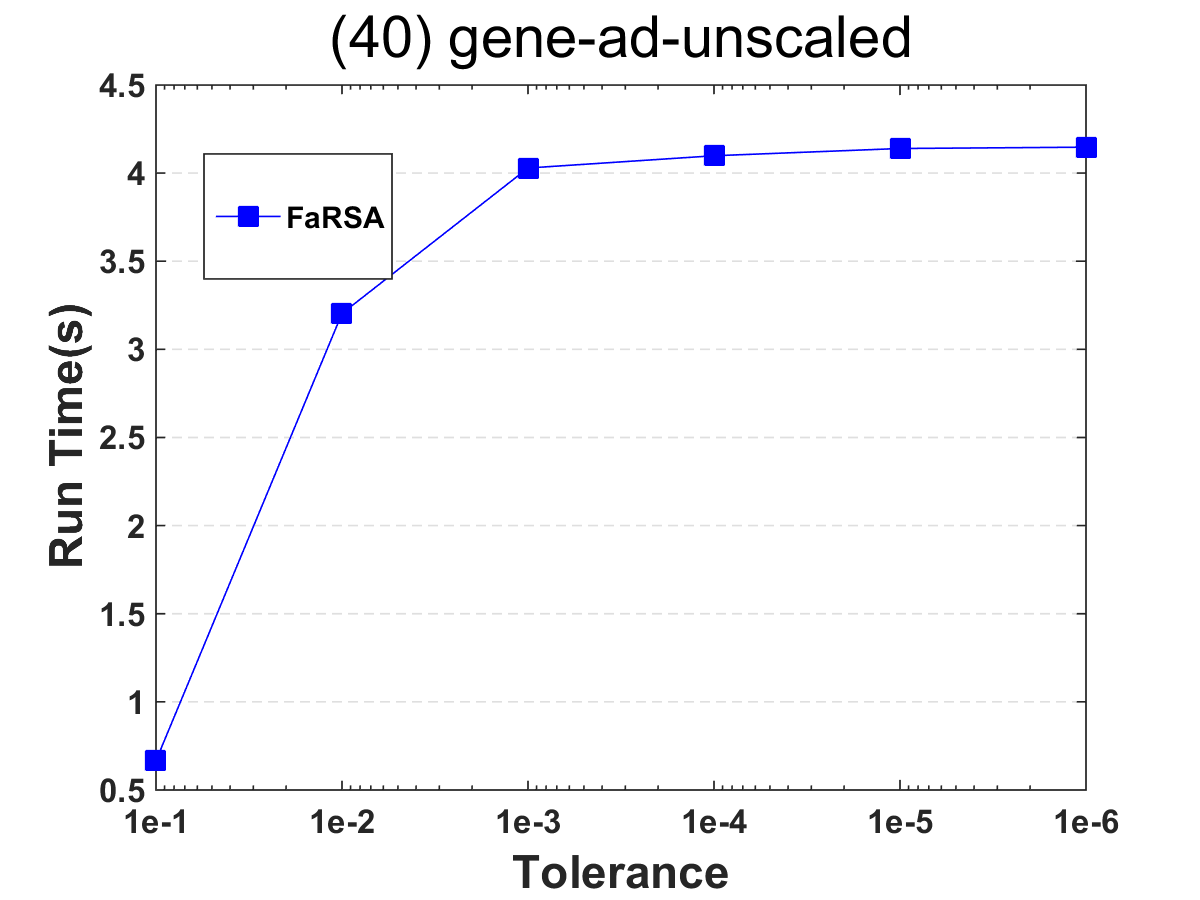}
\includegraphics[width=1.2in]{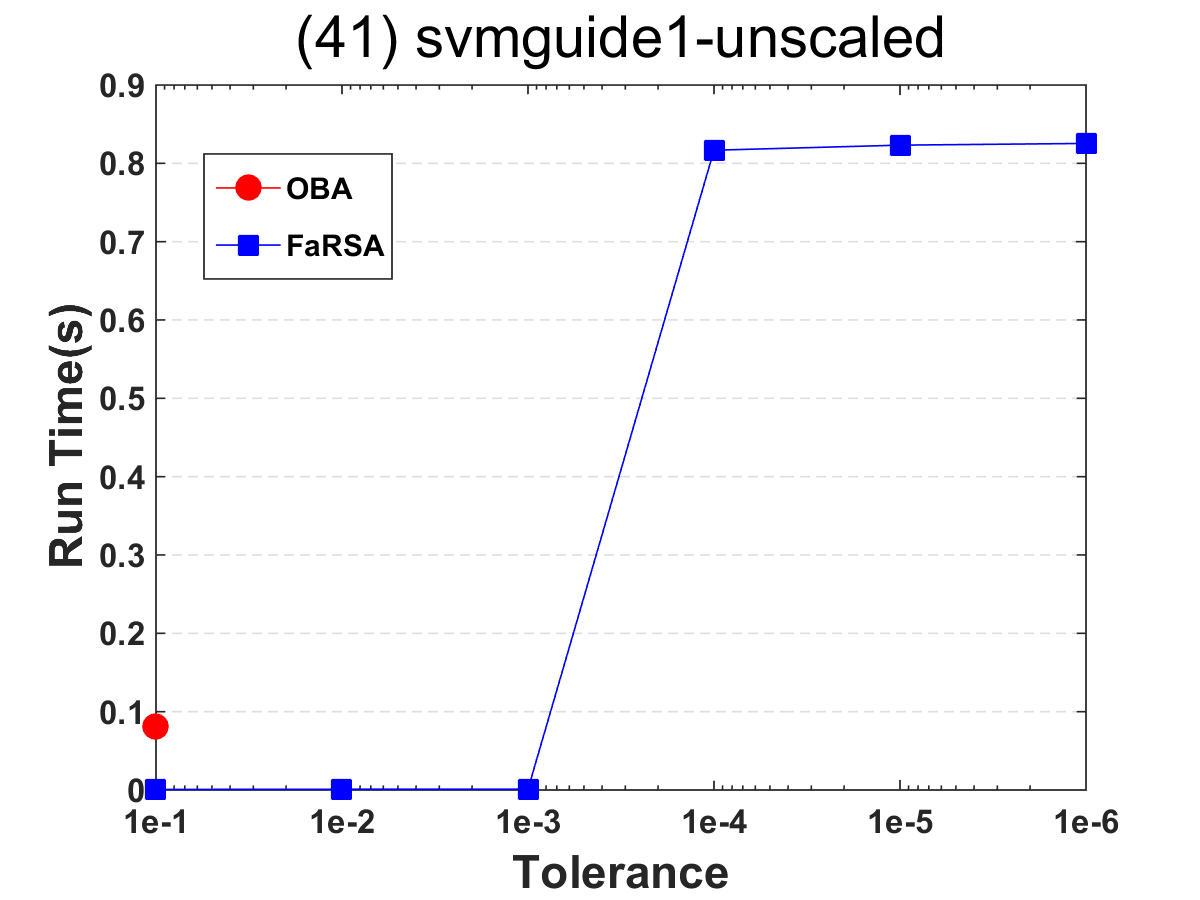}
\includegraphics[width=1.2in]{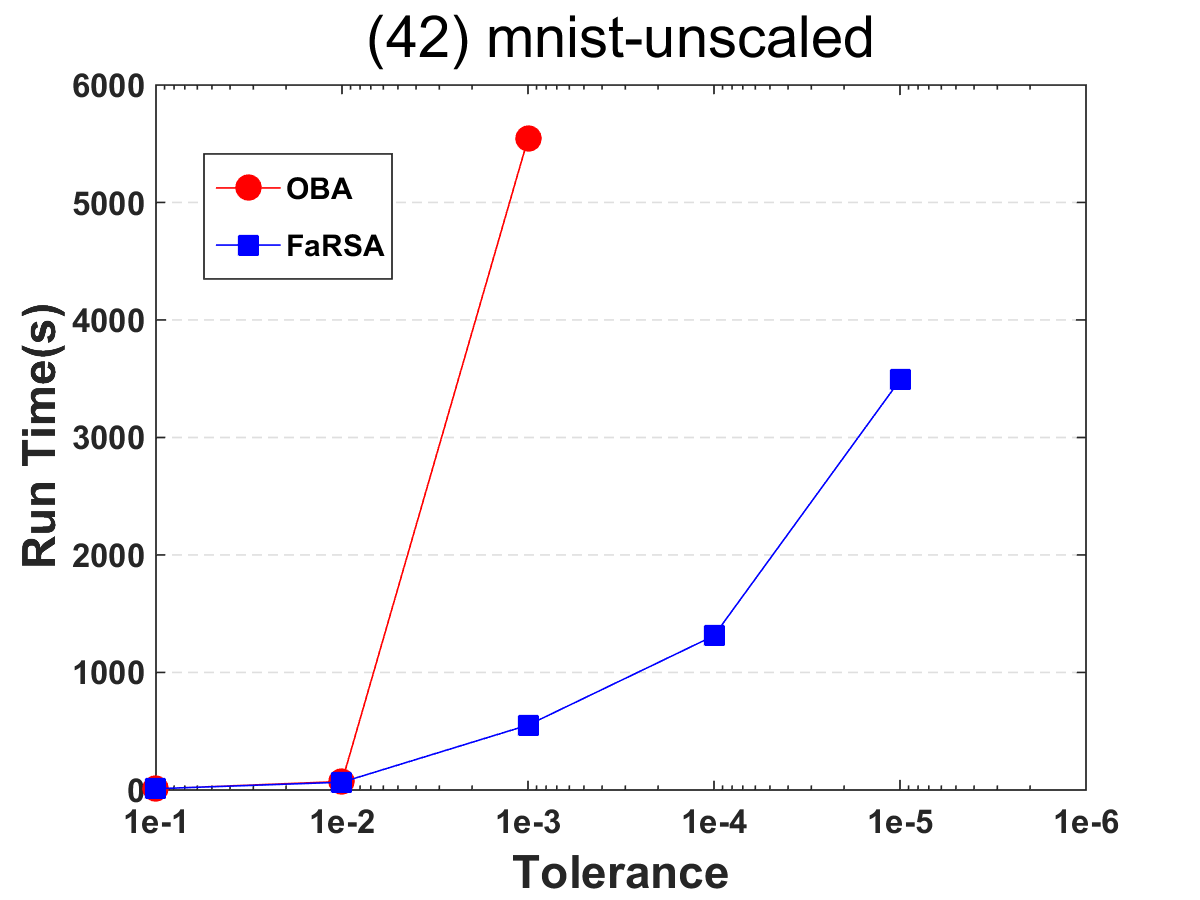}
\includegraphics[width=1.2in]{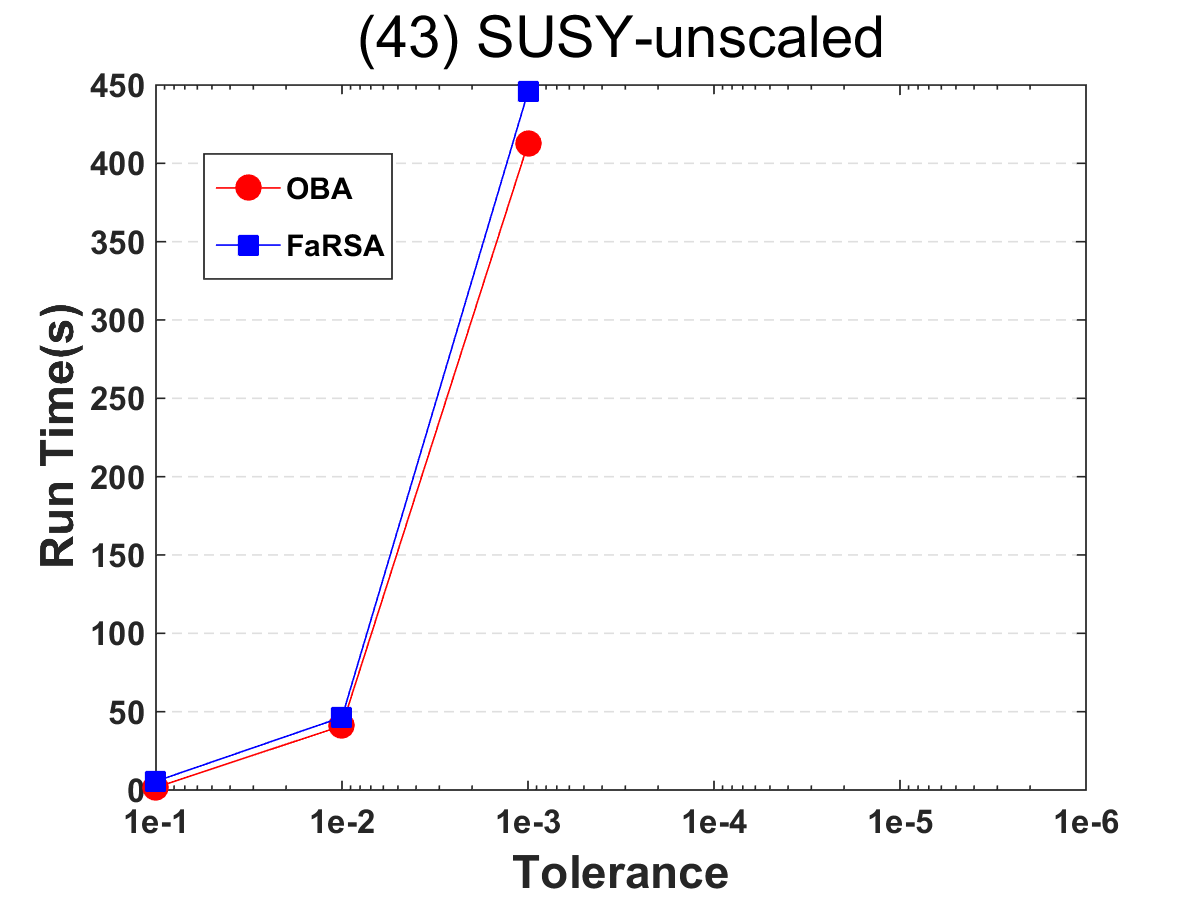}
\caption{CPU time comparison between \algacro{} and OBA for various stopping tolerances on the set of problems with unscaled data.}
\label{figure:cputimeunscaledcmp}
\end{figure}

The previous tables show that \algacro{} efficiently and reliably obtains solutions that satisfy the stopping tolerance value of $\epsilon = 10^{-6}$.  In practice, one sometimes only requests a low accuracy solution, often motivated by problems that may arise due to overfitting. To explore the performance of \algacro{} for various stopping tolerance levels, we created the plots in Figures~\ref{figure:cputimescaledcmp} and~\ref{figure:cputimeunscaledcmp}. Each plot shows the run time ($y$-axis) required to achieve the desired optimality accuracy ($x$-axis) for the stated problem.  These figures show that the superior performance of \algacro{} previously displayed for the stopping tolerance $10^{-6}$ also generally holds for larger stopping tolerances.

%\newpage
%*********
% Section
%*********
\section{Conclusions}
%-----------------------------
We presented a new reduced-space algorithm, \algacro, for minimizing an $\ell_1$-norm regularized convex function. The method uses an adaptive condition to determine when the current reduced-space should be updated, which is itself based on measures of optimality in the current reduced space and its complement.  Global convergence was established for our method, while numerical experiments on $\ell_1$-norm regularized logistic problems exhibited its practical performance.  In particular, the experiments showed that \algacro{} was generally superior to a recently proposed reduced-space orthant-based algorithm called OBA, regardless of the solution accuracy requested. Since OBA was shown in~\cite{keskar2015second} to be better than the state-of-the-art solver used in LIBLINEAR when the second derivative matrices were not diagonally dominant, we expect that \algacro{} will serve as a valuable data analysis tool.  OBA and our preliminary implementation of \algacro{} will often be outperformed by LIBLINEAR when the second derivative matrices are diagonally dominant.  However, \algacro{} was designed with great flexibility in how the subproblem solutions are obtained.  Although our preliminary implementation invoked linear-CG as the subproblem solver, our framework also allows for coordinate-descent based algorithms to be used, such as those used in LIBLINEAR.  We expect to provide such options as well as include features that control the subproblem size in a future release of our solver.  We believe that once these enhancements have been made, \algacro{} will be competitive with LIBLINEAR on all classes of problems, and superior when the second derivative matrices are not diagonally dominant.

\section*{Acknowledgments}
%------------------------------------------
We thank Nitish Keskar, Jorge Nocedal, Figen \"{O}ztoprak, and Andreas W\"{a}chter for providing the \Matlab{} code of their OBA algorithm, and for several discussions on their numerical experience with OBA. We also thank Qingsong Zhu for providing us the datasets gene-ad and pathway-ad used in Section~\ref{subsec:data}.

%**************
% Bibliography
%**************
\bibliographystyle{plain}
\bibliography{fplusl1}

\begin{thebibliography}{10}

\bibitem{andrew2007scalable}
Galen Andrew and Jianfeng Gao.
\newblock Scalable training of $l_1$-regularized log-linear models.
\newblock In {\em Proceedings of the 24th international conference on Machine
  learning}, pages 33--40. ACM, 2007.

\bibitem{beck2009fast}
Amir Beck and Marc Teboulle.
\newblock A fast iterative shrinkage-thresholding algorithm for linear inverse
  problems.
\newblock {\em SIAM journal on imaging sciences}, 2(1):183--202, 2009.

\bibitem{byrd2012family}
Richard~H Byrd, Gillian~M Chin, Jorge Nocedal, and Figen Oztoprak.
\newblock A family of second-order methods for convex l1-regularized
  optimization.
\newblock {\em Unpublished: Optimization Center: Northwestern University, Tech
  Report}, 2012.

\bibitem{byrd2013inexact}
Richard~H Byrd, Jorge Nocedal, and Figen Oztoprak.
\newblock An inexact successive quadratic approximation method for convex l-1
  regularized optimization.
\newblock {\em arXiv preprint arXiv:1309.3529}, 2013.

\bibitem{Dost97}
Zdenek Dost\'al.
\newblock Box constrained quadratic programming with proportioning and
  projections.
\newblock {\em SIAM Journal on Optimization}, 7(3):871--887, 1997.

\bibitem{Dost03}
Zdenek Dost{\'a}l.
\newblock A proportioning based algorithm with rate of convergence for bound
  constrained quadratic programming.
\newblock {\em Numerical Algorithms}, 34(2):293--302, 2003.

\bibitem{DostScho05}
Zdenek Dost\'al and Joachim Schoberl.
\newblock Minimizing quadratic functions subject to bound constraints with the
  rate of convergence and finite termination.
\newblock {\em Computational Optimization and Applications}, 30(1):23--43,
  2005.

\bibitem{hsieh2011sparse}
Cho-Jui Hsieh, Inderjit~S Dhillon, Pradeep~K Ravikumar, and M{\'a}ty{\'a}s~A
  Sustik.
\newblock Sparse inverse covariance matrix estimation using quadratic
  approximation.
\newblock In {\em Advances in Neural Information Processing Systems}, pages
  2330--2338, 2011.

\bibitem{keskar2015second}
Nitish~Shirish Keskar, Jorge Nocedal, Figen Oztoprak, and Andreas Waechter.
\newblock A second-order method for convex $\ell_1$-regularized optimization
  with active set prediction.
\newblock {\em arXiv preprint arXiv:1505.04315}, 2015.

\bibitem{lee2012proximal}
Jason Lee, Yuekai Sun, and Michael Saunders.
\newblock Proximal newton-type methods for convex optimization.
\newblock In {\em Advances in Neural Information Processing Systems}, pages
  836--844, 2012.

\bibitem{scheinberg2013practical}
Katya Scheinberg and Xiaocheng Tang.
\newblock Practical inexact proximal quasi-newton method with global complexity
  analysis.
\newblock {\em arXiv preprint arXiv:1311.6547}, 2013.

\bibitem{hassanLovesPoetry}
Hassan~Mohy ud~Din and Daniel~P. Robinson.
\newblock A solver for nonconvex bound-constrained quadratic optimization.
\newblock {\em SIAM Journal on Optimization}, 25(4):2385--2407, 2015.

\bibitem{wright2009sparse}
Stephen~J Wright, Robert~D Nowak, and M{\'a}rio~AT Figueiredo.
\newblock Sparse reconstruction by separable approximation.
\newblock {\em Signal Processing, IEEE Transactions on}, 57(7):2479--2493,
  2009.

\bibitem{yuan2012improved}
Guo-Xun Yuan, Chia-Hua Ho, and Chih-Jen Lin.
\newblock An improved glmnet for l1-regularized logistic regression.
\newblock {\em The Journal of Machine Learning Research}, 13(1):1999--2030,
  2012.

\bibitem{zhu2015pathway}
Qingsong Zhu, Evgeny Izumchenko, Alexander~M Aliper, Evgeny Makarev, Keren Paz,
  Anton~A Buzdin, Alex~A Zhavoronkov, and David Sidransky.
\newblock Pathway activation strength is a novel independent prognostic
  biomarker for cetuximab sensitivity in colorectal cancer patients.
\newblock {\em Human Genome Variation}, 2, 2015.

\end{thebibliography}

%**********
% Appendix
%**********
\appendix

\section{A Relationship between \algacro{} and ISTA}
The step in Line~\ref{line:ls-beta} of Algorithm~\ref{alg:main.x} may be interpreted as a \emph{reduced} ISTA~\cite{beck2009fast} step.  The next lemma makes this relationship precise.
\begin{lemma} \label{lem:equi-ista}
  For any $k$, let $s_k$ be the \emph{full} ISTA step defined by
  \begin{equation*}
    \begin{aligned}
      s_k := &\ {\rm shrink}\big(x_k - \Grad f(x_k)\big) - x_k,\ \text{where} \\
      &\ [{\rm shrink}\big(x_k - \Grad f(x_k)\big) - x_k]_i :=
      \begin{cases}
        -[\Grad f(x_k)]_i + \lambda & \text{if $[x_k - \Grad f(x_k)]_i < -\lambda$,} \\
        -[x_k]_i & \text{if $[x_k - \Grad f(x_k)]_i \in [-\lambda,\lambda]$,} \\
        -[\Grad f(x_k)]_i - \lambda & \text{if $[x_k - \Grad f(x_k)]_i > \lambda$.}
      \end{cases}
    \end{aligned}
  \end{equation*}
  Then, $s_k = -(\beta(x_k) + \phi(x_k))$.
\end{lemma}
\begin{proof}
  Recall the definitions of the components of $\beta(x_k)$ and $\phi(x_k)$, which may be rewritten in a slightly more convenient form as follows:
  \begin{align*}
    & [\beta(x_k)]_i := \\
    & \begin{cases}
      [\Grad f(x_k)]_i + \lambda & \text{if $[x_k]_i = 0$ and $[\Grad f(x_k)]_i +\lambda < 0$,} \\
      [\Grad f(x_k)]_i - \lambda & \text{if $[x_k]_i = 0$ and $[\Grad f(x_k)]_i - \lambda > 0$,} \\
      0 & \text{otherwise,}
    \end{cases}
    \\
    & [\phi(x_k)]_i := \\
    & \begin{cases}
      0 & \text{if $[x_k]_i = 0$,} \\
      \min\{[\Grad f(x_k)]_i + \lambda, \max\{[x_k]_i,[\Grad f(x_k)]_i - \lambda\} \} & \text{if $[x_k]_i > 0$ and $[\Grad f(x_k)]_i +\lambda > 0$,} \\
      \max\{[\Grad f(x_k)]_i - \lambda, \min\{[x_k]_i,[\Grad f(x_k)]_i + \lambda\} \} & \text{if $[x_k]_i < 0$ and $[\Grad f(x_k)]_i -\lambda < 0$,} \\
      [\Grad f(x_k) + \lambda \cdot \text{sgn}(x_k)]_i & \text{otherwise.}
    \end{cases}
  \end{align*}
  For any component $i$, we proceed by considering various cases and subcases. \\
  \textbf{Case 1:}
  %===========
  %\begin{enumerate}
   %\item[Case 1:]
   Suppose that
    \begin{equation}\label{eq.case1}
      [x_k - \Grad f(x_k)]_i > \lambda \text{, meaning that } [x_k]_i > [\Grad f(x_k)]_i + \lambda.
    \end{equation}
    \textbf{Subcase 1a:}
    %=============
    %\begin{enumerate}
     %\item[Subcase 1.a:]
     Suppose that $[x_k]_i > 0$ and $[\Grad f(x_k)]_i + \lambda > 0$, so $[\Grad f(x_k)]_i > -\lambda$.  Then, $[\beta(x_k)]_i = 0$ and
      \begin{equation}\label{eq.t1}
        [\phi(x_k)]_i = \min\{[\Grad f(x_k)]_i + \lambda, \max\{[x_k]_i,[\Grad f(x_k)]_i - \lambda\}\}.
      \end{equation}
      By \eqref{eq.case1}, it follows that $[x_k]_i > [\Grad f(x_k)]_i + \lambda > [\Grad f(x_k)]_i - \lambda$, which along with $[x_k]_i > 0$ means that the $\max$ in \eqref{eq.t1} evaluates as $[x_k]_i$.  Then, again with \eqref{eq.case1}, the $\min$ in \eqref{eq.t1} yields
      \begin{equation}\label{eq.result1}
        [\phi(x_k)]_i = [\Grad f(x_k)]_i + \lambda = -[s_k]_i.
      \end{equation}
      \textbf{Subcase 1b:}
      %==============
      %\item[Subcase 1.b:]
       Suppose that $[x_k]_i > 0$ and $[\Grad f(x_k)]_i + \lambda \leq 0$, so $[\Grad f(x_k)]_i \leq -\lambda$.  Then, $[\beta(x_k)]_i = 0$ and
      \begin{equation}\label{eq.result2}
        [\phi(x_k)]_i = [\Grad f(x_k)]_i + \lambda = -[s_k]_i.
      \end{equation}
      \textbf{Subcase 1c:}
      %==============
      %\item[Subcase 1.c:]
      Suppose that $[x_k]_i = 0$ and $[\Grad f(x_k)]_i + \lambda \leq 0$, so $[\Grad f(x_k)]_i \leq -\lambda$.  Then, $[\phi(x_k)]_i = 0$ and
      \begin{equation}\label{eq.result3}
        [\beta(x_k)]_i = [\Grad f(x_k)]_i + \lambda = -[s_k]_i.
      \end{equation}
      \textbf{Subcase 1d:}
      %==============
      %\item[Subcase 1.d:]
      Suppose that $[x_k]_i < 0$ and $[\Grad f(x_k)]_i + \lambda < 0$, so $[\Grad f(x_k)]_i < -\lambda$.  Then, $[\beta(x_k)]_i = 0$ and
      \begin{equation}\label{eq.t2}
        [\phi(x_k)]_i = \max\{[\Grad f(x_k)]_i - \lambda, \min\{[x_k]_i,[\Grad f(x_k)]_i + \lambda\}\}.
      \end{equation}
      By \eqref{eq.case1}, it follows that $[x_k]_i > [\Grad f(x_k)]_i + \lambda$, which along with $[\Grad f(x_k)]_i + \lambda < 0$ means that the $\min$ in \eqref{eq.t2} evaluates as $[\Grad f(x_k)]_i + \lambda$.  Then, since $[\Grad f(x_k)]_i + \lambda > [\Grad f(x_k)]_i - \lambda$, the $\max$ in \eqref{eq.t2} yields
      \begin{equation}\label{eq.result4}
        [\phi(x_k)]_i = [\Grad f(x_k)]_i + \lambda = -[s_k]_i.
      \end{equation}
    %\end{enumerate}
    Since Subcases 1a--1d exhaust all possibilities under \eqref{eq.case1}, we conclude from the results in \eqref{eq.result1}, \eqref{eq.result2}, \eqref{eq.result3}, and \eqref{eq.result4} that for Case 1 we have $[s_k]_i = -[\beta(x_k) + \phi(x_k)]_i$.

   \noindent\textbf{Case 2:}
   %==========
    %\item[Case 2:]
    Suppose that
    \begin{equation}\label{eq.case2}
      [x_k - \Grad f(x_k)]_i < -\lambda \text{, meaning that } [x_k]_i < [\Grad f(x_k)]_i - \lambda.
    \end{equation}
    We claim that the analysis for this case is symmetric to that in Case 1 above, from which we may conclude that for this case we again have $[s_k]_i = -[\beta(x_k) + \phi(x_k)]_i$.

    \noindent\textbf{Case 3:}
    %==========
    %\item[Case 3:]
    Suppose that
    \begin{equation}\label{eq.case3}
      [x_k - \Grad f(x_k)]_i \in [-\lambda,\lambda] \text{, meaning that } [x_k]_i \in [\Grad f(x_k)]_i + [-\lambda,\lambda].
    \end{equation}
    %\begin{enumerate}
    \textbf{Subcase 3a:}
    %==============
     %\item[Subcase 3.a:]
      Suppose that $[x_k]_i > 0$.  Then, $[\beta(x_k)]_i = 0$ and, since $[x_k]_i > 0$ and \eqref{eq.case3} imply $[\Grad f(x_k)_i] > -\lambda$,
      \begin{equation}\label{eq.t3}
        [\phi(x_k)]_i = \min\{[\Grad f(x_k)]_i + \lambda, \max\{[x_k]_i,[\Grad f(x_k)]_i - \lambda\}\}.
      \end{equation}
      Since \eqref{eq.case3} also implies $[x_k]_i > [\Grad f(x_k)]_i - \lambda$, it follows along with $[x_k]_i > 0$ that the $\max$ in \eqref{eq.t3} evaluates as $[x_k]_i$.  Then, since \eqref{eq.case3} implies $[x_k]_i < [\Grad f(x_k)]_i + \lambda$, the $\min$ in \eqref{eq.t3} yields
      \begin{equation}\label{eq.result5}
        [\phi(x_k)]_i = [x_k]_i = -[s_k]_i.
      \end{equation}
      \textbf{Subcase 3b:}
      %===============
      %\item[Subcase 3.b:]
      Suppose that $[x_k]_i = 0$.  Then, $[\phi(x_k)]_i = 0$ and, along under \eqref{eq.case3},
      \begin{equation}\label{eq.result6}
        [\beta(x_k)]_i = -[s_k]_i = 0.
      \end{equation}
      \textbf{Subcase 3c:}
      %==============
      %\item[Subcase 3.c:]
      Suppose that $[x_k]_i < 0$.  We claim that the analysis for this case is symmetric to that in Subcase 3.a, from which we may conclude that for this subcase we again have
      \begin{equation}\label{eq.result7}
        [\phi(x_k)]_i = [x_k]_i = -[s_k]_i.
      \end{equation}
    %\end{enumerate}
    Since Subcases 1.a--1.d exhaust all possibilities under \eqref{eq.case3}, we conclude from the results in \eqref{eq.result5}, \eqref{eq.result6}, and \eqref{eq.result7} that for Case 3 we have $[s_k]_i = -[\beta(x_k) + \phi(x_k)]_i$.
 % \end{enumerate}
  The result follows as we have proved the desired result under all cases.
\end{proof}
\end{document}